\documentclass[reqno,a4paper]{amsart}

\usepackage[italian,english]{babel}
\usepackage{eucal,amsfonts,amssymb,amsmath,amsthm,epsfig,mathrsfs}
\usepackage{color}
\usepackage[dvipsnames]{xcolor}
\usepackage{amscd,amsxtra}
\usepackage{enumerate}
\usepackage{latexsym}
\usepackage{bm}
\usepackage{tcolorbox}

\newtheorem{thm}{Theorem}[section]

\newtheorem{pro}[thm]{Proposition}
\newtheorem{cor}[thm]{Corollary}
\theoremstyle{definition}
\newtheorem{defn}[thm]{Definition}

\newtheorem{hyp}[thm]{Hypothesis}

\theoremstyle{remark}
\newtheorem{remark}[thm]{Remark}

\numberwithin{equation}{section}

\renewcommand{\hat}[1]{\widehat{#1}}
\renewcommand{\tilde}[1]{\widetilde{#1}}

\newcommand{\pa}[1]{{\left(#1\right)}}

\newcommand{\abs}[1]{{\left|#1\right|}}
\newcommand{\norm}[1]{{\left\|#1\right\|}}
\newcommand{\ra}{\rightarrow}

\newcommand{\N}{\mathbb{N}}

\newcommand{\R}{\mathbb{R}}

\newcommand{\Rd}{\mathbb R^d}
\newcommand{\Rm}{\mathbb R^m}
\newcommand{\uu}{{\bm{u}}}
\newcommand{\ww}{{\bm{w}}}
\newcommand{\f}{{\bm f}}
\newcommand{\g}{{\bm g}}

\newcommand{\eqsys}[1]{{\left\{\begin{array}{ll}#1\end{array}\right.}}

\newcommand{\eps}{\varepsilon}

\newcommand{\A}{{\operatorname{\mathcal{A}}}}
\newcommand{\OU}{{\operatorname{\mathcal{L}}}}

\DeclareMathOperator{\diver}{\operatorname{div}}

\allowdisplaybreaks

\begin{document}

\frenchspacing

\title[Regularity results for elliptic and parabolic systems of PDEs]{Regularity results for elliptic and parabolic systems of partial differential equations}

\author[L. Angiuli]{{Luciana Angiuli}}
\author[S. Ferrari]{{Simone Ferrari}}
\author[L. Lorenzi]{{Luca Lorenzi}$^*$}

\thanks{$^*$Corresponding author.}

\address[L. Angiuli and S. Ferrari]{Dipartimento di Matematica e Fisica ``Ennio de Giorgi'', Universit\`a del Salento, Via per Arnesano sn, 73100 Lecce, Italy.}
\email{luciana.angiuli@unisalento.it}
\email{simone.ferrari@unisalento.it}

%\address[S. Ferrari]{Dipartimento di Matematica e Fisica ``Ennio de Giorgi'', Universit\`a del Salento, Via per Arnesano sn, 73100 Lecce, Italy.}
%\email{simone.ferrari@unisalento.it}

\address[L. Lorenzi]{Plesso di Matematica, Dipartimento di Scienze Matematiche, Fisiche e Informatiche, Universit\`a degli Studi di Parma, Parco Area delle Scienze 53/A, 43124, Italy.}
\email{luca.lorenzi@unipr.it}

%\dedicatory{This paper is dedicated to Professor ABCD}

\begin{abstract}
We study Cauchy problems associated to elliptic operators acting on vector-valued functions and coupled up to the first-order. We prove pointwise estimates for the spatial derivatives of the semigroup associated to these problems in the space of bounded and continuous functions over $\R^d$. Consequently, we deduce relevant regularity results  both in  H\"older and Zygmund spaces and in Sobolev and Besov spaces.
\end{abstract}


\maketitle
\section{Introduction}
Schauder estimates and more generally regularity results for elliptic and parabolic problems is an issue which has a great relevance in the analysis of partial differential equations both for its own interest and for the applications: in many cases mathematical models lead to nonlinear (systems of) partial differential equations, where the nonlinearity depends on the derivatives of the maximum order of the unknown function(s).

The first results in this direction have been proved in the sixties of the last century for elliptic and parabolic equations (and systems) with smooth and bounded coefficients (see e.g., \cite{eidelman,friedman,krylov-1,krylov-2,LSU} for a detailed exposition of such optimal regularity results). Later on, in the eighties, attention has been paid to the case of non autonomous equations with nonsmooth in time coefficients, with the pioneering papers by Kruzhkov, Castro and Lopes \cite{kruzhkov-3,kruzhkov-1, kruzhkov-2}. More recently, the results in the previous papers have been extended in \cite{bramanti,lorenzi-0}. One common feature in all the previous papers is the boundedness of the coefficients of the operators considered.

Starting from the nineties, much attention has been devoted to elliptic and parabolic equations with unbounded coefficients, since these equations are strictly connected to stochastic partial differential equations and appear naturally in models of finance (such as in the famous Black and Scholes model) and also in physics (see e.g., \cite{Freid} and the references here below). In the case of a single elliptic or parabolic equation, Schauder estimates have been obtained with different assumptions and different techniques  in the case of nondegenerate equations \cite{bertoldi-lorenzi} and in the case of degenerate equations \cite{daprato,farkas-lorenzi,lorenzi-1,lorenzi-2,saintier}, using analytic tools, and in \cite{priola-1}, using probabilistic tools, which enlighten the strict connection between elliptic and parabolic equations with unbounded coefficients and stochastic analysis.
Even if the previous papers are concerned with autonomous equations, in most cases the results can be generalized to nonautonomous equations and provide Schauder optimal regularity results in space. 

Already in the autonomous case, one of the main differences with respect to the classical case where the coefficients are bounded (and smooth enough), is the fact that in general one cannot expect optimal regularity results with respect to time. This situation is clearly highlighted by the Onstein--Uhlenbeck operator $\bm{\OU}$, whose associated semigroup admits an explicit integral formula through a Gaussian kernel: the solution to the parabolic equation $D_t\uu=\bm{\OU}\uu$, which satisfies the initial condition $\uu(0,\cdot)=\f$ does not gain regularity with respect to time even if $\f$ is smooth.
Based on this remark, in the nonautonomous case, the results in \cite{bramanti,kruzhkov-3,kruzhkov-1,kruzhkov-2,lorenzi-0} has been extended in \cite{lorenzi-3,lorenzi-4} to some class of elliptic operators whose coefficients are H\"older continuous with respect to the spatial variables and just bounded and measurable with respect to time. As a consequence of this fact, the interest confines in investigating regularity in spatial variables. In this respect, the analysis of scalar elliptic and parabolic equations with unbounded coefficients is at a good stage. On the other hand, the analysis of systems of elliptic and parabolic equations with unbounded coefficients is still at a preliminary stage, notwithstanding the relevance of models which lead to these kind of systems (e.g., Nash equilibria to stochastic differential games \cite{AALT}, the weighted $\overline{\partial}$-problem \cite{haslinger}, the time-dependent Born--Openheimer theory \cite{betz}, Navier--Stokes equations \cite{hansel-rhandi}).
The analysis has been performed in the setting of bounded and continuous functions over $\R^d$ and in $L^p$-spaces (mainly) related the Lebesgue measure. In the first setting, one fundamental tool to study the homogeneous equation $D_t\uu=\bm{\A}\uu$ (and consequently to associate a semigroup of bounded operators with the operator $\bm{\A}$) in the scalar case is a variant of the classical maximum principle, which is used both to provide existence and uniqueness of a classical solution to the previous parabolic equation, which is bounded in every strip $[0,T]\times\R^d$ and subject to the initial condition $\uu(0,\cdot)=\f$, $\f$ being a bounded and continuous function over $\R^d$.
As it is known, in the case of systems of elliptic equations (even with bounded coefficients) no maximum principles can be proved when the elliptic equations are coupled up to the second-order. More precisely, a necessary and sufficient condition for a maximum modulus principle to hold is available in literature (see e.g., \cite{kresin}). To the best of our knowledge, the results so far available in the literature, in the setting of bounded and continuous vector-valued functions, cover the case of weakly coupled systems (where the diffusion may also change equation by equation)
\cite{AALT,delmonte-lorenzi} and systems coupled up to the first-order with the same diffusion part in all the equations \cite{AD,AL}.

Different is the case of the $L^p$-setting, where the well posedness of the equation $D_t\uu=\bm{\A}\uu$ in the $L^p$-spaces of the Lebesgue measure has been proved also in situations where the equations are coupled up to the second-order (see \cite{AngLorMan}). The well posedness of the previous parabolic equations allows to associate a strongly continuous semigroup to the operator $\bm{\A}$.
In particular, under suitable assumptions, in \cite{AngLorMan} a complete description of the domain of the infinitesimal generator is available. It is worth to mention that, even under weaker assumptions, it is possible to prove the well posedness in $L^p(\R^d;\R^m)$ of the equation $D_t\uu=\bm{\A}\uu$ (see \cite{ALP}) 
starting from the well posedness in the space of bounded and continuous functions. Clearly, this strategy is confined to the case when the coupling between the equations occurs at most at the first-order. Moreover, this extension technique does not provide precise information on the domain of the infinitesimal generator of
the associated strongly continuous semigroup. For this reason, it is important to find different arguments to obtain at least a partial characterization of the domain of the infinitesimal generator.

In this paper, we consider a family of second-order uniformly elliptic operators $\bm{\mathcal A}$ (having all the same principal part), defined on smooth enough functions ${\bm w}:\R^d\to \R^m$ by
\begin{equation*}
(\bm{\mathcal A}{\bm w})(x)=\sum_{i,j=1}^dq_{ij}(x)D^2_{ij}{\bm w}(x)+\sum_{j=1}^d B_j(x)D_j{\bm w}(x)+C(x){\bm w}(x),\qquad\;\, x \in\R^d,
%\label{operat-A}
\end{equation*}
where the entries $q_{ij}$ of the matrix-valued function $Q$ are smooth enough functions, possibly unbounded and the infimum over $\R^d$ of the minimum eigenvalue of the matrix $Q(x)$ is positive.
As far as $B_j$ ($j=1, \ldots,d$) and $C$ are concerned, they are $m\times m$ matrices whose elements are smooth enough and (possibly) unbounded real-valued functions, see Hypotheses \ref{base}. Such a class of operators have been already considered in \cite{AALT} where, among other results, it has been proved that, under suitable assumptions, the Cauchy problem
\begin{equation}\label{cp-intro}
\left\{\begin{array}{ll}
D_t\uu(t,x)=(\bm{\mathcal A}\uu)(t,x),\qquad\;\, &(t,x)\in (0,\infty)\times\R^d,\\[1mm]
\uu(0,x)=\f(x),& x \in \Rd,
\end{array}
\right.
\end{equation}
is well posed in the space of $\R^m$-valued bounded and continuous functions over $\R^d$, which means that for every $\f$ in such a space, the above problem admits a unique classical solution, which is bounded in every strip $[0,T]\times\R^d$. As it has been already remarked, the well posedness of the Cauchy problem \eqref{cp-intro} is connected to the possibility of associating a semigroup of linear and bounded operators $\{\bm{T}(t)\}_{t\ge 0}$ in $C_b(\R^d;\R^m)$ with the operator $\bm{\A}$.

Strengthening the assumptions on the coefficients of the operator $\bm{\A}$, we prove 
pointwise estimates for the spatial derivatives of the function $\bm{T}(t)\f$, when $\f$ is bounded and continuous in $\R^d$, or even smoother. The behaviour of such estimates
with respect to the variable $t$, when $t$ approaches zero, is the same as in the classical case of bounded coefficients.

Starting from these estimates, we prove several remarkable consequences. First of all, we prove optimal (in space) Schauder regularity results for the solution to the nonhomogeneous Cauchy problem
\begin{equation}\label{cp-nnohom}
\left\{\begin{array}{ll}
D_t\uu(t,x)=(\bm{\mathcal A}\uu)(t,x)+\bm{g}(t,x),\qquad\;\, &(t,x)\in (0,T)\times \Rd,\\[1mm]
\uu(0,x)=\f(x),& x \in \Rd,
\end{array}
\right.
\end{equation}
and the elliptic problem $\lambda\uu-\bm{\A}\uu=\f$.

Next, we move to the $L^p$-setting (related to the Lebesgue measure) and show that, under an additional condition, the restriction of the semigroup $\{\bm{T}(t)\}_{t\ge 0}$ to the space of $\R^m$-valued continuous and compactly supported functions over $\R^d$ extends to a strongly continuous semigroups on the $L^p(\R^d;\R^m)$-scale for any $p\in [2,\infty)$. Such semigroups are
consistent. Using the pointwise estimates proved, we show that the semigroup regularizes the functions to which it is applied. This allows us to prove regularity results for the Cauchy problem \eqref{cp-nnohom} and the elliptic problem $\lambda\uu-\bm{\A}\uu=\f$ when $\f$ and
$\bm{g}$ belong to suitable subspaces of $L^p(\R^d;\R^m)$. In particular, from such results we obtain a partial characterization of the domain of the infinitesimal generator of the semigroup
$\{\bm{T}(t)\}_{t\ge 0}$ in $L^p(\R^d;\R^m)$, showing that it is continuously embedded in
the Besov space $B^{\sigma}_{p,p}(\R^d;\R^m)$
for every $\sigma\in [0,2)$.

The paper is organized as follows. In Section \ref{sect-2}, we introduce the basic assumptions and collect some preliminary results, taken from \cite{AALT} and \cite{AngLorCom}, as well as some basic facts about real interpolation theory. Section \ref{sect-3} is devoted to the pointwise estimates, which are obtained adapting the Bernstein method.
Sections \ref{sect-4} and \ref{sect-5} collect consequences of the pointwise estimates in the setting of bounded and continuous functions (Section \ref{sect-4}) and in the $L^p(\R^d;\R^m)$ context (Section \ref{sect-5}).
Finally, in Section \ref{sect-6} we provide concrete examples of systems of elliptic equations with unbounded coefficients to which the results of the paper apply and, in the Appendix, we collect a regularity result of solutions to parabolic systems of partial differential equations and interior Schauder estimates that are used in the paper.

\smallskip

{\bf General notation.} 
For any $\alpha\ge 0$, we denote by $[\alpha]$ its integer part and by $\{\alpha\}$ its fractional part.
We denote  by $\langle\cdot, \cdot\rangle_{\R^k}$ and $\abs{\cdot}_{\R^k}$, respectively, the Euclidean inner product and the Euclidean norm in $\mathbb R^k$. If no confusion arises, we will drop the subscript. For a given matrix-valued function $M:\Rd\ra \R^{m^2}$, we denote by $\lambda_{M}$ and $\Lambda_{M}$ real valued functions such that
$\lambda_M(x)|\xi|^2\le \langle M(x)\xi, \xi\rangle \le \Lambda_M(x)|\xi|^2$
for any $x \in \Rd$ and $\xi \in \R^m$. We denote by $\chi_A$ the characteristic function of the set $A \subseteq \R$. By $B(0,r)$ and $\partial B(0,r)$ we denote the ball of center the origin and radius $r$ and its boundary.

Vector-valued functions will be displayed in bold style. The components of $\bm{u}:\Omega\subseteq\R^d\to\R^m$ are denoted by
$u_1,\ldots,u_m$. $D^1\uu$ denotes the Jacobian matrix of $\uu$, whereas $D^2\uu$ (resp. $D^3\uu$) denotes the vector whose entries are the matrices (resp. the tensors) $D^2u_1,\ldots,D^2u_m$ (resp. $D^3u_1,\ldots,D^3u_m$).
To shorten the notation, we write
$|D^1\uu|^2=\sum_{j=1}^m|\nabla u_j|^2$,
$|D^2\uu|^2=\sum_{j=1}^m|D^2u_j|^2$ and
$|D^3\uu|^2=\sum_{j=1}^m|D^3u_j|^2$. 
Given a multi-index $\gamma$, we use the notation $\partial^{\gamma}$ to denote the partial derivative $\partial^{\gamma}/\partial x^{\gamma}$.

{\bf Function spaces.} 
If $X$ and $Y$ are Banach spaces we denote by $\mathcal{L}(X;Y)$ the Banach space of the bounded and linear operators from $X$ to $Y$ with its natural norm. If $X=Y$ we simply write $\mathcal{L}(X)$.

$C^k_b(\Rd;\mathbb R^m)$ is the space  of continuously differentiable functions on $\Rd$, up to the $k$-th order, which are  bounded  together with their derivatives.
We assume that the reader is familiar with the classical H\"older space $\mathcal{C}^{\alpha}_b(\R^d;\R^m)$ ($\alpha\in (0,1)$) and with its seminorm $[\cdot]_{\mathcal{C}^{\alpha}_b(\R^d;\R^m)}$. For a non integer $\alpha>1$,
$\mathcal{C}^{\alpha}_b(\R^d;\R^m)$ is the subset of 
$C^{[\alpha]}_b(\R^d;\R^m)$ of functions $\f$ such that $D^{\beta}\f\in \mathcal{C}^{\{\alpha\}}_b(\R^d;\R^m)$ for ever multi-index $\beta$ with length $[\alpha]$. When $m=1$, we simply write
$C^k_b(\R^d)$ and $\mathcal{C}^{\alpha}_b(\R^d)$.
We use the subscript ``$c$'' when we need to enlighten that we deal with compactly supported functions. Similarly, the subscript ``loc'' stands for locally. For instance, 
$\mathcal{C}^{\alpha}_{\rm loc}(\R^d)$ is the space of functions $f:\R^d\to\R$ whose restriction to every compact set of $\R^d$ is $\alpha$-H\"older continuous. We also use parabolic spaces of continuous functions. By $\mathcal{C}^{0,\alpha}_b(I\times\Omega)$, when $\alpha\in (0,\infty)$, $I$ is an interval and $\Omega$ is an open subset of $\R^d$, we denote the set of continuous functions $\f:I\times\Omega\to\R^m$ such that $\|\f\|_{\mathcal{C}^{0,\alpha}_b(I\times\Omega;\R^m)}:=\sup_{t\in I}
\|\f(t,\cdot)\|_{\mathcal{C}^{\alpha}_b(\Omega)}<\infty$. Similarly, for $\alpha,\beta\in (0,1)$, $\mathcal{C}^{\alpha,\beta}_b(I\times\Omega;\R^m)$ is the subset of
$\mathcal{C}^{0,\beta}_b(I\times\Omega;\R^m)$ consisting of functions $\f$ such that $\|\f\|_{\mathcal{C}^{\alpha,\beta}_b(I\times\Omega;\R^m)}:=
\|\f\|_{\mathcal{C}^{0,\beta}_b(I\times\Omega;\R^m)}+\sup_{x\in\Omega}[\f(\cdot,x)]_{\mathcal{C}^{\alpha}_b(I)}<\infty$.
More generally, for every $\alpha,\beta\in (0,\infty)\setminus\N$, 
$\mathcal{C}^{\alpha,\beta}_b(I\times\Omega;\R^m)$ is the set of all the functions $\f:I\times\Omega\to\R^m$, which admit time derivatives up to the order $[\alpha]$ and spatial derivatives up to the order $[\beta]$ and all the existing derivatives belong to $\mathcal{C}^{\{\alpha\},\{\beta\}}_b(I\times\Omega;\R^m)$. This space is normed in the natural way. In this setting, we also use the subscript ``loc'' instead of ``$b$'', when needed.

For every $p\in [1,\infty)$,  $L^p(\Rd;\mathbb R ^m)$  is the classical vector-valued Lebesgue space endowed with the norm
$\|\f\|^p_{L^p(\Rd;\mathbb R ^m)}:=\int_{\Rd} |\f(x)|^pdx$, 
while, for $k\in\mathbb N$, $W^{k,p}(\Rd, \mathbb R^m)$ is the classical vector-valued Sobolev space (i.e., the space of all functions $\uu\in L^p(\Rd;\mathbb R^m)$ whose components have distributional derivatives up to the order $k$, which belong to $L^p(\R^d;\mathbb R)$), endowed with  the usual norm 
$\norm{\cdot}_{k,p}$. When $m=1$, we simply write $L^p(\Rd)$.

Finally, if $(X(\Rd;\Rm),\norm{\cdot}_{X(\Rd;\Rm)})$ is a Banach space, by $L^p((0,T);X(\Rd;\Rm))$ we denote the set of functions $g:(0,T)\times \Rd \to \Rm$ such that $g(t,\cdot)$ belongs to $X(\Rd;\Rm)$ for any $t \in (0,T)$ and $\|g\|_{L^p((0,T);X(\Rd;\Rm))}^p:=\int_0^T\|g(t,\cdot)\|_{X(\Rd;\Rm)}^pdt$ is finite.

\section{Assumptions and preliminaries}
\label{sect-2}
The following are standing assumptions that guarantee existence and uniqueness of a locally in time bounded classical solution to the Cauchy problem 
\begin{equation}\label{cp}
\left\{\begin{array}{ll}
D_t\uu(t,x)=(\bm{\mathcal A}\uu)(t,x),\qquad\;\, &(t,x)\in (0,\infty)\times \Rd,\\[1mm]
\uu(0,x)=\f(x),& x \in \Rd,
\end{array}
\right.
\end{equation}
when $\f\in C_b(\R^d;\R^m)$ and the system of elliptic operators 
$\bm{\mathcal A}$ is defined on smooth functions ${\bm w}:\Rd\to \R^m$ by
\begin{equation*}
(\bm{\mathcal A}{\bm w})(x)=\sum_{i,j=1}^dq_{ij}(x)D^2_{ij}{\bm w}(x)+\sum_{j=1}^d B_j(x)D_j{\bm w}(x)+C(x){\bm w}(x),\qquad\;\, x \in \Rd.
%\label{operat-A}
\end{equation*}

\begin{hyp}\label{base}
\begin{enumerate}[\rm(i)]
\item 
The coefficients $q_{ij}=q_{ji}$ and $C_{hk}$ (the entries of the matrix $C$) belong to $\mathcal{C}^{1+\alpha}_{\rm loc}(\R^d)$ for some $\alpha\in (0,1)$, any $i,j=1,\ldots,d$ and any $h,k=1,\ldots,m$;

\item 
$\lambda_0:=\inf_{\Rd}\lambda_Q>0$;

\item 
for any $i=1,\ldots,d$ there exist functions $b_i\in \mathcal{C}^{1+\alpha}_{\rm loc}(\R^d)$ and $\hat B_i \in \mathcal{C}^{1+\alpha}_{\rm loc}(\R^d; \R^{m^2})$, with null entries on the main diagonal, such that 
$B_i=b_iI_m+\hat B_i$
and $\max_{i=1,\ldots,d}|\hat{B}_i(x)|\leq \hat{\beta}_0(x)$, for all $x\in\R^d$ and some function $\hat{\beta}_0$ which satisfies the condition 
\begin{eqnarray*} 
H:=\sup_{x\in\Rd}\left((2-\nu)\Lambda_C(x)+d\frac{(\hat{\beta}_0(x))^2}{2\lambda_Q(x)}\right)<\infty
\end{eqnarray*}
for some constant $\nu\in [0,2)$;\label{puntoIII}
\item
there exist ${\mu}\in \R$ and a positive $($Lyapunov$)$ function
$\varphi\in C^2(\R^d)$, blowing up as $|x|$ tends to $\infty$, such that
$\sup_{\R^d}(\A_{\nu}\varphi-{\mu}\varphi)<\infty$, where 
\begin{align}\label{Atilde}
\A_{\nu}={\rm Tr}(QD^2)+\langle {\bm b}, \nabla\rangle+\nu\Lambda_C,
\end{align}
with ${\bm b}=(b_1,\ldots,b_d)$ and $\nu$ is the constant introduced in \eqref{puntoIII};
\item 
there exist a positive constant $\kappa$  such that
\begin{align}\label{maxaut<minaut}
\Lambda_Q(x)\le \kappa(1+|x|^2)\lambda_Q(x),\qquad\;\, x\in\Rd,
\end{align}
and $n_0\in\N$ such that 
\begin{align*}
{\rm Tr}(Q(x))-\frac{\langle Q(x)x,x\rangle}{|x|^2}+\langle \bm{b}(x),x\rangle\le 0,\qquad\;\,x\in\R^d\setminus B(0,n_0).
\end{align*}
\end{enumerate}
\end{hyp}

To state the additional set of assumptions, we introduce the functions 
$\xi_i, \beta_{\ell}, \widehat\beta_i, \gamma_i:\R^d\to\R$, defined as
\begin{align*}
%\begin{array}{ll}    
&\xi_i(x):=\displaystyle\max_{|\alpha|=i}|D^{\alpha}Q(x)|,&  &\beta_\ell(x)
:=\displaystyle\max_{\substack{j=1,\ldots,d\\ |\alpha|=\ell}}|D^{\alpha}b_j(x)|,\\%[3mm]
&\hat\beta_i(x):=\displaystyle\max_{\substack{j=1,\ldots,d\\ |\alpha|=i}}|D^{\alpha}\hat{B}_j(x)|,& &
\gamma_i(x):=\displaystyle\max_{|\alpha|=i}|D^{\alpha}C(x)|
%\end{array}
\end{align*}
for every $x\in\R^d$ and $i=1,2,3$, $\ell=2,3$, provided that the derivatives involved in the previous definitions exist.

\begin{hyp}\label{hyp-derivata-1}
 \begin{enumerate}[\rm (i)]
\item 
There exist constants  $\alpha_1,\mu_1\in [0,2]$ and $A_1, A_2>0$ such that for any $x\in\R^d$ it holds
\begin{align*}
(\xi_1(x))^{\alpha_1}\le A_1\lambda_Q(x),\qquad\;\,
(\widehat\beta_1(x))^{\mu_1}\le A_2\lambda_Q(x).
\end{align*}
\item 
there exist $\tau_1\in [0,2]$ and
constants $M_0>\frac{d}{2}$ and
$N_1,N_2>0$ such that
\begin{align*}
\;\;\;\;\;\;\;\;&\sup_{\R^d}\Big ((2-\nu)\Lambda_C+M_0\widehat\beta_0^2\lambda_Q^{-1}+N_1\gamma_1^{\tau_1}\Big)<\infty;\\[2mm]
&\sup_{\R^d}\Big (\Lambda_{D^1{\bm b}}+\frac{2-\nu}{2}\Lambda_C
+\frac{d}{2}\big (\widehat\beta_0\lambda_Q^{-\frac{1}{2}}+\sqrt{A_1}\xi_1^{1-\frac{\alpha_1}{2}}\big )^2
+N_2\big (
\hat{\beta}_1^{2-{\mu_1}}+\gamma_1^{2-\tau_1}\big )\Big )<\infty.
\end{align*}
Here $\nu$ is the constant introduced in Hypotheses \ref{base}\eqref{puntoIII}.
\end{enumerate}
\end{hyp}

\begin{hyp}\label{hyp-derivata-2}
 \begin{enumerate}[\rm (i)]
\item Assume Hypothesis \ref{hyp-derivata-1}(i) is satisfied and that the coefficients $q_{ij}$, $b_i$, $C_{hk}$ belong to $\mathcal{C}^{2+\alpha}_{\rm loc}(\Rd)$ and $\hat B_i \in \mathcal{C}_{\rm loc}^{2+\alpha}(\R^d; \R^{m^2})$ for any $i,j=1, \ldots,d$ and any $h,k=1, \ldots,m$;
\item 
there exist $\alpha_2\in [0,2]$ and a positive constant $A_3$ such that for any $x\in\R^d$ it holds that
\begin{align*}
(\xi_2(x))^{\alpha_2}\le A_3\lambda_Q(x).
\end{align*}
\item 
there exist constants $\rho_2$, $\mu_2$, $\tau_j$ ($j=1,2$) in $[0,2]$,
constants $M_0>\frac{d}{2}$, $M_1>\frac{d}{3}$ and
$N_1,N_2,N_3>0$ such that
\begin{align*}
\;\;\;\;\;\;\;\;&\sup_{\R^d}\Big ((2-\nu)\Lambda_C+M_0\frac{\hat\beta_0^2}{\lambda_Q}+N_1(\gamma_1^{\tau_1}+\gamma_2^{\tau_2})\Big)<\infty;\\[2mm]
&\sup_{\R^d}\Big (\Lambda_{D^1{\bm b}}+\frac{2-\nu}{2}\Lambda_C
+M_1\big (\widehat\beta_0\lambda_Q^{-\frac{1}{2}}+\sqrt{A_1}\xi_1^{1-\frac{\alpha_1}{2}}\big )^2
+N_2\big (
\beta_2^{\rho_2}+\hat{\beta}_1^{2-{\mu_1}}+\hat{\beta}_2^{\mu_2}+\gamma_1^{2-\tau_1}\big )\Big )<\infty;\\
&\sup_{\R^d}\Big (\Lambda_{D^1{\bm b}}+\frac{2-\nu}{4}\Lambda_C+\frac{d}{4}\big (\widehat\beta_0\lambda_Q^{-
\frac{1}{2}}+2\sqrt{A_1}\xi_1^{1-\frac{\alpha_1}{2}}\big )^2\\
&\qquad\;\,+N_3\big (\xi_2^{2-\alpha_2}+\beta_2^{2-\rho_2}
+\hat{\beta}_1^{2-{\mu_1}}+\hat{\beta}_2^{2-{{\mu_2}}}+\gamma_1^{\tau_1}
+\gamma_2^{2-{\tau_2}}
\big )\Big )<\infty.
\end{align*}
Here $\nu$ is the constant introduced in Hypotheses \ref{base}\eqref{puntoIII}.
\end{enumerate}
\end{hyp}

\begin{hyp}\label{hyp-derivata-3}
 \begin{enumerate}[\rm (i)]
 \item Assume Hyphotheses \ref{hyp-derivata-1}(i) and \ref{hyp-derivata-2}(ii) are satisfied and that the coefficients $q_{ij}$, $b_i$, $C_{hk}$ belong to $\mathcal{C}_{\rm loc}^{3+\alpha}(\Rd)$ and $\hat B_i \in \mathcal{C}^{3+\alpha}_{\rm loc}(\R^d; \R^{m^2})$ for any $i,j=1, \ldots,d$ and any $h,k=1, \ldots,m$;
\item 
there exist constants $\alpha_3$, $\rho_2$, $\rho_3$, $\mu_2$, $\mu_3$, $\tau_j$ ($j=1,2,3$) in $[0,2]$,
constants $M_0>\frac{d}{2}$, $M_1>\frac{d}{3}$, $M_2>\frac{d}{4}$ and
$N_1,\ldots,N_4>0$
such that
\begin{align*}
&\sup_{\R^d}\Big [(2-\nu)\Lambda_C+M_0\frac{\hat\beta_0^2}{\lambda_Q}+N_1(\gamma_1^{\tau_1}+\gamma_2^{\tau_2}+\gamma_3^{\tau_3})
\Big ]<\infty;\\[2mm]
&\sup_{\R^d}\Big [\Lambda_{D^1{\bm b}}+\frac{2-\nu}{2}\Lambda_C
+M_1\Big (\widehat{\beta}_0\lambda_Q^{-\frac{1}{2}}+\sqrt{A_1}\xi_1^{1-\frac{\alpha_2}{2}}\Big )^2\\
&\qquad\;+N_2\big (
\beta_2^{\rho_2}+\beta_3^{\rho_3}+\hat{\beta}_1^{2-{\mu_1}}+\hat{\beta}_2^{\mu_2}+\hat{\beta}^{{\mu_3}}_3+\gamma_1^{2-\tau_1}+\gamma_2^{{\tau_2}}\big )\bigg ]<\infty;\\
&\sup_{\R^d}\bigg [\Lambda_{D^1{\bm b}}+\frac{2-\nu}{4}\Lambda_C+M_2\Big (\widehat{\beta}_0\lambda_Q^{-\frac{1}{2}}+2\sqrt{A_1}\xi_1^{1-\frac{\alpha_1}{2}}\Big )^2\\
&\qquad\;\,+N_3\big (\xi_2^{2-\alpha_2}\!+\!\xi_3^{\alpha_3}\!+\!\beta_2^{2-\rho_2}\!+\!\hat{\beta}_1^{2-{\mu_1}}\!+\!\hat{\beta}_2^{2-{{\mu_2}}}\!+\!\gamma_1^{\tau_1}
\!+\!\gamma_2^{2-{\tau_2}}
\big )\bigg ]<\infty;\\[2mm]
&\sup_{\R^d}\bigg [\Lambda_{D^1{\bm b}}+\frac{2-\nu}{6}\Lambda_C+\frac{d}{6}\Big (\widehat{\beta}_0\lambda_Q^{-\frac{1}{2}}+3\sqrt{A_1}\xi_1^{1-\frac{\alpha_1}{2}}\Big )^2\\
&\qquad\;\,+N_4\big (\xi_2^{2-\alpha_2}\!+\!\xi_3^{2-\alpha_3}\!+\!\beta_2^{\rho_2}+\beta_3^{2-\rho_3}+\hat{\beta}_1^{2-\mu_1}+\hat{\beta}_2^{\mu_2}+\hat{\beta}^{2-\mu_3}_3+\gamma_1^{2-{\tau_1}}+\gamma_2^{2-\tau_2}+\gamma_3^{2-\tau_3}\big )\Big ]<\infty.
\end{align*}
Here $\nu$ is the constant introduced in Hypotheses \ref{base}\eqref{puntoIII}.
\end{enumerate}
\end{hyp}

Before stating the main result of this section, we recall some known facts concerning the solvability of the Cauchy problems associated to the realization of the elliptic operators $\A_\nu$, introduced in \eqref{Atilde}, and $\bm \A$ in the space of bounded and continuous functions.
We refer the reader to \cite[Theorem 2.3]{AngLorCom} for the proofs.

\begin{pro}
\label{prop-2.5}
Assume that the coefficients $q_{ij}=q_{ji}, b_i$ belong to $\mathcal{C}^{\alpha}_{\rm loc}(\Rd)$ for any $i,j=1, \ldots,d$ and some $\alpha \in (0,1)$ and that Hypotheses $\ref{base} (ii)$ and  $(iv)$ are satisfied. For any $f\in C_b(\Rd)$, the Cauchy problem
\begin{equation}\label{pb-scalare}
\left\{
\begin{array}{ll}
D_t w(t,x)= (\A_{\nu} w)(t,x), &(t,x) \in (0,\infty)\times\Rd,\\[1mm]
w(0,x)=f(x), &x \in \Rd
\end{array}
\right.
\end{equation}
admits a unique bounded classical solution $w_f$, i.e. there exists a unique $w_f \in C^{1,2}((0,\infty)\times \Rd)\cap C([0,\infty)\times \Rd)$ which is bounded in every strip $[0,T]\times\R^d$ and solves \eqref{pb-scalare}. Setting, for any $f\in C_b(\R^d)$, $S_{
\nu}(t)f=w_f(t,\cdot)$ for every $t\in [0,\infty)$, it holds that $\{S_{\nu}(t)\}_{t\geq 0}$ is a positive semigroup in $C_b(\Rd)$ such that $\|S_{\nu}(t)\|_{\mathcal{L}(C_b(\R^d))}\le e^{\Theta_Ct}$ for every $t\in (0,\infty)$, where $\Theta_C$ denotes the supremum over $\R^d$ of the function $\Lambda_C$. Moreover, for every $t>0$ and $x \in \Rd$ there exists a unique finite measure $p(t,x,\cdot)$ such that $p(t,x,\R^d)\le e^{\Theta_Ct}$ for every $t\in (0,\infty)$ and $x\in\R^d$, and
\begin{align}\label{int_rep} 
(S_{\nu}(t)f)(x)= \int_{\Rd}f(y)p(t,x,dy)
\end{align}
for any $t\in (0,\infty)$, $x \in \Rd$ and $f \in C_b(\Rd)$. Finally, if $
f\in C_b(\R^d)$ is nonnegative then $S_{\nu,n}(t)f\le S_{
\nu}(t)f$ for every $t\in (0,\infty)$, where
$S_{\nu,n}(\cdot)f$ is the classical solution to the Cauchy--Dirichlet problem
\begin{eqnarray*}
\left\{
\begin{array}{ll}
D_tu(t,x)=\A_{\nu}u(t,x), & (t,x)\in (0,\infty)\times B(0,n),\\[1mm]
u(t,x)=0, & (t,x)\in (0,\infty)\times\partial B(0,n),\\[1mm]
u(0,x)=f(x), & x\in B(0,n).
\end{array}
\right.
\end{eqnarray*}
%Here $B(0,n)$ and $\partial B(0,n)$ denote the ball of center the origin and radius $n$ and its boundary.
\end{pro}

Concerning the vector-valued case the following results hold true.

\begin{thm}\label{exi_thm}
Assume Hypotheses $\ref{base}(i)$-$(v)$. For every $\f\in C_b(\Rd;\Rm)$, problem \eqref{cp} admits a unique classical  solution $\uu$ which is bounded in the strip $[0,T]\times\Rd$, for all $T>0$. Moreover, %$\uu$ belongs to $\mathcal{C}^{1+\alpha/2,2+\alpha}_{\rm loc}((0,\infty)\times \Rd;\Rm)$ and
\begin{equation}
|\uu(t,\cdot)|^2\le e^{Ht}S_{\nu}(t)|\f|^2,\qquad\;\,t\in (0,\infty),
\label{cossali}
\end{equation}
where $H$ is the constant introduced in Hypotheses $\ref{base}(iii)$. Finally it holds 
\begin{equation}\label{norm_infty}
\|\uu(t, \cdot)\|_{C_b(\Rd;\Rm)}\le e^{H_{\nu}t}\|\f\|_{C_b(\Rd;\Rm)},\qquad\;\, t\in (0,\infty).
\end{equation}
where $H_{\nu}=H+\Theta_C\nu$, $\Theta_C$ denotes the supremum over $\R^d$ of the function $\Lambda_C$.
\end{thm}

\begin{proof}
The same result has been proved in \cite[Theorem 2.9]{AALT} with a different condition in Hypotheses \ref{base}(iii).
The proof in \cite{AALT} relies on the approximation of the Cauchy problem \eqref{cp} with Cauchy--Dirichlet problem in the ball $B(0,n)$. For every $n\in\N$ sufficiently large and $\f\in C_c(\R^d;\R^m)$ this problem admits a unique solution $\uu_n\in C^{1,2}((0,\infty)\times B(0,n))\cap C([0,\infty)\times\overline{B(0,n)})$ which is bounded in 
$(0,T)\times B(0,n)$ for every $T>0$ by a constant, independent of $n$. To prove that the constant is independent of $n$, it suffices to observe that the function $v=|\uu_n|^2$ solves the Cauchy problem 
\begin{eqnarray*}
\left\{
\begin{array}{ll}
D_tv(t,x)=(\A_{\nu}v)(t,x)+g(t,x), & (t,x)\in (0,\infty)\times  B(0,n),\\[1mm]
v(t,x)=0, & (t,x)\in (0,\infty)\times\partial B(0,n),\\[1mm]
v(0,x)=|\f(x)|^2, &x\in B(0,n)
\end{array}
\right.
\end{eqnarray*}
where
\begin{eqnarray*}
g=-2\sum_{h=1}^d\langle Q\nabla u_{n,h},\nabla u_{n,h}\rangle+2\sum_{i=1}^d\langle \widehat{B}_iD_i\uu_n,\uu_n\rangle+2\langle C\uu_n,\uu_n\rangle-\nu\Lambda_C|\uu_n|^2.
\end{eqnarray*}
Note that
\begin{align*}
g\le &
-2\lambda_Q|D^1\uu_n|^2+2\widehat{\beta}_0\sqrt{d}|\uu_n||D^1\uu_n|+(2-\nu)\Lambda_C|\uu_n|^2\\
\le & 
\big ((2-\nu)\Lambda_C+d\varepsilon^{-1}\widehat{\beta}_0^2\big )|\uu_n|^2-(2\lambda_Q-\varepsilon)|D^1\uu_n|^2
\end{align*}
for every $\varepsilon>0$. Choosing $\varepsilon=2\lambda_Q$, we conclude that $g\le Hv$ in $(0,\infty)\times B(0,n)$ and the classical maximum principle implies 
\begin{equation*}
|\uu_n(t,\cdot)|^2\le e^{Ht}S_{\nu}(t)|\f|^2,\qquad\;\,t\in (0,\infty),
\end{equation*}
From here on out, the proof follows the same lines as that of \cite[Theorem 2.9]{AALT}, which is based on the previous estimate, compactness and approximation arguments. 
\end{proof}

As a consequence of Theorem \ref{exi_thm}, we can associate a semigroup $\{{\bm T}(t)\}_{t\ge 0}$ to $\bm \A$ in $C_b(\Rd;\Rm)$ by setting ${\bm T}(t)\f:=\uu(t,\cdot)$ for every $t\ge 0$. The uniqueness of the solution to problem \eqref{cp} guarantees the semigroup property of the family $\{{\bm T}(t)\}_{t \ge 0}$ that, thanks to estimate \eqref{norm_infty}, turns out to be a family of bounded linear operators in $C_b(\Rd;\Rm)$ satisfying the estimate
\begin{equation*}
\|\bm T(t) \|_{\mathcal{L}(C_b(\Rd;\Rm))}\le e^{H_{\nu}t},\qquad\;\, t\in (0,\infty).
\end{equation*}

\begin{remark}
\label{rmk-1}
{\rm 
As \cite[Proposition 3.1]{AALT} shows, if $(\f_n)_{n\in\N}\subset C_b(\R^d;\R^m)$ is a bounded sequence converging pointwise in $\R^d$ to $\f\in C_b(\R^d;\R^m)$, then $(\bm{T}(\cdot)\f_n)_{n\in\N}$ converges to $\bm{T}(\cdot)\f$ locally uniformly in $(0,\infty)\times\R^d$.}

{\rm Suppose that $(\f_n)\subset \mathcal{C}^{3+\alpha}_{\rm loc}(\R^d;\R^m)$ is a bounded sequence with respect to the sup-norm, which converges to $\f$ pointwise in $\R^d$, and the coefficients of the operator $\bm{\A}$ belong to $\mathcal{C}^{1+\alpha}_{\rm loc}(\R^d)$. Under these assumptions, the function $\bm{T}(\cdot)\f$ admits third-order spatial derivatives which belong to $\mathcal{C}^{\alpha/2,2+\alpha}_{\rm loc}((0,\infty)\times\R^d;\R^m )$ and $D^3\bm{T}(\cdot)\f_n$ converges to $D^3\bm{T}(\cdot)\f$ in $\mathcal{C}^{\alpha/2,\alpha}((\tau,T)\times\Omega;\R^m)$ for every $0<\tau<T$ and every
bounded open set $\Omega\subset\R^d$. For this purpose, it suffices to observe that, by Theorem \ref{teo-reg-int}, the function $\bm{T}(\cdot)\f_n$ belongs to $\mathcal{C}^{1+\alpha/2,2+k+\alpha}_{\rm loc}([0,\infty)\times\R^d;\R^m)$ for every $n\in\N$. Moreover, for every $h=1,\ldots,d$ the function $\bm{w}=\partial_h\bm{T}(\cdot)\f_n$ is a solution of the equation
\begin{eqnarray*}
D_t\bm{v}(t,x)-\bm{\mathcal A}{\bm v}(t,x)=\sum_{i,j=1}^d\partial_hq_{ij}(x)D^2_{ij}{\bm v}(t,x)+\sum_{j=1}^d \partial_hB_j(x)D_j{\bm v}(t,x)+\partial_hC(x){\bm v}(t,x)
\end{eqnarray*}
in $(0,\infty)\times \R^d$. By Theorem \ref{thm-A2} we deduce that the third-order derivatives of the functions $\bm{T}(\cdot)\f_n$ constitutes a Cauchy sequence in $\mathcal{C}^{\alpha/2,\alpha}((\tau,T)\times\Omega;\R^m)$ for every $0<\tau<T$ and every bounded open set $\Omega\subset\R^d$.
This implies that $\bm{T}(\cdot)\f$ admits third-order spatial derivatives, which are continuous functions in $(0,\infty)\times\R^d$, and $\bm{T}(t)\f_n$ converges to 
$\bm{T}(t)\f$ in $C^3(\Omega;\R^m)$ as $n$ approaches $\infty$ for every $t\in [0,\infty)$.}
\end{remark}

\subsection{Interpolation results}
In order to obtain regularity results for linear elliptic and parabolic systems associated to the operator $\bm\A$ both  in the set of bounded and continuous functions and in $L^p$-spaces, it is useful to recall some classical definitions of functional spaces and some well-known interpolation results. We start by recalling the definition of Zygmund spaces.

\begin{defn}\label{defn_Hold_Zyg}
The space $\mathcal{C}^1_b(\R^{d};\Rm)$ is the subspace of $C_b(\R^{d};\Rm)$ consisting of the functions $\f$ such that the seminorm
\begin{align*}
[\f]_{\mathcal{C}^1_b(\R^{d};\Rm)}=\sup_{\substack {x,h\in\R^{d}\\ h\neq 0}}\frac{|\f(x+2h)-2\f(x+h)+\f(x)|}{|h|}
\end{align*}
is finite.
We endow the space $\mathcal{C}^1_b(\R^{d};\Rm)$ with the norm
\begin{align*}
\|\f\|_{\mathcal{C}^1_b(\R^{d};\Rm)}:=\|\f\|_{C_b(\R^{d};\Rm)}+[\f]_{{\mathcal{C}^1_b(\R^{d};\Rm)}}.
\end{align*}
For $k\in\N$, $k\geq 2$, we set
$\mathcal{C}^k_b(\R^d;\Rm):=\{\bm{f}\in C_b^{k-1}(\R^d;\Rm)\,|\, D^{k-1}\bm{f}\in \mathcal{C}^1_b(\R^d;\R^{md^{k-1}})\}$,
endowed with the norm
\begin{align*}
\|\bm{f}\|_{\mathcal{C}^k(\R^{d};\Rm)}:=\|\bm{f}\|_{C^{k-1}_b(\R^{d};\Rm)}+[D^{k-1}\bm{f}]_{{\mathcal{C}^1(\R^{d};\R^{md^{k-1}})}}.
\end{align*}
\end{defn}

Now, we recall the definition of Besov spaces.

\begin{defn}%\label{defn_Sob_Bes}
Let $s>0$ and $p \in [1,\infty)$. For $s\notin\N$, $B^{s}_{p,p}(\Rd;\Rm)$ denotes the set of functions $\f \in W^{[s],p}(\Rd;\Rm)$ such that
\begin{eqnarray*}
[\f]_{B^s_{p,p}(\Rd;\Rm)}:=\sum_{|\alpha|=[s]}\left(\int_{\Rd}\left(\int_{\Rd}|D^\alpha \f(x+h)-D^\alpha \f(x)|^p dx\right)\frac{dh}{|h|^{d+\{s\}p}}\right)^{\frac{1}{p}}
\end{eqnarray*}
is finite.
We endow the space $B^{s}_{p,p}(\Rd;\Rm)$ with the  norm 
\begin{align*}
\|\f\|_{B^{s}_{p,p}(\Rd;\Rm)}=\|\f\|_{W^{[s],p}(\Rd;\Rm)}+ [\f]_{B^s_{p,p}(\Rd;\Rm)}.
\end{align*}
\noindent
The space $B^1_{p,p}(\Rd;\Rm)$ consists of the functions $\f\in L^p(\Rd;\Rm)$ such that
\begin{align*}
[\f]_{B^1_{p,p}(\Rd;\Rm)}:=\left(\int_{\Rd}\left(\int_{\Rd}|\f(x+2h)-2 \f(x+h)+ \f(x)|^p dx\right)\frac{dh}{|h|^{d+p}}\right)^{\frac{1}{p}}
\end{align*}
is finite.
The space $B^1_{p,p}(\Rd;\Rm)$ is endowed with the norm 
\begin{eqnarray*}
\|\f\|_{B^1_{p,p}(\Rd;\Rm)}= \|\f\|_{L^p(\Rd;\Rm)}+[\f]_{B^1_{p,p}(\Rd;\Rm)}.
\end{eqnarray*}
For $k \in \N$, $k\ge 2$, we set
$B^k_{p,p}(\Rd;\Rm)=\{\f \in W^{k-1,p}(\Rd;\Rm)\,|\, D^{k-1}\bm{f} \in B^1_{p,p}(\Rd;\R^{md^{k-1}})\}$,
endowed with the norm
\begin{eqnarray*}
\|\f\|_{B^k_{p,p}(\Rd;\Rm)}=\|\f\|_{W^{k-1,p}(\Rd;\Rm)}+[D^{k-1}\bm{f}]_{B^1_{p,p}(\Rd;\R^{md^{k-1}})}.
\end{eqnarray*}
\end{defn}

We shall use a few tools from interpolation theory. 
% $K$-method for real interpolation spaces (see, for example, \cite{Lun_interpolation,Tri}). Let $\K_1$ and $\K_2$ be two Banach spaces, with norms $\norm{\cdot}_{\K_1}$ and $\norm{\cdot}_{\K_2}$, respectively. If $\K_2\subseteq \K_1$ with a continuous embedding, then for every $r>0$ and $x\in \K_1$ we define
% \begin{align*}
% K(r,x):=\inf\set{\|a\|_{\K_1}+r\|b\|_{\K_2}\tc x=a+b,\ a\in \K_1,\ b\in \K_2}.
% \end{align*}
% For any $\vartheta\in(0,1)$ and $1\le p\le \infty$, we set
% \begin{align*}
% (\K_1,\K_2)_{\vartheta,p}&:=\left\{x\in \K_1+\K_2\,: r\mapsto r^{-\vartheta}K(r,x) \in L^p((0,\infty),t^{-1}dt)\right\}.
% \end{align*}
% The functional space $(\K_1,\K_2)_{\vartheta,p}$ endowed with the norm
% $$\|x\|_{(\K_1,\K_2)_{\vartheta,p}}=\|r^{-\vartheta}K(r,x) \|_{ L^p((0,\infty),x^{-1}dx)}$$
% is a Banach space. 
The following are standard interpolation results concerning both spaces of continuous functions and $L^p$-spaces.

\begin{thm}\label{thm_interpolation}
Let $\alpha\in (0,1)$, $k \in \N$ and $\theta>0$. It holds that
\begin{align}
\label{interpolation_uguaglianza_1}
(C_b(\R^d;\R^m),C_b^k(\R^d;\R^m))_{\alpha,\infty}=\mathcal{C}^{\alpha k}_b(\R^d;\R^m)
\end{align}
and
\begin{align*}
(C_b(\R^d;\R^m),\mathcal{C}^{\theta}_b(\R^d;\R^m))_{\alpha,\infty}=\mathcal{C}^{\alpha \theta}_b(\R^d;\R^m),
\end{align*}
with equivalence of the corresponding norms.
\end{thm}

\begin{proof}
Identity \eqref{interpolation_uguaglianza_1} follows from \cite[Example 1.10]{Lun_interpolation} applying it componentwise. By \eqref{interpolation_uguaglianza_1} and the reiteration theorem (see \cite[Corollary 1.24]{Lun_interpolation}), we get
\begin{align*}
(C_b(\R^d;\R^m),\mathcal{C}^{\theta}_b(\R^d;\R^m))_{\alpha,\infty} &=\bigg(C_b(\R^d;\R^m),\big(C_b(\R^d;\R^m),C_b^{1+[\theta]}(\R^d;\R^m)\big)_{\frac{\theta}{1+[\theta]},\infty}\bigg)_{\alpha,\infty}\\
&=(C_b(\R^d;\R^m),C_b^{1+[\theta]}(\R^d;\R^m))_{\frac{\alpha\theta}{1+[\theta]},\infty}\\
&=\mathcal{C}^{\alpha\theta}_b(\R^d;\R^m).
\end{align*}
This concludes the proof.
\end{proof}

In an analogous way (using again \cite[Example 1.10 and Corollary 1.24]{Lun_interpolation}), one can prove the following result.

\begin{thm}
Let $\alpha\in (0,1)$, $k \in \N$ and $\theta>0$. It holds that
\begin{equation}\label{int_p}
(L^p(\Rd;\Rm), W^{k,p}(\Rd;\Rm))_{\alpha,p}=B^{\alpha k}_{p,p}(\Rd;\Rm)
\end{equation}
and
\begin{equation}\label{int_pp}
(L^p(\Rd;\Rm),B^{\theta}_{p,p}(\Rd;\Rm))_{\alpha,p}=B^{\alpha \theta}_{p,p}(\Rd;\Rm).
\end{equation}
\end{thm}

The following result can be found in \cite[Theorem 1.6]{Lun_interpolation}.

\begin{thm}\label{thm_operator_interpolation}
Let $X_i, Y_i$, $i=0,1$, be Banach spaces such that $X_1\subseteq X_0$ and $Y_1\subseteq Y_0$ with continuous embeddings, let $\theta \in (0,1)$ and $q\in[1,\infty]$. If $R$ is a linear mapping such that $\|R f\|_{Y_i}\le M_i\|f\|_{X_i}$ for any $f\in X_i$ and some positive constants $M_i$, $i=0,1$, then $R$ is bounded from  $(X_0,X_1)_{\theta,q}$ into  $(Y_0,Y_1)_{\theta,q}$
and $\|Rf\|_{(Y_0,Y_1)_{\theta,q}}\le M_0^{1-\theta}M_1^\theta\|f\|_{(X_0,X_1)_{\theta,q}}$, for any $f \in (X_0,X_1)_{\theta,q}$. 
\end{thm}

\section{Pointwise estimates for the spatial derivatives of ${\bm T}(t)\f$}
\label{sect-3}
In this section we prove pointwise estimates for the spatial derivatives of the function ${\bm T}(t)\f$ in terms of the semigroup $S_{\nu}(t)$ associated to $\A_{\nu}$ in $C_b(\Rd)$ (see Proposition \ref{prop-2.5}), when $\f$ is a $\R^m$-valued bounded and continuous function on $\R^d$ or it is even smoother. From here onward, Hypotheses \ref{base} will be our standing assumptions throughout the whole paper.

\begin{thm}%\label{03}
Assume Hypotheses $\ref{hyp-derivata-1}$ hold true. There exists a positive constant $c$, independent of $t$, $x$ and $\f$, such that
\begin{equation}\label{pointwise}
|(D^{\ell} {\bm T}(t)\f)(x)|^2 \le c e^{Ht}\max\{t^{-(\ell-k)},1\}\bigg (S_{\nu}(t)\sum_{j=0}^k|D^j\f|^2\bigg )(x),\qquad\;\,t\in (0,\infty),\;\,x\in\R^d,
\end{equation}
for every $k,\ell\in\{0,1\}$ with $k\le \ell$ and $\f \in C_b^k(\Rd;\Rm)$, where $H$ is the constant introduced in Hypothesis $\ref{base}(iii)$.
If Hypotheses $\ref{hyp-derivata-1}$ are replaced by
Hypotheses $\ref{hyp-derivata-2}$ $($resp. Hypotheses $\ref{hyp-derivata-3})$, then estimate $\eqref{pointwise}$ is satisfied by
any $k,\ell\in\{0,1,2\}$ $($resp.
$k,\ell\in\{0,1,2,3\})$ with $k\le \ell$.
\end{thm}

\begin{proof}
Let $\ell, k$ and $\f$ be as in the statement. We will show that
\begin{equation}\label{poi-pro}
t^{\ell-k}|D^{\ell} {\bm T}(t)\f(x)|^2 \le c_*\bigg (S_{\nu}(t)\sum_{j=0}^k|D^j\f|^2\bigg )(x),\qquad\;\,t\in (0,1],\;\,x\in\R^d,
\end{equation}
for some positive constants $c_*$, independent of  $t$, $x$ and $\f$. Once \eqref{poi-pro} is proved, the case $t \in (1,\infty)$ can be covered by using the semigroup law and the fact that the semigroup $\{S_{\nu}(t)\}_{t\ge 0}$ is positive.
Indeed, fix $t\in (1,\infty)$. Thanks to \eqref{poi-pro} (with $k=0$) and \eqref{cossali}, which show that $|D^{\ell}\bm{T}(1)\bm{g}|^2\le c_*S_{\nu}(1)|\bm{g}|^2$ and $|\bm{T}(t-1)\bm{g}|^2\le e^{H(t-1)}S_{\nu}(t-1)|\bm{g}|^2$ for every $\bm{g}\in C_b(\R^d;\R^m)$,
and the semigroup law, we can estimate 
\begin{align*}
|D^\ell {\bm T}(t)\f|^2=|D^{\ell}\big({\bm T}(1){\bm T}(t-1)\f\big)|^2
\le c_*S_{\nu}(1)|{\bm T}(t-1)\f|^2
=c_*e^{H(t-1)}S_{\nu}(t)|\f|^2.
\end{align*}
Hence \eqref{pointwise} follows.

Now, in order to prove \eqref{poi-pro} we introduce a few notation and some preliminary estimates that we use throuoghout the proof. Given a smooth enough function $\bm{w}:\R^d\ra\R^m$, we set
\begin{align*}
%\begin{array}{ll}
&\mathfrak{q}_0(\bm{w})=\displaystyle\sum_{j=1}^m\langle Q\nabla w_j,\nabla w_j\rangle,\quad & &
\mathfrak{q}_1(\bm{w})=\displaystyle\sum_{j=1}^m\sum_{h,i,l=1}^d\partial_hq_{il}\partial_hw_j\partial_{il}w_j,\\ %[1mm]
&\displaystyle\mathfrak{q}_2(\bm{w})=\sum_{j=1}^d\sum_{h,i,l,s=1}^d\partial_{hs}q_{il}\partial_{hs}w_j\partial_{il}w_j, & & \displaystyle\mathfrak{b}_1(\bm{w})=\sum_{j=1}^m\langle D^1\bm{b}\nabla w_j,\nabla w_j\rangle,\\ %[1mm]
&\displaystyle\mathfrak{b}_2(\bm{w})=
\sum_{j=1}^m\sum_{h,l,s=1}^d\partial_{hs}b_l\partial_{hs}w_j\partial_lw_j, & & \displaystyle\mathfrak{B}_0(\bm{w})=\sum_{i,j=1}^m\sum_{l=1}^d(\widehat B_l)_{ji}w_j\partial_lw_i,\\%[1mm]
&\displaystyle\mathfrak{B}_1(\bm{w})=\sum_{i,j=1}^m\sum_{h,l=1}^d\partial_h(\widehat B_l)_{ji}\partial_hw_j\partial_{l}w_i, & &
\displaystyle\mathfrak{B}_2(\bm{w})=\sum_{i,j=1}^m\sum_{h,l,s=1}^d\partial_{hs}(\widehat B_l)_{ji}\partial_lw_i\partial_{hs}w_j,\\%[1mm]
&\mathfrak{c}_0(\bm{w})=\langle C\bm{w},\bm{w}\rangle, & &
\displaystyle\mathfrak{c}_1(\bm{w})=\sum_{i,j=1}^m\partial_hc_{ji}w_i\partial_hw_j,\\%[1mm]
&\displaystyle\mathfrak{c}_2(\bm{w})=\sum_{j=1}^m\sum_{h,r,s=1}^d\partial_{hs}c_{ji}w_i\partial_{hs}w_j,
& &\displaystyle\mathfrak{O}(\bm{w})=\sum_{j=1}^m\sum_{h=1}^d
\langle Q\nabla\partial_hw_j,\nabla\theta\rangle\partial_hw_j.
%\end{array}
\end{align*}
% Similarly, we set $\mathfrak{b}_0(\bm{w})=\sum_{j=1}^m\langle J_b\nabla w_j,\nabla w_j\rangle$, $\mathfrak{c}_0(\bm{w})=\langle C\bm{w},\bm{w}\rangle$ and define  $\mathfrak{b}_j(\bm{w})$, $\mathfrak{c}_j(\bm{w})$ $(j=1,2,3)$ as the corresponding terms $\mathfrak{q}_j(\bm{w})$, with $\mathfrak{q}_0$ being replaced by $\mathfrak{b}_0$ and $\mathfrak{c}_0$, respectively.
Note that 
\begin{align}\label{stup_estim}
\mathfrak{q}_0(\bm{w})\ge\lambda_Q|D^1\bm{w}|^2,\qquad\;\,\mathfrak{b}_1(\bm{w})\le\Lambda_{D^1{\bm b}}|D^1\bm{w}|^2,\qquad\;\,\mathfrak{c}_0(\bm{w})\le\Lambda_C|\bm{w}|^2.
\end{align}
In what follows it will be useful to consider the sequence $(\vartheta_n)_{n\in\N}$ defined by $\vartheta_n(x):=\psi(n^{-1}|x|)$ for every $x\in\Rd$ and every $n\in\N$, where the function $\psi$ is a smooth decreasing function in $[0,\infty)$ such that $\chi_{[0,1]}(\xi)\le \psi(\xi) \le \chi_{[0,2]}(\xi)$ for any $\xi \in [0,\infty)$. Taking \eqref{maxaut<minaut} into account, it is easy to show that, for any $n\le |x|\le 2n$, we can estimate
\begin{align}\label{qtheta}
|Q^{\frac{1}{2}}(x)\nabla \vartheta_n(x)|=\frac{|Q^{\frac{1}{2}}(x)x|}{n|x|}\left|\psi'\bigg (\frac{|x|}{n}\right)\bigg |
\le 2|x|^{-1}(\Lambda_Q(x))^{1/2}\|\psi'\|_\infty
\le 2\sqrt{2\kappa}c_1(\lambda_Q(x))^{1/2},
\end{align}
where $c_1=\|\psi'\|_{\infty}$ and $\kappa$ is the constant introduced in \eqref{maxaut<minaut}. In a similar way, letting $\A_0={\rm Tr}(QD^2)+\langle\bm{b},\nabla\rangle$, it holds
\begin{align*}
({\mathcal A}_0\vartheta_n)(x)=\frac{\psi'(n^{-1}|x|)}{n|x|}\pa{{\rm Tr}(Q(x))-\frac{\langle Q(x)x,x\rangle}{|x|^2}+\langle \bm{b}(x),x\rangle}
+\psi''(n^{-1}|x|)\frac{\langle Q(x)x,x\rangle}{n^2|x|^2},
\end{align*}
by Hypothesis \ref{base}(v), we can estimate
\begin{align}\label{Atheta}
{\mathcal A}_0\vartheta_n\ge -
c_2\lambda_Q(x),\qquad\;\, x\in\Rd\setminus B(0,n),
\end{align}
for every $n\ge n_0$, and a suitable constant $c_2$, which depends on  $\kappa$ (the constant introduced in \eqref{maxaut<minaut}) and $\|\psi''\|_{\infty}$.

We can now prove the claim  under Hypotheses \ref{hyp-derivata-3}, the other cases will be treated later. For a fixed $\f \in C^{\infty}_c(\Rd;\R^m)$ and $n\ge n_0$, let us consider the classical solution $\uu_n$ to the Cauchy--Dirichlet problem
\begin{equation*}
\left\{
\begin{array}{ll}
D_t\ww(t,x)=(\bm{\mathcal A}\ww)(t,x),\qquad\;\, &(t,x)\in (0,\infty)\times B(0,n),\\[1mm]
\ww(t,x)=\bm{0}, & (t,x)\in (0,\infty)\times\partial B(0,n),\\[1mm]
\ww(0,x)=\f(x),& x \in B(0,n).
\end{array}
\right.
\end{equation*}
By Theorem \ref{teo-reg-int},  $\uu_n$ belongs to $\mathcal{C}^{1+\alpha/2,5+\alpha}_{\rm loc}([0,T]\times B(0,n))$ for every $T>0$. For $\ell$ and $k$ as in the statement, we set $\ell_k:=\max\{0,\ell-k\}$. 
Moreover, we set $r_{2,k}=2_k-1_k$ and $r_{3,k}=3_k-2_k$. For any $a \in (0,1)$ the function $v_{k,n}^{(3)}:[0,\infty)\times\overline{B(0,n)}\to\R$, defined by
\begin{align*}
v_{k,n}^{(3)}(t,x):=&|\uu_n(t,x)|^2+at^{1_k}(\vartheta_n(x))^2|D^1\uu_n(t,x)|^2+a^2t^{2_k}(\vartheta_n(x))^4|D^2\uu_n(t,x)|^2\\
&+a^3t^{3_k}(\vartheta_n(x))^6|D^3\uu_n(t,x)|^2
\end{align*}
for every $(t,x)\in [0,\infty)\times\overline{B(0,n)}$, belongs to $C^{1,2}((0,\infty)\times B(0,n))$ and is bounded and continuous in
$([0,T]\times\overline{B(0,n)})\setminus (\{0\}\times\partial B(0,n))$
for every $T>0$. 
To lighten the notation we drop the subscripts $k$ and $n$ in $v_{k,n}^{(3)}$.

Taking into account that $\ell_kt^{\ell_k-1}=\ell_kt^{(\ell-1)_k}$ for every $t\in (0,1]$ and $\ell,k$ as in the statement  of the theorem, a very long but straightforward computation yields that
$D_tv^{(3)}-{\mathcal A}_{\nu}v^{(3)}=\sum_{i=0}^4 g_{i,i}^{(3)}+\sum_{i=0}^3g_{i,i+1}^{(3)}+\sum_{i=0}^1g_{i,i+2}^{(3)}+g_{0,3}^{(3)}=:G^{(3)}$, with
\begin{align*}
g_{0,0}^{(3)}=& 2\mathfrak{c}_0(\uu)-\nu\Lambda_C|\uu|^2;\\[1mm]
g_{1,1}^{(3)}=&-2\mathfrak{q}_0(\uu)+1_ka\vartheta^2|D^1\uu|^2+2at^{1_k}\vartheta^2\mathfrak{b}_1(\uu)+2at^{1_k}\vartheta^2\sum_{h=1}^d\mathfrak{c}_0(\partial_h\uu)
-a\nu t^{1_k}\Lambda_C\vartheta^2|D^1\uu|^2\\
&+2at^{1_k}\vartheta^2\mathfrak{B}_1(\uu)-2at^{1_k}\langle Q\nabla\vartheta,\nabla\vartheta\rangle|D^1\uu|^2-2at^{1_k}\vartheta{\mathcal A}_0\vartheta|D^1\uu|^2;\\[1mm]
g_{2,2}^{(3)}=& -2at^{1_k}\vartheta^2\sum_{h=1}^d\mathfrak{q}_0(\partial_h\uu)+2_ka^2t^{1_k}\vartheta^4|D^2 \uu|^2+2a^2t^{2_k}\vartheta^4\mathfrak{q}_2(\uu)+4a^2t^{2_k}\vartheta^4\sum_{h=1}^d\mathfrak{b}_1(\partial_h\uu)\\
&+4a^2t^{2_k}\vartheta^4\sum_{s=1}^d\mathfrak{B}_1(\partial_s\uu)+2a^2t^{2_k}\vartheta^4\sum_{h,r=1}^d\mathfrak{c}_0(\partial_{hr}\uu)-12a^2t^{2_k}\vartheta^2\langle Q\nabla\vartheta,\nabla\vartheta\rangle|D^2\uu|^2\\
&-4a^2t^{2_k}\vartheta^3{\mathcal A}_0\vartheta|D^2\uu|^2
-a^2\nu t^{2_k}\Lambda_C\vartheta^4|D^2\uu|^2;\\[1mm]
g_{3,3}^{(3)}=& -2a^2t^{2_k}\vartheta^4\sum_{h,r=1}^d\mathfrak{q}_0(\partial_{hr}\uu)+3_ka^3t^{2_k}\vartheta^6|D^3\uu|^2+6a^3t^{3_k}\vartheta^6\sum_{r=1}^d\mathfrak{q}_2(\partial_ru)+6a^3t^{3_k}\vartheta^6\sum_{h,r=1}^d\mathfrak{b}_1(\partial_{hr}\uu)\\
&
+6a^3t^{3_k}\vartheta^6\sum_{j=1}^m\sum_{r,s=1}^d\mathfrak{B}_1(\partial_{rs}\uu)
+2a^3t^{3_k}\vartheta^6\sum_{h,r,s=1}^d\mathfrak{c}_0(\partial_{hrs}\uu)-30a^3t^{3_k}\vartheta^4\langle Q\nabla\vartheta,\nabla\vartheta\rangle|D^3\uu|^2\\
&-6a^3t^{3_k}\vartheta^5{\mathcal A}_0\vartheta|D^3\uu|^2-a^3\nu t^{3_k}\Lambda_C\vartheta^6|D^3\uu|^2;\\[1mm]
g_{4,4}^{(3)}=&-2a^3t^{3_k}\vartheta^6\sum_{h,r,s=1}^d\mathfrak{q}_0(\partial_{hrs}\uu);\\[1mm]
g_{0,1}^{(3)}=& 2\mathfrak{B}_0(\uu)+2at^{1_k}\vartheta^2\mathfrak{c}_1(\uu);\\[1mm]
g_{1,2}^{(3)}=& 2at^{1_k}\vartheta^2\mathfrak{q}_1(\uu)+2at^{1_k}\vartheta^2\sum_{h=1}^d\mathfrak{B}_0(\partial_h\uu)-8at^{1_k}\vartheta\mathfrak{O}(\uu)+2a^2t^{2_k}\vartheta^4\mathfrak{b}_2(\uu)+2a^2t^{2_k}\vartheta^4\mathfrak{B}_2(\uu)\\
&+4a^2t^{2_k}\vartheta^4\sum_{s=1}^d\mathfrak{c}_1(\partial_s\uu);\\[1mm]
g_{2,3}^{(3)}=&4a^2t^{2_k}\vartheta^4\sum_{s=1}^d\mathfrak{q}_1(\partial_s\uu)+2a^2t^{2_k}\vartheta^4\sum_{s=1}^d\mathfrak{B}_0(\partial_{hs}\uu)-16a^2t^{2_k}\vartheta^3\sum_{s=1}^d
\mathfrak{O}(\partial_s\uu)+6a^3t^{3_k}\vartheta^6\sum_{r=1}^d\mathfrak{b}_2(\partial_r\uu)\\
&+2a^3t^{3_k}\vartheta^6\sum_{h,i,l,r,s=1}^d\sum_{j=1}^m\partial_{hrs}q_{li}\partial_{li}u_j\partial_{hrs}u_j
+6a^3t^{3_k}\vartheta^6\sum_{r=1}^d\mathfrak{B}_2(\partial_r\uu)+6a^3t^{3_k}\vartheta^6\sum_{r,s=1}^d\mathfrak{c}_1(\partial_{rs}\uu);\\[1mm]
g_{3,4}^{(3)}=& 6a^3t^{3_k}\vartheta^6\sum_{r,s=1}^d\mathfrak{q}_1(\partial_{rs}\uu)
+2a^3t^{3_k}\vartheta^6\sum_{r,s=1}^d\mathfrak{B}_0(\partial_{hrs}\uu)-24a^3t^{3_k}\vartheta^5\sum_{r,s=1}^d\mathfrak{O}(\partial_{rs}\uu);\\[1mm]
g_{0,2}^{(3)}=& 2a^2t^{2_k}\vartheta^4\mathfrak{c}_2(\uu);\\[1mm]
g_{1,3}^{(3)}=& 2a^3t^{3_k}\vartheta^6\sum_{h,l,r,s=1}^d\sum_{j=1}^m\partial_{hrs}b_{l}\partial_{l}u_j\partial_{hrs}u_j+2a^3t^{3_k}\vartheta^6\sum_{j=1}^m\sum_{h,l,r,s=1}^d\sum_{\substack{i=1\\i\neq j}}^m\partial_{hrs}(\hat{B}_l)_{ji}\partial_{l}u_i\partial_{hrs}u_j\\
&+6a^3t^{3_k}\vartheta^6\sum_{r=1}^d\mathfrak{c}_2(\partial_r\uu);\\[1mm]
g_{0,3}^{(3)}=&2a^3t^{3_k}\vartheta^6\sum_{i,j=1}^m\sum_{h,r,s=1}^d\partial_{hrs}c_{ji}u_i\partial_{hrs}u_j.
\end{align*}

\paragraph{\bf Estimates of the terms $\bm{g_{i,i}}^{(3)}$} Clearly, we can estimate $g_{0,0}^{(3)}$ from above by $(2-\nu)\Lambda_C|\uu|^2$.
Thanks to the Young inequality, we can write
\begin{align}
2at^{1_k}\vartheta^2\mathfrak{B}_1(\uu)\leq 2at^{1_k}\vartheta^2\hat{\beta}_1d|D^1\uu|^2\le  t^{1_k}\big (\sqrt{a}\vartheta^2d\hat{\beta}_1^{{\mu_1}}+a\sqrt{a}\vartheta^2d\hat{\beta}_1^{2-{\mu_1}}\big )|D^1\uu|^2.
\label{B1-est}
\end{align}
On the other hand, thanks to \eqref{Atheta}, it follows that
$-2a\vartheta{\mathcal A}_0\vartheta\leq 2a\vartheta c_2\lambda_Q \le 2ac_2\lambda_Q$.
Consequently, recalling that $\langle Q \nabla \vartheta,\nabla \vartheta\rangle \ge 0$ and %$\widehat\beta_1^{\mu_1}\le A_2\lambda_Q$
Hypotheses \ref{hyp-derivata-1}(i), we deduce
that
\begin{align}
g_{1,1}^{(3)}\leq &\big[-2\lambda_Q+1_ka\vartheta^2+\sqrt{a}dA_2\vartheta^2t^{1_k}\lambda_Q+2at^{1_k}\vartheta^2(\Lambda_{D^1{\bm b}}+\Lambda_C)-a\nu t^{1_k}\Lambda_C\vartheta^2\notag\\
&\;+a\sqrt{a}t^{1_k}\vartheta^2d\hat{\beta}_1^{2-{\mu_1}}+2at^{1_k}c_2\lambda_Q\big ]|D^1\uu|^2.
\label{estim-g11}
\end{align}

Arguing similarly, estimating the third, fifth (taking \eqref{B1-est} into account) and eighth terms in $g_{2,2}^{(3)}$ as follows:
\begin{align}
&2a^2t^{2_k}\vartheta^4\mathfrak{q}_2(\uu)\leq 2a^2 t^{2_k}\vartheta^4\xi_2 d|D^2\uu|^2\leq a^2dt^{2_k}\bigg (\frac{1}{\sqrt{a}}\vartheta^2\xi_2^{\alpha_2}+\sqrt{a}\vartheta^4\xi_2^{2-\alpha_2}\bigg )|D^2\uu|^2;\label{q2-est}\\
&4a^2t^{2_k}\vartheta^4\sum_{s=1}^d\mathfrak{B}_1(\partial_s\uu)\leq 4a^2t^{2_k}\vartheta^4\hat{\beta}_1 d|D^2\uu|^2\leq \big (2a\sqrt{a}t^{2_k}\vartheta^2d\hat{\beta}_1^{{\mu_1}}+2a^2\sqrt{a}t^{2_k}\vartheta^4\hat{\beta}_1^{2-{\mu_1}}\big )|D^2\uu|^2\notag
\end{align}
and
$-4a^2\vartheta^3\A_0\vartheta
\le 4a^2\vartheta^2c_2\lambda_Q$,
we obtain (taking Hypotheses \ref{hyp-derivata-3}(i) into account) that
\begin{align}
g_{2,2}^{(3)}\le & -2at^{1_k}\vartheta^2\sum_{h=1}^d\mathfrak{q}_0(\partial_h\uu)\notag\\
&+\big(2_ka^2t^{1_k}\vartheta^4+a\sqrt{a}A_3dt^{2_k}\vartheta^2\lambda_Q+a^2\sqrt{a}t^{2_k}\vartheta^4d\xi_2^{2-{\alpha_2}}+4a^2t^{2_k}\vartheta^4\Lambda_{D^1{\bm b}}+
2a\sqrt{a}A_2t^{2_k}\vartheta^2d\lambda_Q\notag\\
&\quad\;\,+2a^2\sqrt{a}t^{2_k}\vartheta^4\hat{\beta}_1^{2-{\mu_1}}+ (2-\nu)a^2t^{2_k}\vartheta^4\Lambda_C+4a^2t^{2_k}\vartheta^2c_2\lambda_Q\big)|D^2\uu|^2.
\label{estim-g22}
\end{align}

Next, taking \eqref{B1-est} and \eqref{q2-est} into account, we estimate the third and the fifth terms in $g_{3,3}^{(3)}$ as follows:
\begin{align*}
&6a^3t^{3_k}\vartheta^6\sum_{r=1}^d\mathfrak{q}_2(\partial_r\uu)
\leq 3a^2\sqrt{a}t^{3_k}\vartheta^4d\xi_2^{\alpha_2}|D^3\uu|^2
+3a^3t^{3_k}\vartheta^6d\sqrt{a}\xi_2^{2-\alpha_2}|D^3\uu|^2;\\
&6a^3t^{3_k}\vartheta^6\sum_{r,s=1}^d\mathfrak{ B}_1(\partial_{rs}\uu)\leq 3a^2\sqrt{a}t^{3_k}\vartheta^4d\widehat\beta_1^{\mu_1}|D^3\uu|^2
+3a^3\sqrt{a}t^{3_k}\vartheta^6d\widehat\beta_1^{2-\mu_1}|D^3\uu|^2.
\end{align*}
Moreover, using again \eqref{Atheta} we get the following estimate for the last term in $g_{3,3}^{(3)}$:
$-6a^3t^{3_k}\vartheta^5\tilde{\mathcal A}\vartheta
\le 6a^3t^{3_k}\vartheta^4c_2\lambda_Q$. 
Putting all together and taking Hypothesis \ref{hyp-derivata-3}(i) into account, we conclude that
\begin{align*}
g_{3,3}^{(3)}\leq &a^2t^{2_k}\vartheta^4\bigg [-2\sum_{h,r=1}^d\mathfrak{q}_0(\partial_{hr}\uu)+3\sqrt{a}dt^{r_{3,k}}(A_2+A_3)\lambda_Q|D^3\uu|^2+a\big (3_k
+6t^{r_{3,k}}c_2\lambda_Q\big )
|D^3\uu|^2\bigg ]%\notag
\\
&+a^3t^{3_k}\vartheta^6\big [6\Lambda_{D^1{\bm b}}+(2-\nu)\Lambda_C
+3d\sqrt{a}\big (\xi_2^{2-\alpha_2}+\hat{\beta}_1^{2-\mu_1}\big )
\big ]|D^3\uu|^2.
%\label{estim-g33}
\end{align*}

\paragraph{\bf Estimates of the terms $\bm{g_{i,i+1}}^{(3)}$}
As far as $g_{0,1}^{(3)}$ is concerned, properly using the Young inequality, we deduce that
\begin{align}
2\mathfrak{B}_0(\uu)\le \frac{\widehat\beta_0^2d}{\lambda_Q\varepsilon_4}|\uu|^2
+\varepsilon_4\lambda_Q|D\uu|^2,\qquad\;\,
2at^{1_k}\vartheta^2\mathfrak{c}_1(\uu)\le 
at^{1_k}\vartheta^2\pa{\frac{d\gamma_1^{{{\tau_1}}}}{\sqrt{a}}|\uu|^2+\sqrt{a}\gamma_1^{2-{{\tau_1}}}|D^1\uu|^2}
\label{B0-c1-est}
\end{align}
and conclude that
\begin{align}
g_{0,1}^{(3)}\leq &\bigg (\frac{\hat{\beta}_0^2d}{\lambda_Q\varepsilon_4} +\sqrt{a}dt^{1_k}\vartheta^2\gamma_1^{{{\tau_1}}}\bigg )|\uu|^2+\pa{\varepsilon_4\lambda_Q+a\sqrt{a}t^{1_k}\vartheta^2\gamma_1^{2-{{\tau_1}}}}|D^1\uu|^2.
\label{estim-g01}
\end{align}

Concerning $g_{1,2}^{(3)}$ we can estimate its terms as follows:
\begin{align}
2at^{1_k}\vartheta^2\mathfrak{q}_1(\uu)
\leq at^{1_k}\vartheta^2\pa{\frac{d\xi_1^{2-\alpha_1}}{\eps_1}|D^1\uu|^2+\eps_1\xi_1^{\alpha_1}|D^2\uu|^2}
\label{num-0}
\end{align}
for every $\varepsilon_1>0$, and, using the condition $\max_{i=1,\ldots,d}|\widehat B_i(x)|\le\widehat\beta_0(x)$ for every $x\in\R^d$, in Hypotheses \ref{base}(iii), we obtain
\begin{equation}
2at^{1_k}\vartheta^2\sum_{h=1}^d\mathfrak{B}_0(\partial_h\uu)\leq
a\frac{\widehat\beta_0^2d}{\lambda_Q\varepsilon_5}t^{1_k}\vartheta^2|D^1\uu|^2
+a\varepsilon_5t^{1_k}\vartheta^2
\lambda_Q|D^2\uu|^2.
\label{B0-est}
\end{equation}
Moreover, taking \eqref{qtheta} into account, we get
\begin{align}
-8at^{1_k}\vartheta\mathfrak{O}(\uu)
&\leq 8at^{1_k} \vartheta \sum_{j=1}^m \sum_{h=1}^d|Q^{\frac{1}{2}}\nabla \partial_hu_j||Q^{\frac{1}{2}}\nabla \vartheta||\partial_hu_j|\notag\\
&\leq 8at^{1_k} \vartheta\bigg (\sum_{j=1}^m \sum_{h=1}^d|Q^{\frac{1}{2}}\nabla \partial_hu_j|^2\bigg )^{\frac{1}{2}}|D^1\uu||Q^{\frac{1}{2}}\nabla \vartheta|\notag\\
& \le 4a\sqrt{a}t^{1_k}\vartheta^2\sum_{h=1}^d\mathfrak{q}_0(\partial_h\uu) +4\sqrt{a}t^{1_k} |Q^{\frac{1}{2}}\nabla \vartheta|^2|D^1\uu|^2\notag\\
& \le 4a\sqrt{a}t^{1_k}\vartheta^2\sum_{h=1}^d\mathfrak{q}_0(\partial_h\uu)+32\sqrt{a}t^{1_k}\kappa\lambda_Qc_1^2|D^1\uu|^2.
\label{num}
\end{align}
Further, it holds that
\begin{align}
&2a^2t^{2_k}\vartheta^4\mathfrak{b}_2(\uu)\leq a\sqrt{a}t^{2_k}\vartheta^4\beta_2^{\rho_2} d^3|D^1\uu|^2+
a^2\sqrt{a}t^{2_k}\vartheta^4
\beta_2^{2-\rho_2}|D^2\uu|^2;\label{b2-est}\\[1mm]
&2a^2t^{2_k}\vartheta^4\mathfrak{B}_2(\uu)\leq a\sqrt{a}t^{2_k}\vartheta^4\hat{\beta}_2^{\mu_2} d^3|D^1\uu|^2+a^2\sqrt{a}t^{2_k}\vartheta^4\hat{\beta}_2^{2-\mu_2}|D^2\uu|^2;
\label{B2-est}\\[1mm]
&4a^2t^{2_k}\vartheta^4\sum_{s=1}^d\mathfrak{c}_1(\partial_s\uu)\leq 2a\sqrt{a}t^{2_k}\vartheta^4\gamma_1^{2-\tau_1}d|D^1\uu|^2+2a^2\sqrt{a}t^{2_k}\vartheta^4\gamma_1^{\tau_1}|D^2\uu|^2.\notag
\end{align}
Summing up, using again Hypothesis \ref{hyp-derivata-3}(i), we conclude that
\begin{align}
g_{1,2}^{(3)}\le & 4a\sqrt{a}t^{1_k}\vartheta^2\sum_{h=1}^d\mathfrak{q}_0(\partial_h\uu)+a\vartheta^2\frac{\widehat\beta_0^2d}{\lambda_Q\varepsilon_5}t^{1_k}|D^1\uu|^2\notag\\
&+\sqrt{a}t^{1_k}\bigg (\sqrt{a}\vartheta^2\frac{d\xi_1^{2-\alpha_1}}{\eps_1}+32\kappa c_1^2\lambda_Q+at^{r_{2,k}}\vartheta^4d(\beta_2^{\rho_2}d^2+\hat{\beta}_2^{\mu_2} d^2+2\gamma_1^{2-\tau_1})\bigg )|D^1\uu|^2\notag\\
&+at^{1_k}\big [A_1\eps_1\vartheta^2\lambda_Q+\varepsilon_5\vartheta^2\lambda_Q+a\sqrt{a}t^{r_{2,k}}\vartheta^4(\beta_2^{2-{\rho_2}}+\hat{\beta}_2^{2-{{\mu_2}}}+2\gamma_1^{\tau_1})\big ]|D^2\uu|^2.
\label{estim-g12}
\end{align}

Now, to estimate $g_{2,3}^{(3)}$ we use the following inequalities (that are all, but the third one, direct consequences of \eqref{num-0}, \eqref{B0-est}, \eqref{b2-est}, \eqref{B2-est} and \eqref{B0-c1-est}, respectively): 
\begin{align}
&4a^2t^{2_k}\vartheta^4\sum_{s=1}^d\mathfrak{q}_1(\partial_s\uu)\leq  2a^2t^{2_k}\vartheta^4\pa{\frac{d\xi_1^{2-\alpha_1}}{\eps_2}|D^2\uu|^2+\eps_2\xi_1^{\alpha_1}|D^3\uu|^2};
\label{g23-1}
\\[1mm]
&2a^2t^{2_k}\vartheta^4\sum_{s=1}^d\mathfrak{B}_0(\partial_{hs}\uu)\leq a^2\frac{\hat{\beta}_0^2d}{\lambda_Q\varepsilon_6}t^{2_k}\vartheta^4|D^2\uu|^2+\varepsilon_6a^2t^{2_k}\vartheta^4\lambda_Q|D^3\uu|^2;
\notag%\label{g23-2}
\\[1mm]
&6a^3t^{3_k}\vartheta^6\sum_{r=1}^d\mathfrak{b}_2(\partial_r\uu)\leq 3a^2\sqrt{a}t^{3_k}\vartheta^6\beta_2^{2-\rho_2} d^3|D^2\uu|^2+3a^3\sqrt{a}t^{3_k}\vartheta^6\beta_2^{\rho_2}|D^3\uu|^2;\notag\\[1mm]
&2a^3t^{3_k}\vartheta^6\sum_{h,i,l,r,s=1}^d\sum_{j=1}^m\partial_{hrs}q_{li}\partial_{il}u_j\partial_{hrs}u_j \leq a^2\sqrt{a}t^{3_k}\vartheta^6\xi_3^{\alpha_3} d^3|D^2\uu|^2+a^3\sqrt{a}t^{3_k}\vartheta^6\xi_3^{2-\alpha_3}|D^3\uu|^2;\notag\\[1mm]
&6a^3t^{3_k}\vartheta^6\sum_{r=1}^d\mathfrak{B}_2(\partial_r\uu)\leq 3a^2\sqrt{a}t^{3_k}\vartheta^6\hat{\beta}^{2-\mu_2}_2 d^3|D^2\uu|^2+3a^3\sqrt{a}t^{3_k}\vartheta^6\hat{\beta}^{\mu_2}_2|D^3\uu|^2;\notag\\[1mm]
&6a^3t^{3_k}\vartheta^6\sum_{r,s=1}^d\mathfrak{c}_1(\partial_{rs}\uu)\leq 3a^2\sqrt{a}t^{3_k}\vartheta^6\gamma_1^{\tau_1}d|D^2\uu|^2+3a^3\sqrt{a}t^{3_k}\vartheta^6\gamma_1^{2-{\tau_1}}|D^3\uu|^2.
\notag
\end{align}
From \eqref{num}, we deduce that
\begin{align}
-16a^2t^{2_k}\vartheta^3\sum_{s=1}^d\mathfrak{O}(\partial_s\uu)
\le 8a^2\sqrt{a}t^{2_k}\vartheta^4\sum_{h,r=1}^d\mathfrak{q}_0(\partial_{hr}\uu)+64 a\sqrt{a}t^{2_k}\vartheta^2\kappa c_1^2\lambda_Q|D^2\uu|^2,
\label{g23-3}
\end{align}
obtaining
\begin{align*}
g_{2,3}^{(3)}\leq& 8a^2\sqrt{a}t^{2_k}\vartheta^4\sum_{h,r=1}^d\mathfrak{q}_0(\partial_{hr}\uu)
+\bigg(2a^2t^{2_k}\vartheta^4\frac{d\xi_1^{2-\alpha_1}}{\eps_2}+a^2t^{2_k}\frac{\widehat\beta_0^2d}{\lambda_Q\varepsilon_6}\vartheta^4+64a\sqrt{a}t^{2_k}\kappa c_1^2\vartheta^2\lambda_Q\\
&\qquad\qquad\qquad\qquad\qquad\qquad\quad\;\,
+a^2\sqrt{a}t^{3_k}\vartheta^6d\big (d^2\xi_3^{\alpha_3}
+3d^2\beta_2^{2-\rho_2}+3d^2\hat{\beta}^{2-\mu_2}_2 +3\gamma_1^{\tau_1}\big )\bigg)|D^2\uu|^2\\
&+\big[a^2t^{2_k}\vartheta^4(2A_1\eps_2+\varepsilon_6)\lambda_Q+a^3\sqrt{a}t^{3_k}\vartheta^6\xi_3^{2-\alpha_3}+3a^3\sqrt{a}t^{3_k}\vartheta^6\beta_2^{\rho_2}\\
&\qquad+3a^3\sqrt{a}t^{3_k}\vartheta^6\hat{\beta}^{\mu_2}_2+3a^3\sqrt{a}t^{3_k}\vartheta^6\gamma_1^{2-{\tau_1}}\big)|D^3\uu|^2.
\end{align*}
Finally, the term $g_{3,4}^{(3)}$ can be estimated as follows:
\begin{align*}
g_{3,4}^{(3)}\le &3a^3t^{3_k}\vartheta^6\bigg (\frac{d\xi_1^{2-\alpha_1}}{\eps_3}|D^3\uu|^2+\eps_3\xi_1^{\alpha}|D^4\uu|^2\bigg )+a^3t^{3_k}\vartheta^6\bigg (\frac{\widehat\beta_0^2d}{\lambda_Q\varepsilon_7}|D^3\uu|^2+\varepsilon_7\lambda_Q|D^4\uu|^2\bigg )\\
&+24a^3t^{3_k}\vartheta^5|Q^{\frac{1}{2}}\nabla\vartheta|\sum_{j=1}^m\sum_{h,r,s=1}^d|Q^{\frac{1}{2}}\nabla\partial_{hrs}u_j| |\partial_{hrs}u_j|\\
\leq &-4\sqrt{a}g_{44}^{(3)}+\bigg (144a^2\sqrt{a}t^{3_k}\vartheta^4\kappa c_1^2\lambda_Q+a^3\frac{\hat{\beta}_0^2d}{\lambda_Q\varepsilon_7}t^{3_k}\vartheta^6+3a^3dt^{3_k}\vartheta^6\frac{\xi_1^{2-\alpha_1}}{\eps_3}\bigg )|D^3\uu|^2\\
&+a^3t^{3_k}\vartheta^6\big (3A_1\eps_3+\varepsilon_7\big )\lambda_Q|D^4\uu|^2.
\end{align*}

\paragraph{\bf Estimates of the terms $\bm{g_{i,i+2}}^{(3)}$.}
We start by considering the term $g_{0,2}^{(3)}$ and note that
\begin{align}
g_{0,2}^{(3)}\le a\sqrt{a}t^{2_k}\vartheta^4 d^2\gamma_2^{\tau_2}|\uu|^2+a^2\sqrt{a}t^{2_k}\vartheta^4\gamma_2^{2-\tau_2}|D^2\uu|^2.
\label{estim-g02}
\end{align}

%Concerning $g_{1,3}^{(3)}$, we use the estimates
Further, we can estimate
\begin{align*}
g_{1,3}^{(3)}\le & a^2\sqrt{a}t^{3_k}\vartheta^6\beta_3^{\rho_3} d^4|D^1\uu|^2+a^3\sqrt{a}t^{3_k}\vartheta^6\beta_3^{2-\rho_3}|D^3\uu|^2\\
&+a^2\sqrt{a}t^{3_k}\vartheta^6\hat{\beta}^{{\mu_3}}_3 d^4|D^1\uu|^2+a^3\sqrt{a}t^{3_k}\vartheta^6\hat{\beta}^{2-{\mu_3}}_3|D^3\uu|^2\\
&
+3a^2\sqrt{a}t^{3_k}\vartheta^6\gamma_2^{{\tau_2}} d^2|D^1\uu|^2+3a^3\sqrt{a}t^{3_k}\vartheta^6\gamma_2^{2-{\tau_2}}|D^3\uu|^2\\
=&
a^2\sqrt{a}t^{3_k}\vartheta^6d^2\big (\beta_3^{\rho_3} d^2+\hat{\beta}^{\mu_3}_3 d^2+3\gamma_2^{\tau_2}\big )|D^1\uu|^2+ a^3\sqrt{a}t^{3_k}\vartheta^6\big (\beta_3^{2-\rho_3}+\hat{\beta}^{2-{\mu_3}}_3+3\gamma_2^{2-{\tau_2}}\big )|D^3\uu|^2.
\end{align*}
% \begin{align*} 2a^3t^{3_k}\vartheta^6\sum_{h,l,r,s=1}^d\sum_{j=1}^m\partial_{hrs}b_{l}\partial_{l}u_j\partial_{hrs}u_j \leq a^2\sqrt{a}t^{3_k}\vartheta^6\beta_3^{\rho_3} d^4|D^1\uu|^2+a^3\sqrt{a}t^{3_k}\vartheta^6\beta_3^{2-\rho_3}|D^3\uu|^2;
% \end{align*}
% \begin{align*}
% 2a^3t^{3_k}\vartheta^6\sum_{j=1}^m\sum_{h,l,r,s=1}^d\sum_{\substack{i=1\\i\neq j}}^m\partial_{hrs}(\hat{B}_l)_{ji}\partial_{l}u_i\partial_{hrs}u_j \leq a^2\sqrt{a}t^{3_k}\vartheta^6\hat{\beta}^{{\mu_3}}_3 d^4|D^1\uu|^2+a^3\sqrt{a}t^{3_k}\vartheta^6\hat{\beta}^{2-{\mu_3}}_3|D^3\uu|^2
% \end{align*}
% and
% \begin{align*}
% 6a^3t^{3_k}\vartheta^6\sum_{i,j=1}^m\sum_{h,r,s=1}^d\partial_{hs}c_{ji}\partial_{r}u_i\partial_{hrs}u_j \leq 3a^2\sqrt{a}t^{3_k}\vartheta^6\gamma_2^{{\tau_2}} d^2|D^1\uu|^2+3a^3\sqrt{a}t^{3_k}\vartheta^6\gamma_2^{2-{\tau_2}}|D^3\uu|^2
% \end{align*}
% to infer that
% \begin{align*}
% g_{1,3}^{(3)} \le 
% a^2\sqrt{a}t^{3_k}\vartheta^6d^2\big (\beta_3^{\rho_3} d^2+\hat{\beta}^{\mu_3}_3 d^2+3\gamma_2^{\tau_2}\big )|D^1\uu|^2+ a^3\sqrt{a}t^{3_k}\vartheta^6\big (\beta_3^{2-\rho_3}+\hat{\beta}^{2-{\mu_3}}_3+3\gamma_2^{2-{\tau_2}}\big )|D^3\uu|^2.
% \end{align*}

Finally, it holds that
\begin{align*}
g_{0,3}^{(3)}\leq a^2\sqrt{a}t^{3_k}\vartheta^6d^3\gamma_3^{\tau_3} |\uu|^2+a^3\sqrt{a}t^{3_k}\vartheta^6\gamma_3^{2-{\tau_3}}|D^3\uu|^2.
\end{align*}

Collecting all the terms, using \eqref{stup_estim}, estimating $\vartheta$ and $t$ with $1$, when the dependences on $\vartheta$  or on $t$ are not really needed, and taking $a<\frac{1}{64}$ to get
\begin{align*}
&(1-2\sqrt{a})\sum_{h=1}^d\mathfrak{q}_0(\partial_h\uu)
\ge \frac{3}{4}\lambda_Q|D^2\uu|^2,\qquad(1-4\sqrt{a})\sum_{h,r=1}^d\mathfrak{q}_0(\partial_{hr}\uu)
\ge \frac{1}{2}\lambda_Q|D^3\uu|^2,\\[1mm]
&(1-4\sqrt{a})\sum_{h,r,s=1}^d\mathfrak{q}_0(\partial_{hrs}\uu)\ge \frac{1}{2}\lambda_Q|D^4\uu|^2,
\end{align*}
in $(0,1]\times B(0,n)$, we obtain that $G^{(3)}\le\sum_{k=0}^4\mathcal{F}_k^{(3)}|D^k\uu|^2$, where
\begin{align*}
\mathcal{F}_0^{(3)}=&(2-\nu)\Lambda_C+\frac{\hat{\beta}_0^2d}{\lambda_Q
\varepsilon_4}+\sqrt{a}d\gamma_1^{\tau_1}+a\sqrt{a}d^2\gamma_2^{\tau_2} +a^2\sqrt{a}d^3\gamma_3^{\tau_3} ;\\[3mm]
\mathcal{F}_1^{(3)}=&
-\big[2-\varepsilon_4-\sqrt{a}(32\kappa c_1^2+dA_2)-2ac_2\big ]\lambda_Q+1_ka\\
&+at^{1_k}\vartheta^2\bigg [2\Lambda_{D^1{\bm b}}+(2-\nu)\Lambda_C
+d\bigg (\frac{\xi_1^{2-\alpha_1}}{\eps_1}
+\frac{\widehat\beta_0^2}{\lambda_Q\varepsilon_5}\bigg )\\
&\qquad\qquad\;
+\sqrt{a}\big (d^3\beta_2^{\rho_2}+d\hat{\beta}_1^{2-{\mu_1}}+d^3\hat{\beta}_2^{{{\mu_2}}} +(2d+1)\gamma_1^{2-\tau_1} \big )+a\sqrt{a}d^2\big (\beta_3^{\rho_3} d^2+\hat{\beta}^{{\mu_3}}_3 d^2+3\gamma_2^{{\tau_2}}\big )\bigg ];\\[3mm]
\mathcal{F}_2^{(3)}=&
at^{1_k}\vartheta^2\bigg [-\bigg (\frac{3}{2}-\varepsilon_5-\varepsilon_1A_1-\sqrt{a}d(2A_2+A_3)-64\sqrt{a}\kappa c_1^2-4ac_2\bigg )\lambda_Q+
2_ka\bigg ]\\
&+a^2t^{2_k}\vartheta^4\bigg [4\Lambda_{D^1{\bm b}}+(2-\nu)\Lambda_C+d\bigg (2\frac{\xi_1^{2-\alpha_1}}{\eps_2}
+\frac{\widehat\beta_0^2}{\lambda_Q\varepsilon_6}\bigg )\\
&\qquad\qquad\quad+\sqrt{a}\big (d\xi_2^{2-{\alpha_2}}
+d^3\xi_3^{\alpha_3}
+(3d^3+1)\beta_2^{2-{\rho_2}}
+2\hat{\beta}_1^{2-{\mu_1}}\\
&\qquad\qquad\qquad\qquad
+(3d^3+1)\hat{\beta}_2^{2-{{\mu_2}}}+(3d+2)\gamma_1^{\tau_1}+\gamma_2^{2-{\tau_2}}\big )\bigg ];\\[3mm]
\mathcal{F}_3^{(3)}=&
a^2t^{2_k}\vartheta^4\big [-\big (1-\varepsilon_6-144
\sqrt{a}\kappa c_1^2-
6ac_2-2\eps_2A_1-3\sqrt{a}d(A_2+A_3)\big )\lambda_Q+3_ka
\big ]\\
&+a^3t^{3_k}\vartheta^6\bigg [6\Lambda_{D^1{\bm b}}+(2-\nu)\Lambda_C+d\bigg (3\frac{\xi_1^{2-\alpha_1}}{\eps_3}
+\frac{\widehat\beta_0^2}{\lambda_Q\varepsilon_7}\bigg )\\
&\qquad\qquad\quad
+\sqrt{a}\big (3d\xi_2^{2-\alpha_2}+\xi_3^{2-\alpha_3}+ 3\beta_2^{\rho_2}+\beta_3^{2-\rho_3}+3d\hat{\beta}_1^{2-\mu_1}+3\hat{\beta}^{\mu_2}_2+\hat{\beta}^{2-\mu_3}_3\\
&\qquad\qquad\qquad\qquad\;
+3\gamma_1^{2-{\tau_1}}+3\gamma_2^{2-\tau_2}+\gamma_3^{2-\tau_3}\big )\bigg ];\\[3mm]
\mathcal{F}_4^{(3)}=&-a^3t^{3_k}\vartheta^6\big (1-
3\eps_3A_1-\varepsilon_7)\lambda_Q.
\end{align*}

To go further, we denote by ${\mathcal F}_{i,1}^{(3)}$ the first line of ${\mathcal F}_i^{(3)}$ ($i=1,2,3$) and set ${\mathcal F}_{i,2}^{(3)}=
{\mathcal F}_i^{3}-{\mathcal F}_{i,1}^{(3)}$ for the same values of $i$. Hence, to bound $G^{(3)}$ from above by a constant times $v^{(3)}$, it suffices to fix $a<\frac{1}{64}$, $\varepsilon_j$ ($j=1,\ldots,7)$ such that
\begin{equation}
{\mathcal F}_{1,1}^{(3)}\le 0,\qquad\;\, {\mathcal F}_{2,1}^{(3)}\le 0,\qquad\;\,{\mathcal F}_{3,1}^{(3)}\le 0,\qquad\;\, {\mathcal F}_4^{(3)}\le 0
\label{cond-fin-1-bis1}
\end{equation} 
% \begin{equation}
% \left\{
% \begin{array}{ll}
% \big[1-\sqrt{a}(32\kappa c_1^2+d\widehat\beta_0\lambda_0^{\sigma-1}+dA_2)-2ac_2\big ]\lambda_Q-1_ka\ge 0;\\[2mm]
% \big [1-\sqrt{a}\big (64\kappa c_1^2
% (1+d)\widehat\beta_0\lambda_0^{\sigma-1}+d(A_3+2A_2)\big )-4ac_2-\varepsilon_1A_1
% \big ]\lambda_Q-2_ka\ge 0;\\[2mm]
% \big [1-
% \sqrt{a}\big (144\kappa c_1^2+(1+d)\hat{\beta}_0\lambda_0^{\sigma-1}+3d(A_2+A_3)\big )-6ac_2-2\eps_2A_1\big ]\lambda_Q-3_ka\ge 0;\\[2mm]
% 1-\sqrt{a}\widehat\beta_0\lambda_0^{\sigma-1}-3\varepsilon_3A_1\ge 0
% \end{array}
% \right.
% \label{cond-fin-1-bis1}
% \end{equation}
on $(0,1]\times B(0,n)$ 
and, at the same time, ${\mathcal F}_0^{(3)}$, ${\mathcal F}_{1,2}^{(3)}$,  ${\mathcal F}_{2,2}^{(3)}$ and ${\mathcal F}_{3,2}^{(3)}$ are all bounded from above on $(0,1]\times B(0,n)$.
% \begin{equation}
% \left\{
% \begin{array}{ll}
% \displaystyle 
% (2-\textcolor{red}{\nu})\Lambda_C+\hat{\beta}_0^2d\lambda_Q^{2\sigma-1}+\sqrt{a}
% d\gamma_1^{\tau_1} +a\sqrt{a}d^2\gamma_2^{\tau_2} +a^2\sqrt{a}d^3\gamma_3^{\tau_3};\\[2mm]
% \displaystyle
% 2\Lambda_{D^1{\bm b}}+(2-\textcolor{red}{\nu})\Lambda_C
% +\frac{d\xi_1^{2-\alpha_1}}{\eps_1}
% +\sqrt{a}d\hat{\beta}_1^{2-{\mu_1}}+\sqrt{a}(1+2d)\gamma_1^{2-\tau_1}+\sqrt{a}d^3\big (\beta_2^{\rho_2}+\hat{\beta}_2^{\mu_2} \big )\\
% \displaystyle
% +a\sqrt{a}d^2\big (\beta_3^{\rho_3} d^2+\hat{\beta}^{{\mu_3}}_3 d^2+3\gamma_2^{{\tau_2}}\big );\\[3mm]
% \displaystyle
% 4\Lambda_{D^1{\bm b}}+(2-\textcolor{red}{\nu})\Lambda_C+2\frac{d\xi_1^{2-\alpha_1}}{\eps_2}+\sqrt{a}\big (d\xi_2^{2-{\alpha_2}}+\beta_2^{2-{\rho_2}}
% +2\hat{\beta}_1^{2-{\mu_1}}
% +\hat{\beta}_2^{2-{{\mu_2}}}+2\gamma_1^{\tau_1}+\gamma_2^{2-{\tau_2}}\big )\\
% \displaystyle
% +\sqrt{a}d^2\big (d\xi_3^{\alpha_3} +3d\beta_2^{2-\rho_2}
% +3d\hat{\beta}^{2-\mu_2}_2 +3\gamma_1^{{\tau_1}}\big );\\[3mm]
% \displaystyle
% 6\Lambda_{D^1{\bm b}}+\textcolor{red}{(2-\nu)}\Lambda_C+3\frac{d\xi_1^{2-\alpha_1}}{\eps_3}\\
% +\sqrt{a}\big (3d\xi_2^{2-\alpha_2}+\xi_3^{2-\alpha_3}+ 3\beta_2^{\rho_2}+\beta_3^{2-\rho_3}+3d\hat{\beta}_1^{2-\mu_1}
% +3\hat{\beta}^{\mu_2}_2
% +\hat{\beta}^{2-\mu_3}_3+3\gamma_1^{2-{\tau_1}}+3\gamma_2^{2-\tau_2}+\gamma_3^{2-\tau_3}\big )
% \end{array}
% \right.
% \label{cond-fin-21}
% \end{equation}
% bounded over $(0,1]\times B(0,n)$. 

To make the terms in \eqref{cond-fin-1-bis1} nonpositive, it suffices to choose $\varepsilon_j$ such that
\begin{eqnarray*}
2-\varepsilon_4>0,\qquad\;\,\frac{3}{2}-\varepsilon_5-\varepsilon_1A_1>0,\qquad\;\,1-\varepsilon_6-2\varepsilon_2A_1>0,\qquad\;\,1-3\varepsilon_3A_1-\varepsilon_7=0
\end{eqnarray*}
and, then, choose $a$ sufficiently small.  
We fix $\varepsilon_4=\frac{d}{M_0}$ and 
$\varepsilon_5$, $\varepsilon_6$, $\varepsilon_7$ positive and subject to the conditions 
$\varepsilon_{4+j}+j\varepsilon_jA_1=\frac{d}{2jM_j}$ $(j=1,2)$ and $\varepsilon_7+3\varepsilon_3A_1=1$ in order to minimize the functions
$\zeta_j(\varepsilon_j,\varepsilon_{4+j})=j\frac{\xi^{2-\alpha_1}}{\varepsilon_j}+\frac{\widehat\beta_0^2}{\lambda_Q\varepsilon_{4+j}}$ $(j=1,2)$ and
$\zeta_3(\varepsilon_3,\varepsilon_7)=3\frac{\xi^{2-\alpha_1}}{\varepsilon_3}+\frac{\widehat\beta_0^2}{\lambda_Q\varepsilon_7}$. 
The minimum of the functions $\zeta_j$ ($j=1,2$) subject to the previous condition is
$\frac{2jM_j}{d}\left (j\sqrt{A_1}\xi_1^{1-\frac{\alpha_1}{2}}+\widehat\beta_0\lambda_Q^{-\frac{1}{2}}\right )^2$ and the minimum of $\zeta_3$ with the constraint $1-3\varepsilon_3A_1-\varepsilon_7=0$ is
$\left (3\sqrt{A_1}\xi_1^{1-\frac{\alpha_1}{2}}+\widehat\beta_0\lambda_Q^{-\frac{1}{2}}\right )^2$.
Due to Hypothesis \ref{hyp-derivata-3}(ii), we conclude that the functions $\mathcal{F}_0^{(3)}$, $\mathcal{F}_{1,2}^{(3)}$, $\mathcal{F}_{2,2}^{(3)}$ and
$\mathcal{F}_{3,2}^{(3)}$ are bounded from above in $(0,1]\times B(0,n)$ by a constant $c_3$, independent of $n$, up to taking a smaller value of $a$ if needed.

Summing up, we have proved that the function $v^{(3)}$ is a solution of the problem
\begin{eqnarray*}
\left\{
\begin{array}{ll}
(D_tw)(t,x)-(\A_{\nu}w)(t,x)\le c_3w(t,x), & (t,x)\in (0,1]\times B(0,n),\\[1mm]
w(t,x)=0, & (t,x)\in (0,1]\times\partial B(0,n),\\[1mm]
w(0,x)=\sum_{j=0}^k|D^j\f(x)|^2, & x\in B(0,n)
\end{array}
\right.
\end{eqnarray*}
and the maximum principle allows us to conclude that
$v^{(3)}(t,\cdot)\le e^{c_3t}S_{\nu}(t)\sum_{j=0}^k|D^j\f|^2$ in $B(0,n)$. 
Letting $n$ approach $\infty$, we get 
\begin{equation}
\sum_{h=0}^3t^{h_k}|D^h\uu(t,x)|^2 \le e^{c_3t}S_{\nu}(t)\sum_{j=0}^k|D^j\f|^2,\qquad\;\,t\in (0,1],
\label{alain}
\end{equation}
for any $k \in \{0,1,2,3\}$, whence the claim in this case.
For a general $\f\in C_b(\R^d;\R^m)$, we consider a bounded (with respect to the norm of $C^k_b(\R^d;\R^m)$) sequence $(\f_n)_{n\in\N}\subset C^3_c(\R^d)$, which converges to $\f$ in $C^k(\Omega;\R^m)$ for every bounded open set $\Omega\subset\R^d$. By Remark \ref{rmk-1}, the sequence $(\bm{T}(t
)\f_n)_{n\in\N}$ converges to $\bm{T}(t
)\f$ in $C^3(\Omega;\R^m)$ for every $t\in (0,\infty)$ and every $\Omega$ as above. Moreover, Proposition \ref{prop-2.5} guarantees that, for every $t\in (0,\infty)$, the sequence 
$(S_{\nu}(t)\sum_{j=0}^k|D^j\f_n|^2)_{n\in\N}$ converges to $S_{\nu}(t)\sum_{j=0}^k|D^j\f|^2$ locally uniformly in $\R^d$ as $n$ tends to $\infty$. Writing \eqref{alain}, with $\f$ being replaced by $\f_n$ and letting $n$ approach $\infty$, estimate \eqref{alain} follows in its full generality.

To prove \eqref{pointwise}, under Hypotheses \ref{hyp-derivata-2}, with $k,\ell\in\{0,1,2\}$ we consider the function $v_{k,n}^{(2)}$ defined by
\begin{align*}
v_{k,n}^{(2)}(t,x):=&|\uu_n(t,x)|^2+at^{1_k}(\vartheta_n(x))^2|D^1\uu_n(t,x)|^2+a^2t^{2_k}(\vartheta_n(x))^4|D^2\uu_n(t,x)|^2
\end{align*}
for every $(t,x)\in [0,\infty)\times B(0,n)$.
It turns out that 
$D_tv_{k,n}^{(2)}-{\mathcal A}_{\nu}v_{k,n}^{(2)}=\sum_{i=0}^2 g_{i,i}^{(3)}+
g_{3,3}^{(2)}+
\sum_{i=0}^1g_{i,i+1}^{(3)}+g_{0,2}^{(3)}+g_{2,3}^{(2)}=:G^{(2)}$, where $g_{3,3}^{(2)}(t,\cdot)= -2a^2t^{2_k}\vartheta^4\sum_{h,r=1}^d\mathfrak{q}_0(\partial_{hr}\uu)$ and
$g_{2,3}^{(2)}$ consists of the first three terms in $g_{2,3}^{(3)}$.
Using \eqref{g23-1}-\eqref{g23-3}, we get
\begin{align*}
g_{2,3}^{(2)}\leq& 8a^2\sqrt{a}t^{2_k}\vartheta^4\sum_{h,r=1}^d\mathfrak{q}_0(\partial_{hr}\uu)
+\bigg(2a^2\vartheta^2\frac{d\xi_1^{2-\alpha_1}}{\eps_2}+a^2\vartheta^2\frac{\widehat\beta_0^2d}{\lambda_Q\varepsilon_6}+64a\sqrt{a}\kappa c_1^2\lambda_Q\bigg)t^{2_k}\vartheta^2|D^2\uu|^2\\
&+a^2\big(2\eps_2A_1+\varepsilon_6\big)t^{2_k}\vartheta^4\lambda_Q|D^3\uu|^2.
\end{align*}

Taking also \eqref{estim-g11}-\eqref{estim-g22}, \eqref{estim-g01}, \eqref{estim-g12}, \eqref{estim-g02} into account and fixing $a$ such that $8\sqrt{a}<1$, we conclude that $G^{(2)}\le \sum_{k=0}^3\mathcal{F}_k^{(2)}|D^k\uu|^2$, where
\begin{align*}
\mathcal{F}_0^{(2)}=&(2-\nu)\Lambda_C+\frac{\hat{\beta}_0^2d}{\lambda_Q\varepsilon_4}
+\sqrt{a}dt^{1_k}(\gamma_1^{\tau_1}+ad\gamma_2^{\tau_2});\\[2mm]
\mathcal{F}_1^{(2)}=&
-\big[2-\varepsilon_4-\sqrt{a}(32\kappa c_1^2+dA_2) -2ac_2\big ]\lambda_Q
+1_ka\\
&+at^{1_k}\vartheta^2\bigg[2\Lambda_{D^1{\bm b}}+(2-\nu)\Lambda_C
\!+\!d\bigg (\frac{\xi_1^{2-\alpha_1}}{\eps_1}+\frac{\widehat\beta_0^2}{\lambda_Q\eps_5}\bigg )
\!+\!\sqrt{a}(d^3\beta_2^{\rho_2}+d\hat{\beta}_1^{2-{\mu_1}}\\
&\phantom{+at^{1_k}\vartheta^2\bigg[\;\,}
+d^3\hat{\beta}_2^{{{\mu_2}}} +(2d+1)\gamma_1^{2-\tau_1})\bigg ];\\[2mm]
\mathcal{F}_2^{(2)}=&
-at^{1_k}\vartheta^2\bigg \{\bigg [\frac{3}{2}-\varepsilon_5-\varepsilon_1A_1-\sqrt{a}(64\kappa c_1^2+d(2A_2+A_3))-4ac_2\bigg ]\lambda_Q-
2_ka
\bigg \}\\
&+a^2t^{2_k}\vartheta^4\bigg [4\Lambda_{D^1{\bm b}}+(2-\nu)\Lambda_C+d\bigg (2\frac{\xi_1^{2-\alpha_1}}{\eps_2}
+\frac{\widehat\beta_0^2}{\lambda_Q\varepsilon_6}\bigg )\\
&\qquad\qquad\quad\;\,+\sqrt{a}\big (d\xi_2^{2-{\alpha_2}}+\beta_2^{2-{\rho_2}}
+2\hat{\beta}_1^{2-{\mu_1}}
+\hat{\beta}_2^{2-{{\mu_2}}}+2\gamma_1^{\tau_1}+\gamma_2^{2-{\tau_2}}\big )\bigg ];\\[2mm]
\mathcal{F}_3^{(2)}=&
-a^2t^{2_k}\vartheta^4\big (1-\varepsilon_6-
2\eps_2A_1\big )\lambda_Q.
\end{align*}

Denoting by ${\mathcal F}_{1,1}^{(2)}$ and ${\mathcal F}_{2,1}^{(2)}$ the first line of ${\mathcal F}_1^{(2)}$
and ${\mathcal F}_2^{(2)}$, respectively, and
then setting ${\mathcal F}_{1,2}^{(2)}={\mathcal F}_1^{(2)}-{\mathcal F}_{1,1}^{(2)}$ and ${\mathcal F}_{2,2}^{(2)}={\mathcal F}_2^{(2)}-{\mathcal F}_{2,1}^{(2)}$, it turns out that the conditions which we need to assume are
\begin{equation}
\mathcal{F}_{1,1}^{(2)}\le 0,\qquad\;\,
\mathcal{F}_{2,1}^{(2)}\le 0,\qquad\;\,
\mathcal{F}_3^{(2)}\le 0
\label{cond-fin-1-bis}
\end{equation}
% The conditions to assume are
% \begin{equation}
% \left\{
% \begin{array}{ll}
% \big[1-\sqrt{a}(32
% \kappa c_1^2+d\widehat\beta_0\lambda_0^{\sigma-1}+dA_2)-2ac_2\big ]\lambda_Q
% -1_ka\ge 0,\\[2mm]
% \big [1
% -\sqrt{a}\big (64\kappa c_1^2
% +(1+d)\widehat\beta_0\lambda_0^{\sigma-1}+d(A_3+2A_2)\big )-4ac_2-\varepsilon_1A_1
% \big ]\lambda_Q-
% 2_ka\ge 0;\\[2mm]
% 1-\sqrt{a}\hat{\beta}_0\lambda_0^{\sigma-1}+2\eps_2A_1 \ge 0;
% \end{array}
% \right.
% \label{cond-fin-1-bis}
% \end{equation}
in $(0,1]\times B(0,n)$ 
and ${\mathcal F}^{(2)}_0$, $\mathcal{F}_{1,2}^{(2)}$ and
$\mathcal{F}_{2,2}^{(2)}$
% \begin{equation}
% \left\{
% \begin{array}{ll}
% \displaystyle 
% \textcolor{red}{(2-\nu)}\Lambda_C
% +\hat{\beta}_0^2d\lambda_Q^{2\sigma-1}+\sqrt{a}d\gamma_1^{\tau_1} +a\sqrt{a}d^2\gamma_2^{\tau_2};\\[2mm]
% \displaystyle
% 2\Lambda_{D^1{\bm b}}+\textcolor{red}{(2-\nu)}\Lambda_C
% +\frac{d\xi_1^{2-\alpha_1}}{\eps_1}
% +\sqrt{a}d\hat{\beta}_1^{2-{\mu_1}}+\sqrt{a}\gamma_1^{2-\tau_1}+\sqrt{a}d\big (\beta_2^{\rho_2}d^2+\hat{\beta}_2^{{{\mu_2}}} d^2+2\gamma_1^{2-\tau_1}\big );\\[3mm]
% \displaystyle
% 4\Lambda_{D^1{\bm b}}+\textcolor{red}{(2-\nu)}\Lambda_C+2\frac{d\xi_1^{2-\alpha_1}}{\eps_2}+\sqrt{a}\big (d\xi_2^{2-{\alpha_2}}+\beta_2^{2-{\rho_2}}
% +2\hat{\beta}_1^{2-{\mu_1}}
% +\hat{\beta}_2^{2-{{\mu_2}}}+2\gamma_1^{\tau_1}+\gamma_2^{2-{\tau_2}}\big )
% \end{array}
% \right.
% \label{cond-fin-2}
% \end{equation}
bounded over $(0,1]\times B(0,n)$. 
If we fix
$\varepsilon_4$, $\varepsilon_5$ and $\varepsilon_6$ positive and such that $\varepsilon_4=\frac{d}{M_0}$, $\varepsilon_5+
\varepsilon_1A_1=\frac{d}{2M_1}$ and $\varepsilon_6+2\varepsilon_2A_1=1$, then we can fix $a$ sufficiently small such that
the conditions in \eqref{cond-fin-1-bis} are satisfied.
and the functions ${\mathcal F}^{(2)}_0$, $\mathcal{F}_{1,2}^{(2)}$ and
$\mathcal{F}_{2,2}^{(2)}$ are all bounded from above in $(0,1]\times B(0,n)$ by a constant independent of $n$, due to Hypothesis \ref{hyp-derivata-2}(iii), up to taking a smaller value of $a$, if needed. Now, we can complete the proof as in the case $k,\ell\in\{0,1,2,3\}$.

Finally, to prove \eqref{aim}
under Hypotheses \ref{hyp-derivata-1} (i.e., in the case when $k,\ell\in\{0,1\}$), we need to apply the arguments in the first part of the proof to the function
$v_{1,n}^{(1)}$, defined by
\begin{align*}
v_{k,n}^{(1)}(t,x)=&|\uu_n(t,x)|^2+at^{1_k}(\vartheta_n(x))^2|D^1\uu_n(t,x)|^2,\qquad\;\,
(t,x)\in [0,\infty)\times B(0,n),
\end{align*}
In this case, 
$D_tv_{k,n}^{(1)}-\A_{\nu}v_{k,n}^{(1)}=
g_{0,0}^{(3)}+g_{1,1}^{(3)}+g_{2,2}^{(1)}+g_{0,1}^{(3)}+g_{1,2}^{(1)}=:G^{(1)}$, where  $g_{1,2}^{(1)}$ consists of the first three terms of $g_{1,2}^{(3)}$ and  $g_{2,2}^{(1)}=-2at^{1_k}\vartheta^2\sum_{h=1}^d\mathfrak{q}_0(\partial_h\uu)$.
Using \eqref{num-0} and \eqref{num}, we can infer that
\begin{align*}
g_{1,2}\le &4a\sqrt{a}t^{1_k}\vartheta^2\sum_{h=1}^d\mathfrak{q}_0(\partial_h\uu)+t^{1_k}\bigg [a\vartheta^2\bigg (\frac{d\xi_1^{2-\alpha_1}}{\eps_1}+\frac{\widehat\beta_0^2d}{\eps_5}\bigg )+32\sqrt{a}\kappa c_1^2\lambda_Q\bigg ]|D^1\uu|^2\\
&+at^{1_k}\vartheta^2(\eps_1A_1+\varepsilon_5)\lambda_Q|D^2\uu|^2.
\end{align*}

Hence, this estimate together with \eqref{estim-g11}, \eqref{estim-g01} implies that $G^{(1)}
\le\sum_{k=0}^2\mathcal{F}_k^{(1)}|D^k\uu|^2$ for every $a\in\left (0,\frac{1}{16}\right )$, where
\begin{align*}
\mathcal{F}_0^{(1)}=&(2-\nu)\Lambda_C+\frac{\hat{\beta}_0^2d}{\lambda_Q\varepsilon_4} +\sqrt{a}d\gamma_1^{\tau_1};\\[2mm]
\mathcal{F}_1^{(1)}=&
-\big[2-\varepsilon_4-\sqrt{a}(32\kappa c_1^2+
A_2d)
-2ac_2\big ]\lambda_Q
+1_ka\\
&+at^{1_k}\vartheta^2\bigg[2\Lambda_{D^1{\bm b}}+(2-\nu)\Lambda_C
+d\bigg (\frac{\xi_1^{2-\alpha_1}}{\eps_1}
+\frac{\widehat\beta_0^2}{\lambda_Q\varepsilon_5}\bigg )
+\sqrt{a}(d\hat{\beta}_1^{2-{\mu_1}}+\gamma_1^{2-\tau_1})
\bigg ];\\[2mm]
\mathcal{F}_2^{(1)}=&
-at^{1_k}\vartheta^2\big (1-\varepsilon_5-
\eps_1A_1\big )\lambda_Q.
\end{align*}

Due to Hypothesis \ref{hyp-derivata-1}(ii), we can choose $\varepsilon_4=\frac{d}{M_0}$ and
minimize the function 
$\tilde{\zeta}(\eps_1,\eps_5)=\frac{\xi_1^{2-\alpha_1}}{\varepsilon_1}+\frac{\widehat\beta_0^2}{\lambda_Q\varepsilon_5}$,
subject to the condition $\varepsilon_5+A_1\varepsilon_1=1$ in order to make the functions $\mathcal{F}_0^{(1)}$, $\mathcal{F}_1^{(1)}$ and
$\mathcal{F}_2^{(1)}$ bounded from above in $(0,1]\times B(0,n)$ by a constant, independent of $n$, up to taking a smaller value of $a$, if needed.
Now, the proof can be completed as in the case $k,\ell\in\{0,1,2,3\}$.
\end{proof}

Estimate \eqref{pointwise} and the $L^\infty$-contractivity of the semigroup $\{S_\nu(t)\}_{t\ge 0}$ imply the following uniform estimates for ${\bm T}(t)\f$ and its derivatives 
when $\f$ is smooth enough.

\begin{cor}
If Hypotheses $\ref{hyp-derivata-1}$ are satisfied then, for any $\f \in C_b^k(\Rd;\Rm)$, $\bm{T}(t)\f$ belongs to $C_b^{\ell}(\Rd;\Rm)$ for any $k,\ell\in \{0,1\}$ with $k \le \ell$. Moreover for any $\varepsilon>0$ and $k,\ell$ as above
\begin{equation}\label{aim}
\|\bm{T}(t)\f\|_{C_b^{\ell}(\Rd;\Rm)}\le c_{\varepsilon}e^{(H_\nu+\eps)t}t^{-\frac{\ell-k}{2}}\|\f\|_{C_b^{k}(\Rd;\Rm)}, \qquad\;\, t\in (0,\infty),
\end{equation}
where $H_\nu$ is the constants appearing in \eqref{norm_infty} and $c_{\varepsilon}$ is a positive constant, independent of $t$ and $\f$, that blows up as $\varepsilon$ tends to $0$.
Further, if Hypotheses $\ref{hyp-derivata-2}$ $($resp. Hypotheses $\ref{hyp-derivata-3})$ are satisfied then $\bm{T}(t)\f$ belongs to $C_b^{\ell}(\Rd;\Rm)$ with $\ell,k \in \{0,1,2\}$ $($resp. $\ell,k \in \{0,1,2,3\})$ with $k\leq\ell$ and estimate \eqref{aim} is satisfied for such values of $\ell$ and $k$.
\end{cor}

\section{Optimal Schauder estimates}
\label{sect-4}

The following interpolation result will be useful throughout this section.

\begin{thm}\label{thm-1-2}
Assume that Hypotheses $\ref{hyp-derivata-1}$ are satisfied. For any $t>0$, ${\bm T}(t)$ is a bounded linear operator from $\mathcal{C}^{\theta_1}_b(\Rd;\Rm)$ to $\mathcal{C}^{\theta_2}_b(\Rd;\Rm)$ for any $0 \le \theta_1\le \theta_2\le 1$ and for any $\varepsilon>0$ 
\begin{equation}
\label{teta1teta2infty}
\|{\bm T}(t)\f\|_{\mathcal{C}^{\theta_2}_b(\Rd;\Rm)}\le c_{\varepsilon}e^{(H_\nu+\varepsilon)t}t^{-\frac{\theta_2-\theta_1}{2}}\|\f\|_{\mathcal{C}^{\theta_1}_b(\Rd;\Rm)},
\qquad\;\,t\in (0,\infty),
\end{equation}
where $H_\nu$ is the constant appearing in \eqref{norm_infty} and $c_{\varepsilon}$ is a positive constant blowing up as $\varepsilon$ tends to $0$. 

 If Hypotheses $\ref{hyp-derivata-1}$ are replaced by Hypotheses $\ref{hyp-derivata-2}$ $($resp. Hypotheses $\ref{hyp-derivata-3})$, then for any $t>0$, ${\bf T}(t)$ is a linear and bounded operator from $\mathcal{C}^{\theta_1}_b(\Rd;\Rm)$ to $\mathcal{C}^{\theta_2}_b(\Rd;\Rm)$ for any $0 \le \theta_1\le \theta_2\le 2$ $($resp. for any $0 \le \theta_1\le \theta_2\le 3)$  and estimate \eqref{teta1teta2infty} holds true for such $\theta_1$ and $\theta_2$.

Moreover, if Hypotheses $\ref{hyp-derivata-1}$ $($Hypotheses $\ref{hyp-derivata-2}$ or Hypotheses $\ref{hyp-derivata-3}$, respectively$)$ hold true, then for $\ell=1$ $(\ell\in\{1,2\}$ or $\ell\in\{1,2,3\}$, respectively$)$ and $t>0$, ${\bm T}(t)$ is a bounded linear operator from $\mathcal{C}^{\alpha}(\Rd;\Rm)$ to $C^{\ell}_b(\Rd;\Rm)$, for any $\alpha\in(0,\ell)$ and it satisfies for any $\eps>0$
 \begin{equation}\label{alphaelle}
\|{\bm T}(t)\f\|_{C_b^{\ell}(\Rd;\Rm)}\le c_\eps e^{(H_\nu+\varepsilon)t}t^{-\frac{\ell-\alpha}{2}}\|\f\|_{\mathcal{C}^{\alpha}(\Rd;\Rm)},
\end{equation}
where $H_\nu$ is the constant appearing in \eqref{norm_infty} and $c_{\varepsilon}$ is a positive constant blowing up as $\varepsilon$ tends to $0$. 
\end{thm}

\begin{proof}
We will just prove the theorem under Hypotheses \ref{hyp-derivata-3}, the other cases follow in a similar way. 
Fix $t\in (0,\infty)$, $\theta\in(0,1)$ and apply Theorem \ref{thm_interpolation}, Theorem \ref{thm_operator_interpolation} (with $X_0=Y_0=X_1=C_b(\R^d;\R^m)$, $Y_1=C_b^3(\R^d;\Rm)$ and $R=\bm{T}(t)$) and estimate \eqref{aim} to obtain that $\bm{T}(t)$ is a bounded operator from $C_b(\R^d;\R^m)$ to $\mathcal{C}^{3\theta}_b(\R^d;\R^m)$ and
\begin{align}\label{roger1}
\|\bm{T}(t)\|_{\mathcal{L}(C_b(\R^d;\R^m);\mathcal{C}^{3\theta}_b(\R^d;\R^m))}\leq c_{\varepsilon}e^{(H_\nu+\eps)t}t^{-\frac{3\theta}{2}},\qquad\;\,t\in (0,
\infty).
\end{align}
For the same $\theta$, apply Theorem \ref{thm_interpolation}, Theorem \ref{thm_operator_interpolation} (with $X_0=Y_0=C_b(\R^d;\R^m)$, $X_1=Y_1=C_b^3(\R^d;\Rm)$ and $R=\bm{T}(t)$) and estimate \eqref{aim} to obtain that $\bm{T}(t)$ is a bounded operator from $\mathcal{C}^{3\theta}_b(\R^d;\R^m)$ to $\mathcal{C}^{3\theta}_b(\R^d;\R^m)$ and
\begin{align}\label{roger2}
\|\bm{T}(t)\|_{\mathcal{L}(\mathcal{C}^{3\theta}_b(\R^d;\R^m))}\leq c_{\varepsilon}e^{(H_\nu+\eps)t},\qquad\;\,t\in (0,\infty).
\end{align}
Finally, fix $\alpha\in(0,1)$ and apply Theorem \ref{thm_interpolation}, Theorem \ref{thm_operator_interpolation} (with $X_0=C_b(\R^d;\R^m)$, $X_1=Y_0=Y_1=\mathcal{C}^{3\theta}_b(\R^d;\Rm)$ and $R=\bm{T}(t)$), \eqref{roger1} and \eqref{roger2} to obtain that $\bm{T}(t)$ is a bounded operator from $\mathcal{C}^{3\alpha\theta}_b(\R^d;\R^m)$ to $\mathcal{C}^{3\theta}_b(\R^d;\R^m)$ and
\begin{align}\label{arcangel}
\|\bm{T}(t)\|_{\mathcal{L}(\mathcal{C}^{3\alpha\theta}_b(\R^d;\R^m);\mathcal{C}^{3\theta}_b(\R^d;\R^m))}\leq c_\eps e^{(H_\nu+\eps)t}t^{-\frac{3\theta(1-\alpha)}{2}},\qquad\;\,t\in (0,\infty).
\end{align}
The claim now follows taking $\theta=\theta_2/3$ and $\alpha=\theta_1/\theta_2$ in \eqref{arcangel}, when $\theta_1\neq\theta_2$. The case $\theta_1=\theta_2$ follows arguing as in the proof of \eqref{roger2}. 

The proof of \eqref{alphaelle} follows by the same arguments simply applying Theorem \ref{thm_interpolation}, Theorem \ref{thm_operator_interpolation} (with $X_0=C_b(\R^d;\R^m)$, $X_1=Y_0=Y_1=C^{\ell}_b(\R^d;\Rm)$, $\theta=\alpha/\ell$ and $R=\bm{T}(t)$) and \eqref{aim}.
\end{proof}

\subsection{The elliptic case}

As the scalar case shows, the semigroups associated in the space of bounded and continuous functions to elliptic operators with unbounded coefficients in general are not strongly continuous. Thus, we cannot consider the infinitesimal generator of $\{\bm{T}(t)\}_{t\ge 0}$ but only its weak generator ${\bm A}$ defined throughout the resolvent operator by 
\begin{align}\label{ulambda}
   \uu_{\lambda}(x):=(R(\lambda, {\bm A})\f)(x)=\int_0^{\infty}e^{-\lambda s}(\bm{T}(s)\bm{f})(x)ds,\qquad\;\, x \in \Rd,
\end{align}
for any $\f \in C_b(\Rd;\Rm)$ and any $\lambda>H_\nu$ , thanks to estimate \eqref{aim}. 
Moreover,
\begin{equation}
\|\uu_{\lambda}\|_{C_b(\R^d;\R^m)}\le \frac{1}{\lambda-H_\nu}\|\f\|_{C_b(\R^d;\R^m)},\qquad\;\,\lambda>H_\nu.
\label{stima-norma-sup}
\end{equation}

\begin{thm}\label{thm_Schauder_elliptic}
Assume Hypotheses $\ref{hyp-derivata-3}$. For every $\lambda>H_\nu$, $\alpha\in[0,1)$ and $\bm{f}\in \mathcal{C}^{\alpha}_b(\R^d;\R^m)$, the function $\bm{u}_\lambda$ defined in \eqref{ulambda} belongs to $\mathcal{C}^{2+\alpha}_b(\R^d;\R^m)$ and there exists a positive constant $c$, independent of $\bm{f}$, such that
\begin{align}\label{zyg_ell}
\|\bm{u}_\lambda\|_{\mathcal{C}^{2+\alpha}_b(\R^d;\R^m)}\leq c\|\bm{f}\|_{\mathcal{C}^{\alpha}_b(\R^d;\R^m)}.
\end{align} 
%If, in addition $\bm{f}\in C_b^\alpha(\R^d;\R^m)$, for some $\alpha\in(0,1)$, then $\uu_\lambda$ belongs to $C_b^{2+\alpha}(\R^d;\R^m)$ and there exists a positive constant $C$ depending only on $\alpha$ and $\lambda$ such that
%\begin{align}
%    \|\bm{u}_\lambda\|_{C_b^{2+\alpha}(\R^d;\R^m)}\leq C \|\bm{f}\|_{C_b^{\alpha}(\R^d;\R^m)}.\label{Schauder_ell}
%\end{align}
In particular %, if \textcolor{red}{$\alpha\in (0,1)$}, then 
$\uu_{\lambda}$ is the unique classical solution to the equation $\lambda\uu-\bm{\A}\uu=\f$.
\end{thm}

\begin{proof}
In order to simplify the notation and improve readability, throughout the proof the letter $K$ will represent a positive constant, independent of $\bm{f}$, $t$ and of $x$, that may change from line to line. Since the norms introduced in Definition \ref{defn_Hold_Zyg} are different if $\alpha=0$ or $\alpha\in(0,1)$ we will split the proof into two parts.

\noindent {\boldmath$(\alpha=0).$} We start by showing that if $\lambda>H_\nu$ and $\bm{f}\in C_b(\R^d;\R^m)$, then $\bm{u}_\lambda$ belongs to $C^1_b(\R^d;\R^m)$ and 
\begin{align*}
    (D^1\bm{u}_{\lambda})(x)=\int_0^{\infty}e^{-\lambda s}(D^1\bm{T}(s)\bm{f})(x)ds,\qquad\;\, x \in \Rd.
\end{align*}
Indeed, the continuity of $\bm{u}_\lambda$ is trivial to prove. The differentiability follows from the Lebesgue dominated convergence theorem, and observing that, by \eqref{aim} (with $\eps=1$), it holds
\begin{align}
&|(\bm{T}(t)\bm{f})(x+h)-(\bm{T}(t)\bm{f})(x)- (D^1(\bm{T}(t)\bm{f})(x))h|\notag\\
&\qquad\qquad\qquad=\bigg (\sum_{j=1}^m\bigg (\sum_{i=1}^d\int_0^1\big [(\partial_i\bm{T}(t)\bm{f})_j(x+\sigma h)-(\partial_i\bm{T}(t)\bm{f})_j(x)\big ]h_id\sigma\bigg )^2\bigg )^{\frac{1}{2}}\notag\\
&\qquad\qquad\qquad\le K e^{(H_\nu+1)t}t^{-\frac{1}{2}}|h|\|\bm{f}\|_{C_b(\R^d;\R^m)}.\label{zyg_C1}
\end{align}
Regarding estimate \eqref{zyg_ell}, we start by observing that by \eqref{aim} (with $\eps=\frac{\lambda-H_\nu}{2}$) we get
\begin{align}
\|D^1\bm{u}_\lambda\|_{C_b(\R^d;\R^{md})} 
% &\leq \int_0^{\infty}e^{-\lambda s} \|D^{1}(\bm{T}(s)\bm{f})\|_{C_b(\R^d;\R^{md})}ds\notag\\
\leq K\left(\int_0^{\infty}e^{-\frac{\lambda-H_\nu}{2} s}s^{-\frac{1}{2}} ds\right)\|\bm{f}\|_{C_b(\R^d;\R^m)}\leq K\|\bm{f}\|_{C_b(\R^d;\R^m)}.\label{zyg_deriv1}
\end{align}
%By \eqref{seminorm_bounds} it holds that for any $\alpha\in(0,1)$ the function $\bm{u}_\lambda$ belongs to $C_b^{1+\alpha}(\R^d;\R^m)$. 

Now, in order to estimate the Zygmund seminorm of $D^1\uu_\lambda$, we write $D^1\uu_\lambda=\bm{a}_h+\bm{b}_h$ with
\begin{align*}
    \bm{a}_h(x):=\int_0^{|h|^2}e^{-\lambda s}(D^{1}\bm{T}(s)\bm{f})(x)ds, && \bm{b}_h(x):=\int_{|h|^2}^{\infty}e^{-\lambda s}(D^{1}\bm{T}(s)\bm{f})(x)ds,
\end{align*}
for any $x,h\in\R^d$.
Observing that, for any $x, h \in \Rd$
\begin{align*}
&|(D^1\uu_{\lambda})(x+2h) -2(D^1\uu_{\lambda})(x+h)+(D^1\uu_{\lambda})(x)|\\
 &\qquad\qquad\qquad\le|\bm{a}_h(x+2h) -2\bm{a}_h(x+h)+\bm{a}_h(x)|+|\bm{b}_h(x+2h) -2\bm{b}_h(x+h)+\bm{b}_h(x)|,
\end{align*}
we can estimate the Zygmund seminorms of $\bm{a}_h$ and $\bm{b}_h$, respectively. 

Since
$|(D^{1}\bm{T}(s)\bm{f})(x+2h)-2(D^{1}\bm{T}(s)\bm{f})(x+h)+(D^{1}\bm{T}(s)\bm{f})(x)|\le K\|D^1\bm{T}(s)\f\|_{C_b(\R^d;\R^{md})}$ for every $s\in (0,\infty)$ and $x,h\in\R^d$, by \eqref{aim} (with $\varepsilon=\lambda-H_\nu$) it holds that
\begin{align}
|\bm{a}_h(x+2h)-2\bm{a}_h(x+h)+\bm{a}_h(x)|
% \leq & \int_0^{|h|^2}e^{-\lambda s}|(D^{1}\bm{T}(s)\bm{f})(x+2h)-2(D^{1}\bm{T}(s)\bm{f})(x+h)+(D^{1}\bm{T}(s)\bm{f})(x)|ds\notag\\
% \leq & K\bigg (\int_0^{|h|^2}e^{-\frac{\lambda-H}{2} s}s^{-\frac{1}{2}}ds\bigg)\|\bm{f}\|_{C_b(\R^d;\R^m)}\notag\\
\leq & K\bigg (\int_0^{|h|^2}s^{-\frac{1}{2}}ds\bigg )\|\bm{f}\|_{C_b(\R^d;\R^m)}=K|h|\|\bm{f}\|_{C_b(\R^d;\R^m)}.
\label{zyg_a1_ell}
\end{align}
%while if $|h|\geq 1$, using again \eqref{aim}, we get
%\begin{align}
%|\bm{a}(x+2h) &-2\bm{a}(x+h)+\bm{a}(x)|\notag\\
%&\leq \int_0^{|h|^2}e^{-\lambda s}|D^{1}(\bm{T}(s)\bm{f})(x+2h)-2D^{1}(\bm{T}(s)\bm{f})(x+h)+D^{1}(\bm{T}(s)\bm{f})(x)|ds\notag\\
%&\leq c\left(\int_0^{|h|^2}e^{-(\lambda-H) s}\max\{s^{-1/2},1\}ds\right)\|\bm{f}\|_{C_b(\R^d;\R^m)}\notag\\
%&\leq c\left(\int_0^1s^{-1/2}ds+\int_0^{\infty}e^{-(\lambda-H)s}ds\right)\|\bm{f}\|_{C_b(\R^d;\R^m)}\leq c|h|\|\bm{f}\|_{C_b(\R^d;\R^m)}.\label{zyg_a2_ell}
%\end{align}
To provide an estimate for $\bm{b}_h$ we need the following preliminary observation:
\begin{align}
((D^{1}\bm{T}(s)\bm{f})(y+2h)-(D^{1}\bm{T}(s)\bm{f})(y+h))_{ij} &=\sum_{r=1}^d\int_0^1(\partial_{ir}\bm{T}(s)\bm{f})_j(y+\sigma h)h_r d\sigma
% \\
% [((D^{1}\bm{T}(s)\bm{f})(x+h)-(D^{1}\bm{T}(s)\bm{f})(x))k]_j &=\sum_{i,r=1}^d\int_0^1(\partial_{ir}\bm{T}(s)\bm{f})_j(x+\sigma h)k_ih_r d\sigma.\label{cuts2}
\label{obser-1}
\end{align}
for every $y,h\in\R^d$, $s\in (0,\infty)$, $i=1,\ldots,d$ and $j=1,\ldots,m$, which (applied two times, the first one with the choice $y=x$ and the second time with $y=x-h$) implies that
\begin{align}
&|(D^{1}\bm{T}(s)\bm{f})(x+2h)-2(D^{1}\bm{T}(s)\bm{f})(x+h)+(D^{1}\bm{T}(s)\bm{f})(x)|\notag\\
&\qquad\qquad\qquad=\bigg (\sum_{j=1}^m\bigg |\sum_{i,r=1}^d\int_0^1[(\partial_{ir}\bm{T}(s)\bm{f})_j(x+(1+\sigma)h)-(\partial_{ir}\bm{T}(s)\bm{f})_j(x+\sigma h)]h_rd\sigma\bigg |^2\bigg )^{\frac{1}{2}}\notag\\
&\qquad\qquad\qquad=\bigg (\sum_{j=1}^m\bigg |\sum_{i,l,r=1}^d\int_0^1\int_0^1\big
((\partial_{ilr}\bm{T}(s)\bm{f})_j(x+(\tau+\sigma)h)h_rh_ld\sigma d\tau\bigg |^2\bigg )^{\frac{1}{2}}\notag\\
&\qquad\qquad\qquad\le K\|D^3\bm{T}(s)\bm{f}\|_{C_b(\R^d;\R^{md^3})}|h|^2\notag
\end{align}
for every $x,h\in\R^d$ and $s\in(0,\infty)$. By the previous estimate and \eqref{aim} (with $\varepsilon=\lambda-H_\nu$), we get
\begin{align}
|\bm{b}_h(x+2h)-2\bm{b}_h(x+h)+\bm{b}_h(x)|
% \le &\int_{|h|^2}^{\infty}e^{-\lambda s}|(D^{1}\bm{T}(s)\bm{f})(x+2h)-2(D^{1}\bm{T}(s)\bm{f})(x+h)+(D^{1}\bm{T}(s)\bm{f})(x)|ds\notag\\
% =&\int_{|h|^2}^{\infty}e^{-\lambda s}\max_{|k|=1}\bigg (\sum_{j=1}^m\bigg |\sum_{i,r=1}^d\int_0^1((\partial_{ir}\bm{T}(s)\bm{f})_j(x+(1+\sigma)h)-(\partial_{ir}\bm{T}(s)\bm{f})_j(x+\sigma h))k_jh_$\R^d$\sigma\bigg |^2\bigg )^{\frac{1}{2}}ds\notag\\
% = &\int_{|h|^2}^{\infty}e^{-\lambda s}\max_{|k|=1}\bigg (\sum_{j=1}^m\bigg |\sum_{i,l,r=1}^d\int_0^1\int_0^1\big
% ((\partial_{ilr}\bm{T}(s)\bm{f})_j(x+(\tau+\sigma)h\big )k_jh_rh_ld\sigma d\tau\bigg |^2\bigg )^{\frac{1}{2}}ds\notag\\
% \le &K\bigg (\int_{|h|^2}^{\infty}e^{-\lambda s}\|D^3\bm{T}(s)\bm{f}\|_{C_b(\R^d;\R^{3md})}ds\bigg )|h|^2\notag\\
% \le & K\bigg (\int_{|h|^2}^{\infty}e^{-\frac{\lambda-H}{2}s}s^{-\frac{3}{2}}ds\bigg )|h|^2\|\bm{f}\|_{C_b(\R^d;\R^m)}\notag\\
\leq K\bigg (\int_{|h|^2}^{\infty}s^{-\frac{3}{2}}ds\bigg )|h|^2\|\bm{f}\|_{C_b(\R^d;\R^m)}\le K|h|\|\bm{f}\|_{C_b(\R^d;\R^m)}.
\label{zyg_b_ell}
\end{align}
Combining 
\eqref{stima-norma-sup}, \eqref{zyg_deriv1}, \eqref{zyg_a1_ell} and \eqref{zyg_b_ell} we get \eqref{zyg_ell} in the case $\alpha=0$.
\smallskip

\noindent {\boldmath$(\alpha\in(0,1)).$} By the arguments in \eqref{zyg_C1}, for any $\bm{f}\in \mathcal{C}^{\alpha}_b(\R^d;\R^m)$,  $\bm{u}_\lambda$ belongs to $C^1_b(\R^d;\R^m)$. The facts that $\bm{u}_\lambda$ belongs to $C^2_b(\R^d;\R^m)$ and
\begin{align*}
(D^2\bm{u}_{\lambda})(x)=\int_0^{\infty}e^{-\lambda s}(D^2\bm{T}(s)\bm{f})(x)ds,\qquad\;\, x \in \Rd,
\end{align*}
follow by the Lebesgue dominated convergence theorem.
Moreover, 
% and observing that by \eqref{aim} (with $\eps=1$) it holds that
% \begin{align*}
% &|(D^1\bm{T}(t)\bm{f})(x+h)k-(D^1\bm{T}(t)\bm{f})(x)k- (D^2\bm{T}(t)\bm{f})(x))(h,k)|\\
% =&\bigg (\sum_{j=1}^m\bigg (\sum_{i,r=1}^d\int_0^1\left((\partial_{ir}\bm{T}(t)\bm{f})_j(x+\sigma h)-(\partial_{ir}\bm{T}(t)\bm{f})_j(x)\right)h_rk_id\sigma\bigg )^2\bigg )^{\frac{1}{2}}\\
% \leq &K e^{(H+1)t}t^{\frac{\alpha-2}{2}}|h||k|\|\bm{f}\|_{\mathcal{C}^\alpha(\R^d;\R^m)}.
% \end{align*}
% for any $t\in (0,\infty)$ and $\textcolor{red}{x,h,k \in \Rd}$. Observe that 
by \eqref{teta1teta2infty} (with $\eps=\frac{\lambda-H_\nu}{2}$), it holds that
\begin{align}
\|D^2\bm{u}_\lambda\|_{C_b(\R^d;\R^{md^2})} 
% &\leq \int_0^{\infty}e^{-\lambda s} \|D^{2}(\bm{T}(s)\bm{f})\|_{C_b(\R^d;\R^m\times \R^{2\times d})}ds\notag\\
\leq K\|\bm{f}\|_{\mathcal{C}^{\alpha}_b(\R^d;\R^m)}\left(\int_0^{\infty}e^{-\frac{\lambda-H_\nu}{2} s}s^{\frac{\alpha-2}{2}} ds\right)= K\|\bm{f}\|_{\mathcal{C}^{\alpha}_b(\R^d;\R^m)}.
\label{Schauder_deriv2}
\end{align}
It remains to prove that $D^{2}\bm{u}_\lambda$ belongs to $\mathcal{C}^{\alpha}_b(\R^d,\R^{md^2})$. We set
\begin{align*}
\tilde{\bm{a}}_h(x):=\int_0^{|h|^2}e^{-\lambda s}(D^{2}\bm{T}(s)\bm{f})(x)ds, && \tilde{\bm{b}}_h(x):=\int_{|h|^2}^{\infty}e^{-\lambda s}(D^{2}\bm{T}(s)\bm{f})(x)ds.
    \end{align*}
for every $x,h\in\R^d$. By design we get
\begin{align*}
|D^{2}\bm{u}_\lambda(x+h)-D^{2}\bm{u}_\lambda(x)| &\leq |\tilde{\bm{a}}_h(x+h)-\tilde{\bm{a}}_h(x)|+ |\tilde{\bm{b}}_h(x+h)-\tilde{\bm{b}}_h(x)|,\qquad\;\,x,h\in\R^d.
\end{align*}
Using \eqref{teta1teta2infty} (with $\varepsilon=\lambda-H_\nu$) we estimate
\begin{align}
|\tilde{\bm{a}}_h(x+h)-\tilde{\bm{a}}_h(x)|\leq &\int_0^{|h|^2}e^{-\lambda s}|(D^{2}\bm{T}(s)\bm{f})(x+h)-(D^{2}\bm{T}(s)\bm{f})(x)|ds\notag\\
% &\leq K \int_0^{|h|^2}e^{-\lambda s}\|D^{2}(\bm{T}(s)\bm{f})\|_{C_b(\R^d;\R^{2md})}ds\notag\\
% &\leq K\bigg (\int_0^{|h|^2}e^{-\frac{\lambda-H}{2} s}s^{(\alpha-2)/2}ds\bigg )\|\bm{f}\|_{\mathcal{C}^\alpha_b(\R^d;\R^m)}\notag\\
\leq &K\bigg (\int_0^{|h|^2}s^{\frac{\alpha-2}{2}}ds\bigg )\|\bm{f}\|_{\mathcal{C}^\alpha_b(\R^d;\R^m)}\notag\\
= & K|h|^{\alpha}\|\bm{f}\|_{\mathcal{C}^\alpha_b(\R^d;\R^m)}
\label{Schauder_a_ell}
\end{align}
and
\begin{align}
|\tilde{\bm{b}}_h(x+h)-\tilde{\bm{b}}_h(x)| & \leq \int_{|h|^2}^{\infty}e^{-\lambda s}|(D^{2}\bm{T}(s)\bm{f})(x+h)-(D^{2}\bm{T}(s)\bm{f})(x)|ds\notag\\
&=\int_{|h|^2}^{\infty}e^{-\lambda s}\bigg (\sum_{j=1}^m\bigg |\sum_{i,l=1}^d[(\partial_{il}\bm{T}(s)\bm{f})_j(x+h)-(\partial_{il}\bm{T}(s)\bm{f})_j(x)]\bigg |^2\bigg )^{\frac{1}{2}}ds\notag\\
&=\int_{|h|^2}^{\infty}e^{-\lambda s}\bigg (\sum_{j=1}^m\bigg |\sum_{i,l,r=1}^d\int_0^1(\partial_{ilr}\bm{T}(s)\bm{f})_j(x+\sigma h)h_rd\sigma\bigg |^2\bigg )^{\frac{1}{2}}ds\notag\\
% & \leq K|h|\int_{|h|^2}^{\infty}e^{-\lambda s}\|D^{3}\bm{T}(s)\bm{f}\|_{C_b(\R^d;\R^{3md})}ds\notag\\
%         & \leq K\bigg (\int_{|h|^2}^{\infty}e^{-\frac{\lambda-H}{2} s} s^{(\alpha-3)/2}ds\bigg ) |h|\|\bm{f}\|_{\mathcal{C}^\alpha_b(\R^d;\R^m)}\notag\\
& \leq K\bigg (\int_{|h|^2}^{\infty} s^{\frac{\alpha-3}{2}}ds\bigg ) |h|\|\bm{f}\|_{\mathcal{C}^\alpha_b(\R^d;\R^m)}\notag\\
&= K |h|^{\alpha}\|\bm{f}\|_{\mathcal{C}^\alpha_b(\R^d;\R^m)}. \label{Schauder_b_ell}
\end{align}
Combining \eqref{zyg_deriv1}, \eqref{Schauder_deriv2}, \eqref{Schauder_a_ell} and \eqref{Schauder_b_ell} we get \eqref{zyg_ell} for any $\alpha\in(0,1)$.

To complete the proof, we need to show that, if $\lambda>H_\nu$, then $\uu_{\lambda}$ solves the elliptic equation $\lambda\uu-\bm{\A}\uu=\f$ when $\f\in\mathcal{C}^{\alpha}_b(\R^d;\R^m)$ and $\alpha\in [0,1)$. For this purpose, since we can differentiate under the integral sign, it is immediate to check that
\begin{align*}
(\bm{\A}\uu_{\lambda})(x)=\int_0^{\infty}e^{-\lambda t}(\bm{\A}\bm{T}(t)\f)(x)dt,\qquad\;\,x\in\R^d.
\end{align*}
Since $D_t\f=\bm{\A}\bm{T}(\cdot)\f$, integrating by parts, we easily see that $\lambda\uu_{\lambda}-\bm{\A}\uu_{\lambda}=\f$.
\end{proof}

\subsection{The parabolic case}
For a fixed $T>0$, we study the regularity of the mild solution of the problem
\begin{align}
    \eqsys{D_t\bm{w}(t,x)=(\bm{\A}\bm{w})(t,x)+\bm{g}(t,x), 
    & (t,x)\in(0,T]\times\R^d;\\ [1mm]
    \bm{w}(0,x)=\bm{f}(x) & x\in\R^d,}
\label{para-nonhom}
\end{align}
namely, the regularity of the function $\bm{v}$, defined by 
\begin{align}\label{Sol_Mild_Parab}
    \bm{v}(t,x):=(\bm{T}(t)\bm{f})(x)+\int_0^t[(\bm{T}(s)\bm{g}(t-s,\cdot)](x)ds,\qquad\;\, t\in[0,T],\ x\in\R^d,
\end{align}
when $\bm{f}$ and $\bm{g}$ belong to suitable subsets of bounded and continuous functions. 
%In order to approach the problem we need to define some functional spaces. For $k\in \N\cup\{0\}$ (resp. $\alpha\in (0,\infty)\setminus\N$) we introduce the space $C^{0,k}_b([0,T]\times\R^d;\R^m)$ (resp. $\mathcal{C}^{0,\alpha}_b([0,T]\times\R^d;\R^m)$) as the set of all the bounded and continuous functions $\bm{g}:[0,T]\times\R^d\to\R^m$ such that for every $t\in[0,T]$ the function $\bm{g}(t,\cdot)$ belongs to $C^{k}_b(\R^d;\R^m)$ (resp. to $\mathcal{C}^{\alpha}_b(\R^d;\R^m)$) and
%\begin{align*}          
%&\|\bm{g}\|_{C^{0,k}_b([0,T]\times\R^d;\R^m)}:=\sup_{t\in[0,T]}\|\bm{g}(t,\cdot)\|_{C^{k}_b(\R^d;\R^m)}<\infty,\\[1mm]
%&{\rm (resp.\,} 
%\|\bm{g}\|_{\mathcal{C}^{0,\alpha}_b([0,T]\times\R^d;\R^m)}:=\sup_{t\in[0,T]}\|\bm{g}(t,\cdot)\|_{\mathcal{C}^{\alpha}_b(\R^d;\R^m)}<\infty{\rm )}.
%\end{align*}
%With this norm $C^{0,k}_b([0,T]\times\R^d;\R^m)$ 
%(resp. $\mathcal{C}^{0,\alpha}_b([0,T]\times\R^d;\R^m)$)
%is a Banach space.

\begin{thm}
\label{thm-4.3}
Assume Hypotheses $\ref{hyp-derivata-3}$ and fix $T>0$, $\alpha\in[0,1)$, $\bm{f}\in \mathcal{C}^{2+\alpha}_b(\R^d;\R^m)$ and $\bm{g}\in \mathcal{C}^{0,\alpha}_b([0,T]\times \R^d;\R^m)$. The function $\bm{v}$, introduced in \eqref{Sol_Mild_Parab}, belongs to $\mathcal{C}^{0,2+\alpha}_b([0,T]\times\R^d;\R^m)$ and there exists a constant $c>0$, independent of $\bm{f}$ and $\bm{g}$, such that 
\begin{align}
\|\bm{v}\|_{\mathcal{C}^{0,2+\alpha}_b([0,T]\times\R^d;\R^m)}\leq c\big (\|\bm{f}\|_{\mathcal{C}^{2+\alpha}_b(\R^d;\R^m)}+\|\bm{g}\|_{\mathcal{C}^{0,\alpha}_b([0,T]\times\R^d;\R^m)}\big ).
\label{Schauder_parab}
\end{align}
In particular $\bm{v}$ is the unique bounded classical solution to the Cauchy problem \eqref{para-nonhom}.
\end{thm}

\begin{proof}
As in the proof of Theorem \ref{thm_Schauder_elliptic}, $K$ will denote a positive constant independent of $\bm{f}$, $\bm{g}$, $x$ and $T$ that may change from line to line. 
By \eqref{teta1teta2infty} (with $\theta_1=\theta_2=2+\alpha$ and $\eps=|H_\nu|$), we get
\begin{align}
\sup_{t\in[0,T]}\|\bm{T}(t)\bm{f}\|_{\mathcal{C}^{2+\alpha}_b(\R^d;\R^m)}\leq Ke^{2|H_\nu|T}\|\bm{f}\|_{\mathcal{C}^{2+\alpha}_b(\R^d;\R^m)}.
\label{Preserv_reg}
\end{align}
Now, we set
\begin{align*}
\bm{w}(t,x):=\int_0^t(\bm{T}(s)\bm{g}(t-s,\cdot))(x)ds,\qquad\;\,t\in [0,T],\;\,x\in\R^d,
\end{align*}
and study its regularity. We consider the cases $\alpha=0$ and $\alpha\in (0,1)$ separately.
\smallskip

\noindent {\boldmath$(\alpha=0).$} Arguing as in Theorem \ref{thm_Schauder_elliptic} we get that for every $t\in[0,T]$ the function $\bm{w}(t,\cdot)$ belongs to $C_b^1(\R^d;\R^m)$ and it holds
\begin{align*}
(D^1\bm{w})(t,x)=\int_0^t(D^1\bm{T}(s)\bm{g}(t-s,\cdot))(x)ds,\qquad t\in[0,T],\ x\in\R^d.
\end{align*}
By \eqref{aim} there exists a positive constant $c=c(T)$ such that
\begin{align}\label{ff}
\|\bm{w}\|_{C_b^{0,1}([0,T]\times\R^d;\R^m)}\leq c\|\bm{g}\|_{C_b([0,T]\times\R^d;\R^m)}.
\end{align}

In order to show that $\bm{w}$ belongs to $\mathcal{C}^{0,2}_b([0,T]\times\R^d;\R^m)$ we proceed in a similar way as in the proof of Theorem \ref{thm_Schauder_elliptic} and introduce the functions $\bm{a}_{t,h}$ and $\bm{b}_{t,h}$, defined by
\begin{align*}
\bm{a}_{t,h}(x):=\int_0^{\min\{t,|h|^2\}}(D^1\bm{T}(s)\bm{g}(t-s,\cdot))(x)ds,\qquad\;\, \bm{b}_{t,h}(x):=\int_{\min\{t,|h|^2\}}^{t}(D^1\bm{T}(s)\bm{g}(t-s,\cdot))(x)ds.
\end{align*}
for every $x,h\in\R^d$ and $t\in[0,T]$.
By \eqref{aim} (with $\eps=|H_\nu|$), we get
\begin{align}
&|\bm{a}_{t,h}(x+2h)-2\bm{a}_{t,h}(x+h)+\bm{a}_{t,h}(x)|\notag\\
\leq &\int_0^{\min\{t,|h|^2\}}|(D^1\bm{T}(s)\bm{g}(t-s,\cdot))(x+2h)-2(D^1\bm{T}(s)\bm{g}(t-s,\cdot))(x+h)+(D^1\bm{T}(s)\bm{g}(t-s,\cdot))(x)|ds\notag\\
% \leq & K\int_0^{\min\{t,|h|^2\}}e^{(H+1)s}s^{-\frac{1}{2}}\|\bm{g}(t-s,\cdot)\|_{C_b(\R^d;\R^m)}ds\notag\\
% \leq & K\max\{1,e^{(H+1)T}\}\bigg (\int_0^{\min\{t,|h|^2\}}s^{-\frac{1}{2}}ds\bigg ) \|\bm{g}\|_{C_b([0,T]\times\R^d;\R^m)}\notag\\
\leq & Ke^{2|H_\nu|T}\bigg(\int_0^{|h|^2}s^{-\frac{1}{2}}ds\bigg ) \|\bm{g}\|_{C_b([0,T]\times\R^d;\R^m)}
= Ke^{2|H_\nu|T}|h|\|\bm{g}\|_{C_b([0,T]\times\R^d;\R^m)}.\label{ffa}
\end{align}

Observe that $\bm{b}_{t,h}$ is non-trivial only if $|h|^2<t$, so all the following computations are done in this case. To provide an estimate for $\bm{b}_{t,h}$, we use \eqref{obser-1}, with $\f=\bm{g}(t-s,\cdot)$ and $y\in\{x-h,x\}$, and
\eqref{aim} (with $\eps=|H_\nu|$) to estimate
\begin{align}
&|\bm{b}_{t,h}(x+2h)-2\bm{b}_{t,h}(x+h)+\bm{b}_{t,h}(x)|\notag\\
\le &\int_{\min\{|h|^2,t\}}^{t}|(D^1\bm{T}(s)\bm{g}(t-s,\cdot))(x+2h)-2(D^1\bm{T}(s)\bm{g}(t-s,\cdot))(x+h)+(D^1\bm{T}(t-s)\bm{g}(s,\cdot))(x)|ds\notag\\
\leq & K\bigg (\int_{|h|^2}^{t}e^{(H_\nu+|H_\nu|)s}s^{-\frac{3}{2}}ds\bigg )|h|^2\|\bm{g}\|_{C_b([0,T]\times\R^d;\R^m)}\notag\\
\le & Ke^{2|H_\nu|T}\bigg (\int_{\min\{|h|^2,t\}}^{\infty}s^{-\frac{3}{2}}ds\bigg )|h|^2\|\bm{g}\|_{C_b([0,T]\times\R^d;\R^m)}
= Ke^{2|H_\nu|T}|h|\|\bm{g}\|_{C_b([0,T]\times\R^d;\R^m)}.\label{ffb}
%\label{Stime_b_parziali}
\end{align}
Combining \eqref{Preserv_reg}, \eqref{ff}, \eqref{ffa} and \eqref{ffb} we get \eqref{Schauder_parab} for $\alpha=0$.
\smallskip

%Now, to obtain the desired estimate, we need to distinguish three cases depending on the position of the number $1$ with respect to $t$ and $|h|^2$. If $|h|^2<t\leq 1$, then
%\begin{align*}
 %   |\bm{b}(x+2h)-2\bm{b}(x+h)+\bm{b}(x)|&\leq c\max\{1,e^{HT}\}\left(\int_{|h|^2}^{t}s^{-3/2}ds\right)|h|^2\|\bm{g}\|_{C_b(\R^d;\R^m)}\\
 %   &= c\max\{1,e^{HT}\}\left(-2t^{-1/2}+2|h|^{-1}\right)|h|^2\|\bm{g}\|_{C_b(\R^d;\R^m)}\\
 %   &\leq 2c\max\{1,e^{HT}\}|h|\|\bm{g}\|_{C_b(\R^d;\R^m)}.
%\end{align*}
%If $|h|^2\leq 1\leq t$, then
%\begin{align*}
%    |\bm{b}(x+2h)-2\bm{b}(x+h)+\bm{b}(x)|&\leq c\max\{1,e^{HT}\}\left(\int_{|h|^2}^{1}s^{-3/2}ds+\int_{1}^{t}ds\right)|h|^2\|\bm{g}\|_{C_b(\R^d;\R^m)}\\
 %   &\leq c\max\{1,e^{HT}\}\left(-2+2|h|^{-1}+T\right)|h|^2\|\bm{g}\|_{C_b(\R^d;\R^m)}\\
  %  &\leq c(2+T)\max\{1,e^{HT}\}|h|\|\bm{g}\|_{C_b(\R^d;\R^m)}.
%\end{align*}
%If $1\leq |h|^2< t$, then
%\begin{align*}
 %   |\bm{b}(x+2h)-2\bm{b}(x+h)+\bm{b}(x)|&\leq c\max\{1,e^{HT}\}\left(\int_{|h|^2}^{t}ds\right)|h|^2\|\bm{g}\|_{C_b(\R^d;\R^m)}\\
 %   &\leq cT\max\{1,e^{HT}\}|h|^2\|\bm{g}\|_{C_b(\R^d;\R^m)}\\
 %   &\leq cT^{3/2}\max\{1,e^{HT}\}|h|\|\bm{g}\|_{C_b(\R^d;\R^m)}.
%\end{align*}

\noindent {\boldmath$(\alpha\in(0,1)).$} 
Adapting the arguments in the proof of Theorem \ref{thm_Schauder_elliptic}, it is easy to check that the function $\bm{w}(t,\cdot)$ belongs to $C^2_b(\R^d;\R^m)$
for every $t\in[0,T]$ and
\begin{align}\label{deriv_second_w}
(D^2\bm{w})(t,x)=\int_0^t(D^2\bm{T}(s)\bm{g}(t-s,\cdot))(x)ds,\qquad\;\, t\in[0,T],\ x\in\R^d.
    \end{align}
Moreover by \eqref{alphaelle}, there exists a positive constant $c=c(T,\alpha)$ such that
\begin{align}\label{FF}
\|\bm{w}\|_{C_b^{0,2}([0,T]\times\R^d;\R^m)}\leq c\|\bm{g}\|_{\mathcal{C}^{0,\alpha}_b([0,T]\times\R^d;\R^m)}.
\end{align}
    
In order to show that $\bm{w}$ belongs to $\mathcal{C}^{0,2+\alpha}_b([0,T]\times\R^d;\R^m)$ we fix
$t\in (0,T]$ and introduce the functions $\tilde{\bm{a}}_{t,h}$ and $\tilde{\bm{b}}_{t,h}$, defined by
\begin{align*}
\tilde{\bm{a}}_{t,h}(x):=\int_0^{\min\{t,|h|^2\}}(D^2\bm{T}(s)\bm{g}(t-s,\cdot))(x)ds,\qquad \;\,\tilde{\bm{b}}_{t,h}(x):=\int_{\min\{t,|h|^2\}}^{t}(D^2\bm{T}(s)\bm{g}(t-s,\cdot))(x)ds
\end{align*}
for every $x,h\in\R^d$ and $t\in[0,T]$. By \eqref{deriv_second_w}, we get
\begin{align*}
|D^{2}\bm{w}(x+h)-D^{2}\bm{w}(x)| &\leq |\tilde{\bm{a}}_{t,h}(x+h)-\tilde{\bm{a}}_{t,h}(x)|+ |\tilde{\bm{b}}_{t,h}(x+h)-\tilde{\bm{b}}_{t,h}(x)|.
\end{align*}
Using \eqref{teta1teta2infty} (with $\eps=|H_\nu|$) we have
\begin{align}
|\tilde{\bm{a}}_{t,h}(x+h)-\tilde{\bm{a}}_{t,h}(x)| &\leq \int_0^{\min\{t,|h|^2\}}|(D^2\bm{T}(s)\bm{g}(t-s,\cdot))(x+h)-(D^2\bm{T}(s)\bm{g}(t-s,\cdot))(x)|ds\notag\\
% &\leq K\bigg (\int_0^{|h|^2}e^{(H+|H|)s}s^{\frac{\alpha-2}{2}}ds\bigg )\|\bm{g}\|_{\mathcal{C}^{0,\alpha}_b([0,T]\times\R^d;\R^m)}\notag\\
&\leq Ke^{2|H_\nu|T}\bigg (\int_0^{|h|^2}s^{\frac{\alpha-2}{2}}ds\bigg )\|\bm{g}\|_{\mathcal{C}^{0,\alpha}_b([0,T]\times\R^d;\R^m)}\notag\\
&\leq Ke^{2|H_\nu|T}|h|^{\alpha}\|\bm{g}\|_{\mathcal{C}^{0,\alpha}_b([0,T]\times\R^d;\R^m)}.
\label{FFa}
\end{align}
Furthermore, observing that $\tilde{\bm{b}}_{t,h}$ is non-trivial only if $|h|^2<t$, using \eqref{teta1teta2infty} (with $\eps=|H_\nu|$), we can show that
\begin{align}
|\tilde{\bm{b}}_{t,h}(x+h)-\tilde{\bm{b}}_{t,h}(x)| & \leq \int_{|h|^2}^{t}|(D^2\bm{T}(s)\bm{g}(t-s,\cdot))(x+h)-(D^2\bm{T}(s)\bm{g}(t-s,\cdot))(x)|ds\notag\\
% & \leq K|h|\int_{|h|^2}^{t}\left\|D^3\bm{T}(s)\bm{g}(t-s,\cdot)\right\|_{C_b(\R^d;\R^{3md})}ds\notag\\
& \leq K\bigg (\int_{|h|^2}^{t}e^{(H_\nu+|H_\nu|)s} s^{\frac{\alpha-3}{2}}ds\bigg ) |h|\|\bm{g}\|_{\mathcal{C}^{0,\alpha}_b([0,T]\times\R^d;\R^m)}\notag\\
&\leq Ke^{2|H_\nu|T} |h|^{\alpha}\|\bm{g}\|_{\mathcal{C}^\alpha_b([0,T]\times\R^d;\R^m)}. \label{FFb}
\end{align}
Combining \eqref{Preserv_reg}, \eqref{FF}, \eqref{FFa} and \eqref{FFb} we get \eqref{Schauder_parab} for any $\alpha\in(0,1)$.

To complete the proof, we need to show that $\bm{v}$ is the unique bounded classical solution to the Cauchy problem \eqref{para-nonhom}. For this purpose, we observe that the previous computations and the properties of the semigroup $\{\bm{T}(t)\}_{t\ge 0}$ show that
\begin{eqnarray*}
(\bm{\A}\bm{w})(t,x)=\int_0^t(\bm{\A}\bm{T}(t-s)\g(s,\cdot))(x)ds=\int_0^t(D_t\bm{T}(t-s)\g(s,\cdot))(x)ds,\qquad\;\,t\in [0,T],\;\,x\in\R^d.    
\end{eqnarray*}
Moreover, applying the dominated convergence theorem, together with estimate \eqref{pointwise} it can be shown that $\bm{w}$ is continuously differentiable in $[0,T]\times\Rd$ and 
\begin{eqnarray*}
(D_t\bm{w})(t,x)=\int_0^tD_t\bm{T}(t-s)\g(s,\cdot))(x)ds+\bm{g}(t,x),\qquad\;\,t\in [0,T],\;\,x\in\R^d.    
\end{eqnarray*}
Hence, $\bm{w}$ is a classical solution to the equation $D_t\bm{w}=\bm{\A}\bm{w}+\bm{g}$ in $[0,T]\times\Rd$. Moreover, $\bm{w}(0,\cdot)=\bm{0}$.
As a byproduct, the function $\bm{v}$, defined in \eqref{Sol_Mild_Parab}, is a locally in time bounded classical solution to the Cauchy problem \eqref{para-nonhom}. The uniqueness of the locally in time bounded classical solution to \eqref{para-nonhom} follows from Theorem \ref{exi_thm}.
\end{proof}

\section{Regularity results in $L^p(\Rd;\Rm)$}
\label{sect-5}

In this section we focus our attention to the $L^p$-setting. As the scalar case shows, Cauchy problems associated to elliptic and parabolic equations with unbounded coefficients are, in general, not well posed in $L^p$-spaces of the Lebesgue measure, despite the relevance of such spaces, unless restrictive assumptions are considered. In some cases, the classical $L^p$-spaces are not preserved by the action of the semigroup (see \cite{AngLorCom} for further details). Here, we are interested in studying properties of $\bm{T}(t)$ in $L^p(\Rd;\Rm)$ when this space is preserved by its action and when an estimate like
\begin{equation*}
\|\bm{T}(t)\f\|_{L^p(\Rd;\Rm)}\le c_p(t)\|\f\|_{L^p(\Rd;\Rm)},\qquad\;\,t\in (0,\infty),
\end{equation*}
holds true for some measurable function $c_p:(0,\infty)\to(0,\infty)$.
To this aim, we assume the following
additional hypothesis.

\begin{hyp}\label{ipo-p}
%There exists $\nu\in (0,2)$ such that 
The restriction of $\{S_{\nu}(t)\}_{t\ge 0}$ to $C_c(\R^d)$ extends to $L^1(\Rd)$ with a strongly continuous semigroup. 
\end{hyp}

\begin{remark}%\label{L1-bound}
{\rm  Sufficient conditions that ensure the validity of Hypothesis \ref{ipo-p} can be found in \cite[Theorems 5.3 and 5.4]{AngLorCom}. For instance, one can assume the existence of a positive constant $K$ such that
\begin{equation} 
\nu\Lambda_C(x)-(\diver{{\bm b}})(x)+\sum_{i,j=1}^dD_{ij}q_{ij}(x)\le K \qquad\;\, x \in \Rd.
\label{cond-diver}
\end{equation}
In this case $\|S_{\nu}(t)f\|_{L^1(\Rd)}\le e^{Kt}\|f\|_{L^1(\Rd)}$ for any $t>0$ and any $f \in L^1(\Rd)$. }
\end{remark}

From here onward Hypothesis \ref{ipo-p} will be our standing assumption (we still continue to assume Hypotheses \ref{base} as before).

The following result can be proved by adapting the arguments in \cite{ALP}. For the reader’s convenience, we provide a sketch of the proof.

\begin{pro}\label{towardslp}
Under Hypothesis $\ref{ipo-p}$, the restriction of the semigroup $\{\bm{T}(t)\}_{t\ge 0}$ to $C_c(\R^d;\R^m)$ extends to  
$L^p(\Rd;\Rm)$, for any $p \in [2, \infty)$, with a strongly continuous semigroup, which we still denote by $\{\bm{T}(t)\}_{t\ge0}$.

If Hypotheses $\ref{hyp-derivata-1}$ are satisfied then for any $\f \in W^{k,p}(\Rd;\Rm)$ and $t\in (0,\infty)$, $\bm{T}(t)\f$ belongs to $W^{\ell,p}(\Rd;\Rm)$ for any $k,\ell\in \{0,1\}$ with $k \le \ell$. Moreover, for any $\varepsilon>0$ and $k,\ell$ as above, there exists a positive constant $c_{\varepsilon}$, that blows up as $\eps$ approaches zero, such that
\begin{equation}\label{sob_est}
\|{\bm T}(t)\f\|_{W^{\ell,p}(\Rd;\Rm)} \le  c_{\varepsilon}e^{(\tilde{H}_p+\varepsilon)t}t^{-\frac{\ell-k}{2}}\|S_{\nu}(t)\|_{\mathcal{L}(L^1(\Rd))}^{\frac{1}{p}}\|\f\|_{W^{k,p}(\Rd;\Rm)},\qquad\;\,t\in (0,\infty);
\end{equation}
here, $\tilde{H}_p=\frac{1}{2}H+(\frac{1}{2}-\frac{1}{p})\Theta_C$, where $H$ is the constant appearing in Hypotheses \ref{base}(iii) and $\Theta_C$ is the supremum over $\Rd$ of the function $\Lambda_C$. In particular, the function $t\mapsto\bm{T}(t)\f$ is continuous in $(0,\infty)$ with values in $W^{\ell,p}(\R^d;\R^m)$ for every $\f\in L^p(\R^d;\R^m)$.
Further, if Hypotheses $\ref{hyp-derivata-2}$ $($resp. Hypotheses $\ref{hyp-derivata-3})$ are satisfied, instead of Hypotheses \ref{hyp-derivata-1} then, for every $t\in (0,\infty)$, $\bm{T}(t)\f$ belongs to $W^{\ell,p}(\Rd;\Rm)$ with $k,\ell\in \{0,1,2\}$ $($resp. $k,\ell\in \{0,1,2,3\})$ and $k\leq\ell$, estimate \eqref{sob_est} is satisfied for such values of $k$ and $\ell$, and the function $t\mapsto\bm{T}(t)\f$ is continuous in $(0,\infty)$ with values in $W^{\ell,p}(\R^d;\R^m)$ for every $\f\in L^p(\R^d;\R^m)$.
\end{pro}

\begin{proof}
Estimate \eqref{cossali} shows that, for any $\f \in C_c(\Rd;\Rm)$, %the classical solution $\uu$ to the Cauchy problem \eqref{cp} satisfies the estimate
\begin{equation}\label{p=2}
|({\bm T}(t)\f)(x)|^2\le e^{Ht}(S_{\nu}(t)|\f|^2)(x), \qquad\;\, t\in (0,\infty),\;\, x \in \Rd.
\end{equation}
Thanks to \eqref{int_rep}  
and the H\"older inequality, we can extend estimate \eqref{p=2} to any $p>2$. Indeed, by the Jensen inequality, it holds 
\begin{equation*}
|({\bm T}(t)\f)(x)|^p\le  \big (e^{Ht}(S_{\nu}(t)|\f|^2)(x)\big )^{\frac{p}{2}}= e^{(\frac{p}{2}H+\frac{p-2}{2}\Theta_C)t}(S_{\nu}(t)|\f|^p)(x), \qquad\; t\in (0,\infty),\;\, x \in \Rd,
\end{equation*}
Integrating the previous inequality and using Hypothesis \ref{ipo-p}, we deduce that
\begin{equation}\label{norm-est}
\|\bm{T}(t)\f\|_{L^p(\R^d;\R^m)}\le e^{\tilde{H}_pt}\|S_{\nu}(t)\|_{\mathcal{L}(L^1(\Rd))}^{\frac{1}{p}}\|\f\|_{L^p(\R^d;\R^m)},\qquad\;\, t\in (0,\infty).
\end{equation}
Estimate \eqref{norm-est} and the density of $C_c(\Rd;\Rm)$ in $L^p(\Rd;\Rm)$ allow us to extrapolate $\bm{T}(t)$ to the $L^p$-scale, for $p\in [2,\infty)$.

The strong continuity of the extrapolated semigroup can be proved starting from the equality
\begin{eqnarray*}
({\bm T}(t)\f)(x)-\f(x)=\int_0^t ({\bm T}(r) {\bm \A}\f)(x)dr, \qquad\;\, t\in (0,\infty),\;\, x \in \Rd,
\end{eqnarray*}
which holds true for any $\f\in C^2_c(\Rd;\Rm)$, and that, together with \eqref{norm-est}, yields 
\begin{eqnarray*}
\|{\bm T}(t)\f-\f\|_{L^p(\Rd;\Rm)}\le \|{\bm \A}\f\|_{L^p(\Rd;\Rm)}\sup_{\tau\in (0,1]}\|S_{\nu}(\tau)\|_{\mathcal{L}(L^1(\Rd))}^{\frac{1}{p}}
\int_0^t e^{\tilde{H}_pr}dr,\qquad\;\,t\in (0,1].
\end{eqnarray*}
Since the right hand side of the above inequality vanishes as $t$ approaches zero, ${\bm T}(t)\f$ converges to $\f$ in $L^p(\Rd;\Rm)$ as $t$ tends to zero. A standard density argument allows us to get the same result for $\f \in L^p(\Rd;\Rm)$. 

The space regularity properties of the function $\bm{T}(t)\f$ for every $t\in (0,\infty)$ and estimates \eqref{sob_est} are immediate consequence of \eqref{pointwise} and Hypothesis \ref{ipo-p}.

Finally, the continuity of the map $t\mapsto\bm{T}(t)\f$ in $(0,\infty)$ with values in $W^{\ell,p}(\R^d,\R^m)$ (where $\ell$ depends on which among Hypotheses \ref{hyp-derivata-1}, \ref{hyp-derivata-2} and \ref{hyp-derivata-3} are assumed to be true), when $\f\in L^p(\R^d;\R^m)$, follows from the strong continuity of the semigroup $\{\bm{T}(t)\}_{t\ge 0}$ observing that, for every $t_0\in (0,\infty)$ it holds that
\begin{eqnarray*}
\|\bm{T}(t)\f-\bm{T}(t_0)\f\|_{W^{\ell,p}(\R^d;\R^m)}
\le\|\bm{T}(t_0)\|_{\mathcal{L}(L^p(\R^d;\R^m);W^{\ell,p}(\R^d;\R^m))}\|\bm{T}(t-t_0)\f-\f\|_{L^p(\R^d;\R^m)},
\end{eqnarray*}
for every $t\in (t_0,\infty)$ and 
\begin{align*}
&\|\bm{T}(t)\f-\bm{T}(t_0)\f\|_{W^{\ell,p}(\R^d;\R^m)}
\\
&\qquad\qquad\le 
\left(\sup_{t\in [0,t_0]}\|\bm{T}(t)\|_{\mathcal{L}(L^p(\R^d;\R^m);W^{\ell,p}(\R^d;\R^m))}\right)\|\bm{T}(t_0-t)\f-\f\|_{L^p(\R^d;\R^m)},
\end{align*}
for every $t\in (0,t_0)$.
\end{proof}

Thanks to Proposition \ref{towardslp} the extrapolated semigroups in the $L^p$-setting are consistent, i.e. for any $p,q\in[2,\infty)$,  $t\in (0,\infty)$ and $\f\in L^p(\Rd;\Rm)\cap L^q(\Rd;\Rm)$ it holds that ${\bm T}_p(t)\f=\bm{ T}_q(t)\f$ almost everywhere in $\Rd$. For this reason, from now on, ${\bm T}(t)$ will denote the continuous extension of the semigroup generated in $C_b(\Rd;\Rm)$ in the $L^p$-scale.

By \eqref{int_p}, \eqref{int_pp}, \eqref{sob_est} and arguing as in the proof of Theorem \ref{thm-1-2} we get the following result.

\begin{thm}
Let Hypotheses $\ref{hyp-derivata-1}$ be satisfied. For any $t>0$, ${\bm T}(t)$ is a bounded linear operator from $B^{\theta_1}_{p,p}(\Rd;\Rm)$ into $B^{\theta_2}_{p,p}(\Rd;\Rm)$ for any $0 \le \theta_1\le \theta_2 \leq 1$, and 
\begin{equation}\label{teta1teta2}
\|{\bm T}(t)\f\|_{B^{\theta_2}_{p,p}(\Rd;\Rm)}\le c_{\varepsilon}e^{(\tilde{H}_p+\eps)t}t^{-\frac{\theta_2-\theta_1}{2}}\|S_{\nu}(t)\|_{\mathcal{L}(L^1(\Rd))}^{\frac{1}{p}}\|\f\|_{B^{\theta_1}_{p,p}(\Rd;\Rm)},\qquad\;\,t\in (0,\infty),
\end{equation}
for any $\f \in B^{\theta_1}_{p,p}(\Rd;\Rm)$, any $\varepsilon>0$ and some positive constant $c_{\varepsilon}$, independent of $\f$, that blows up as $\eps$ approaches zero. Here, $\tilde{H}_p$ is the constant in Proposition $\ref{towardslp}$. 
 If Hypotheses $\ref{hyp-derivata-1}$ are replaced by Hypotheses $\ref{hyp-derivata-2}$ $($resp. Hypotheses $\ref{hyp-derivata-3})$, then, for any $t>0$, ${\bm T}(t)$ is a bounded linear operator from $B^{\theta_1}_{p,p}(\Rd;\Rm)$ in $B^{\theta_2}_{p,p}(\Rd;\Rm)$ for any $0 \le \theta_1\le \theta_2 \leq 2$ $($resp. for any $0 \le \theta_1\le \theta_2 \leq 3)$ and estimate \eqref{teta1teta2} holds true for any $\theta_1$ and $\theta_2$ as above.
 
Moreover, if Hypotheses $\ref{hyp-derivata-1}$ $($Hypotheses $\ref{hyp-derivata-2}$ or Hypotheses $\ref{hyp-derivata-3}$, respectively$)$ hold true, then for $\ell=1$ $(\ell\in\{1,2\}$ or $\ell\in \{1,2,3\}$, respectively$)$ and $t>0$, ${\bm T}(t)$ is a bounded linear operator from $B^{\theta_1}_{p,p}(\Rd;\Rm)$ to $W^{\ell,p}(\Rd;\Rm)$, for any $\theta_1 \in(0,\ell)$ and it satisfies, for any $\eps>0$, the estimate
 \begin{equation*}%\label{alphaelleLp}
\|{\bm T}(t)\f\|_{W^{\ell,p}(\Rd;\Rm)}\le c_{\varepsilon}e^{(\tilde{H}_p+\eps)t}t^{-\frac{\ell-\theta_1}{2}}\|S_{\nu}(t)\|_{\mathcal{L}(L^1(\R^d))}^{\frac{1}{p}}\|\f\|_{B^{\theta_1}_{p,p}(\Rd;\Rm)},\qquad\;\,t\in (0,\infty),
\end{equation*}
where $c_{\varepsilon}$ and $\tilde{H}_p$ have the same meaning as above.
\end{thm}

Estimates \eqref{sob_est} and the Sobolev embedding theorem allow also to deduce a hypercontractivity type result for the semigroup and its derivatives.

\begin{thm}
Under Hypotheses $\ref{hyp-derivata-1}$, ${\bm T}(t)$ belongs to $\mathcal{L}(L^p(\Rd;\Rm); W^{1,r}(\Rd;\Rm))$ for any $p \le r \le \infty$ and any $t\in (0,\infty)$. More precisely, ${\bm T}(t)\f$ belongs to $C^1_b(\Rd;\Rm)$ for any $t>0$. If Hypotheses $\ref{hyp-derivata-1}$ are replaced by Hypotheses $\ref{hyp-derivata-2}$ $($resp. Hypotheses $\ref{hyp-derivata-3})$ then ${\bm T}(t)\in \mathcal{L}(L^p(\Rd;\Rm); W^{2,r}(\Rd;\Rm))$ for every $t>0$ and $p \le r \le \infty$ $($resp. ${\bm T}(t)\in \mathcal{L}(L^p(\Rd;\Rm); W^{3,r}(\Rd;\Rm))$ for every $t>0$ and $p \le r \le \infty$$)$. More precisely, ${\bm T}(t)\f$ belongs to $C^2_b(\Rd;\Rm)$ $($resp. $C^3_b(\Rd;\Rm))$ for every $t>0$.
\end{thm}

\begin{proof} 
We just prove the first assertion of the statement since the remaining parts can be proved analogously. For any $\f \in L^p(\Rd;\Rm)$ and $t\in (0,\infty)$, by \eqref{sob_est} with $(k,\ell)=(0,1)$ we deduce that ${\bm T}(t)\f$ belongs to $W^{1,p}(\Rd;\Rm)$. If $p>d$, then, by the Sobolev embedding theorem, ${\bm T}(t)\f \in \mathcal{C}^{1-\frac{d}{p}}_b(\Rd;\Rm)$. The semigroup law and estimate \eqref{teta1teta2infty} with $\theta_1=1-\frac{d}{p}$ and $\theta_2=1$ allow us to conclude. If $p=d$ then, by the Sobolev embedding theorem, ${\bm T}(t)\f$ belongs to $L^q(\Rd;\Rm)$ for any $q \in [p,\infty)$ and $t>0$. The semigroup law and estimate \eqref{sob_est} again with $(k,\ell)=(0,1)$ imply that ${\bm T}(t)\f$ belongs to $W^{1,q}(\Rd;\Rm)$ for any $q\in [p,\infty)$. In this case, arguing as in the case $p>d$ we get the claim.
Finally, if $p<d$, the Sobolev embedding theorem yields that ${\bm T}(t)\f \in L^{p_1}(\Rd;\Rm)$ where $p_1$ denotes the Sobolev conjugate of $p$, i.e. $p_1=p+\frac{p^2}{d-p}$. Using the semigroup law and estimate  \eqref{sob_est} again with $(k,\ell)=(0,1)$ we infer that ${\bm T}(t)\f \in W^{1,p_1}(\Rd;\Rm)$. Iterating this procedure we obtain that ${\bm T}(t)\f \in W^{1,p_n}(\Rd;\Rm)$ where 
\begin{eqnarray*}
p_n=p_{n-1}+\frac{p_{n-1}^2}{d-p_{n-1}}
\end{eqnarray*}
for every value of $n\in\N$ such that $p_{n-1}<d$.
Clearly, after a finite number of steps $p_n>d$ and the argument used in the first case ($p>d$) allows us to conclude. 
\end{proof}

Now, we are interested in proving quasi optimal $L^p$-regularity results for the stationary system
\begin{eqnarray*}
\lambda \uu-{\bm \A}\uu=\f\in L^p(\Rd;\Rm),
\end{eqnarray*}
namely for the function $\uu_\lambda=R(\lambda,{\bm \A})\f$  defined in \eqref{ulambda} where $\bm T(t)$ has to be understood as the extrapolated semigroup in $L^p(\Rd;\Rm)$.
Indeed, for $\lambda$ sufficiently large, the solution to the previous equation is unique and given by the Laplace transform of the function $\bm{T}(\cdot)\f$, since the semigroup 
$\{\bm{T}(t)\}_{t\ge 0}$ is strongly continuous in 
$L^p(\R^d;\R^m)$.
% By taking estimate \eqref{sob_est} into account, the function $\uu_\lambda:=R(\lambda,{\bm \A})\f$ is well defined whenever the map $t\mapsto e^{-(\lambda-\frac{pH}{2})t}\|S(t)\|_{\mathcal{L}(L^1(\Rd))}$ belongs to $L^1((0,\infty))$. In many circumstances the $\mathcal{L}(L^1(\Rd))$-norm of $S(t)$ is allowed to growth exponentially at $\infty$ (see Remark \ref{L1-bound}). In this case $\uu_\lambda$ is well defined for every $\lambda$ large enough.

\begin{thm}\label{appr}
Let $p\in[2,\infty)$ and let $\bm{A}_p$ be the realization of the operator $\bm{\A}$ in $L^p(\R^d;\R^m)$. Assume that there exists $M\geq 1$ and $\omega\in\R$ such that $\|S_{\nu}(t)\|_{\mathcal{L}(L^1(\Rd))}\le Me^{\omega t}$ for any $t\in (0,\infty)$. If Hypotheses $\ref{hyp-derivata-2}$ are satisfied, then for every $\lambda \in\rho(\bm{A}_p)$ and $\bm{f}\in L^p(\R^d;\R^m)$ the function $\uu_{\lambda}=R(\lambda,\bm{A}_p)\f$ belongs to $B^{1+\sigma}_{p,p}(\R^d;\R^m)$ for any $\sigma \in (0,1)$ and there exists a positive constant $c$, depending on $\lambda$ and independent of $\bm{f}$, such that
\begin{align}\label{est_n}
\| \uu_\lambda\|_{B^{1+\sigma}_{p,p}(\R^d;\R^m)}\leq c\|\bm{f}\|_{L^p(\R^d;\R^m)}.
\end{align}
On the other hand, if Hypotheses $\ref{hyp-derivata-3}$ are satisfied and $\f\in B^{\theta}_{p,p}(\R^d;\R^m)$ for some $\theta \in (0,1)$, then $\bm{u}_\lambda$ belongs to $B^{2+\sigma}_{p,p}(\R^d;\R^m)$ for any $\sigma \in (0,\theta)$ and there exists a positive constant $c$, depending on $\lambda$ and independent of $\bm{f}$, such that
\begin{align}
\|\bm{u}_\lambda\|_{B^{2+\sigma}_{p,p}(\Rd;\R^m)}\leq c\|\bm{f}\|_{B^{\theta}_{p,p}(\Rd;\R^m)}.
\label{rullp}
\end{align}
Finally, the domain of $\bm{A}_p$ is continuously embedded into $W^{1,p}(\R^d;\R^m)$ if Hypotheses $\ref{hyp-derivata-1}$ are satisfied and it is continuously embedded into $B^{1+\sigma}_{p,p}(\R^d;\R^m)$ for every $\sigma\in (0,1)$, if Hypotheses $\ref{hyp-derivata-2}$ are satisfied.
\end{thm}

\begin{proof}
Throughout the proof $K$ will denote a positive constant, independent of $\bm{f}$, that may change from line to line. Fix $p\in [2,\infty)$, $\bm{f}\in L^p(\R^d;\R^m)$
and set $\lambda_p:= \tilde{H}_p+\frac{\omega}{p}$, where $\tilde{H}_p$ is the constant defined in the statement of Proposition \ref{towardslp}.
Note that $\lambda_p+1$ belongs to the resolvent set of the operator $\bm{A}_p$. Moreover, using the H\"older inequality and \eqref{teta1teta2} (with $\theta_1=0$, $\theta_2=1+\sigma$ and $\varepsilon=\frac{1}{2}$), we get
\begin{align*}
\|\uu_{\lambda_p+1}\|_{B^{1+\sigma}_{p,p}(\Rd;\Rm)}&\le \int_0^{\infty}e^{-(\lambda_p+1)t}\|{\bm T}(t)\f\|_{B^{1+\sigma}_{p,p}(\Rd;\Rm)}dt\\
&\le K\bigg(\int_0^\infty e^{-\frac{t}{2}}t^{-\frac{1+\sigma}{2}}dt\bigg )\|\f\|_{L^p(\Rd; \Rm)}\\
& \le K\|\f\|_{L^p(\Rd; \Rm)}.
\end{align*}
 %If $\f$ takes values in $\C^m$, it suffices to observe that $R(\lambda,\bm{A}_p)\f=R(\lambda,\bm{A}_p){\rm Re}\,\f+iR(\lambda,\bm{A}_p){\rm Im}\,\f$, where
%${\rm Re}\,\f=({\rm Re}\,f_1,\ldots,{\rm Re}\,f_d)$ and
%${\rm Im}\,\f=({\rm Im}\,f_1,\ldots,{\rm Im}\,f_d)$.

For a general $\lambda\in\rho(\bm{A}_p)$, we use the resolvent identity to write 
$\uu_{\lambda}=R(\lambda_p+1,\bm{A}_p)[\f+(\lambda_p+1-\lambda)R(\lambda,\bm{A}_p)\f]$. Since the function $\f+(\lambda_p+1-\lambda)R(\lambda,\bm{A}_p)\f$ belongs to $L^p(\R^d;\R^m)$ and its $L^p(\R^d;\R^m)$-norm can be bound from above by a positive constant, independent of $\f$, from the above result, we conclude that $\uu_{\lambda}$ belongs to $B^{1+\sigma}_{p,p}(\R^d;\R^m)$, for $\sigma$ as above, and estimate \eqref{est_n} is satisfied.

Estimate \eqref{rullp} can be proved similarly, using estimate \eqref{teta1teta2} with $\theta_1=\theta$, $\theta_2=2+\sigma$ and $\eps=\frac{1}{2}$, when $\lambda=\lambda_p+1$, and then, writing $\uu_{\lambda}=R(\lambda_p+1,\bm{A}_p)[\f+(\lambda_p+1-\lambda)R(\lambda,\bm{A}_p)\f]$ and using the first part of the proof into account to infer that
$\f+(\lambda_p+1-\lambda)R(\lambda,\bm{A}_p)\f$ belongs to
$B^{\theta}_{p,p}(\R^d;\R^m)$ with norm which can be bounded from above by a positive constant, independent of $\f$, times the norm of $\f$ in $B^{\theta}_{p,p}(\R^d;\R^m)$. 

To prove the last part of the assertion, we fix $\uu\in D(\bm{A}_p)$ and set $\f=\lambda\uu-\bm{A}_p\uu$ for some $\lambda\in\rho(\bm{A}_p)$.
Clearly, $\f$ belongs to $L^p(\R^d;\R^m)$ and there exists a positive constant $c$, independent of $\uu$, such that
$\|\f\|_{L^p(\R^d;\R^m)}\le c\|\uu\|_{D(\bm{A}_p)}$. Moreover,
$\uu=R(\lambda,\bm{A}_p)\f$. Estimate \eqref{est_n} yields the assertion when Hypotheses \ref{hyp-derivata-2} are satisfied. On the other hand, under Hypotheses \ref{hyp-derivata-1}, using \eqref{sob_est}, with $k=0$ and $\ell=1$, and adapting the argument in the first part of the proof, it can be easily proved that 
$R(\lambda_p+1,\bm{A}_p)\f$ belongs to $W^{1,p}(\R^d;\R^m)$ and its norm can be estimated from above by a constant (independent of $\f$) times the norm of $\f$ in $L^p(\R^d;\R^m)$.
\end{proof}

\subsection{Regularity estimates in $L^p$: evolution systems}

Now we deal with mild solutions to Cauchy evolution systems. Fix $T>0$ and consider the problem
\begin{equation}\label{evo-sys}
\left\{
\begin{array}{ll}
D_t \uu(t,x)= {\bm \A}\uu(t,x)+\g(t,x), & (t,x)\in [0,T]\times \Rd,\\[1mm]
\uu(0, x)=\f(x), & x\in\R^d,
\end{array}
\right.
\end{equation}
where $\f:\Rd\to \Rm$ and $\g:[0,T]\times \Rd\to \Rm$ belong to suitable $L^p$-spaces. Mild solutions of the evolution system \eqref{evo-sys} are written as in \eqref{Sol_Mild_Parab} where $\{{\bm T}(t)\}_{t\ge 0}$ has to be understood as the semigroup 
in $L^p(\R^d;\R^m)$. 
We begin by studying the function $\bm w$ defined
by
\begin{equation}\label{u0}
{\bm w}(t,x)= \int_0^t ({\bm T}(s)\g(t-s,\cdot))(x)ds, \qquad\;\, t\in [0,T],\;\,x \in \Rd. 
\end{equation}
In the following proposition we establish Sobolev and Besov type regularity results with respect to the space variable for the function $\bm{w}$.

\begin{pro}
\label{prop-5.8}
Assume that there exists $M\geq 1$ and $\omega\in\R$ such that $\|S_{\nu}(t)\|_{\mathcal{L}(L^1(\Rd))}\le Me^{\omega t}$ for any $t\in (0,\infty)$. Assume that Hypotheses $\ref{hyp-derivata-1}$ are satisfied. For every $\g\in L^p((0,T);L^p(\Rd;\Rm))$ the function $\bm w$ defined in \eqref{u0} belongs to
$L^p((0, T);W^{1,p}(\Rd;\Rm))$ and there exists a positive constant $c$, depending on $p$, $T$ and independent of $\bm{g}$, such that 
\begin{equation*}
\|\bm w\|_{L^p((0, T);W^{1,p}(\Rd;\Rm))}\le c\|\g\|_{L^p((0,T);L^p(\Rd;\Rm))}.
\end{equation*}
If Hypotheses $\ref{hyp-derivata-1}$ are replaced by Hypotheses $\ref{hyp-derivata-2}$, then
 $\bm w$ belongs to
$L^p((0, T);B^{\theta}_{p,p}(\Rd;\Rm))$ for any $\theta \in [0,2)$ and there exists a positive constant $c$, depending on $p$, $T$ and independent of $\bm{g}$, such that 
\begin{equation}\label{521}
\|\bm w\|_{L^p((0, T);B^{\theta}_{p,p}(\Rd;\Rm))}\le c\|\g\|_{L^p((0,T);L^p(\Rd;\Rm))}.
\end{equation}
Finally, under Hypotheses $\ref{hyp-derivata-3}$, if $\g\in L^p((0,T);B^{\theta_1}_{p,p}(\Rd;\Rm))$ for some $\theta_1 \in (0,1)$ then $\bm w$ belongs to $L^p((0,T);B^{2+\theta_2}_{p,p}(\Rd;\Rm))$ for any $\theta_2\in [0,\theta_1)$ and
\begin{equation}\label{522}
\|\bm w\|_{L^p((0, T);B^{2+\theta_2}_{p,p}(\Rd;\Rm))}\le c\|\g\|_{L^p((0,T);B^{\theta_1}_{p,p}(\Rd;\Rm))}
\end{equation}
for a positive constant $c$ depending on $p$, $T$, $\theta_1$ and $\theta_2$ and independent of $\bm{g}$.
\end{pro}
\begin{proof}
Throughout the proof $K$ will denote a positive constant, independent of $\bm{g}$, that may change from line to line. We can limit ourselves to proving the claim for smooth functions. Indeed, arguing as in the proof of Theorem \ref{appr} we can extend the estimate for functions $\g$ that belong to $L^p((0,T);L^p(\Rd;\Rm))$. Thus, fix $\g\in L^p((0,T);C_c(\Rd;\Rm))$. Then, for every $t \in (0, T)$, thanks to estimate \eqref{sob_est} with $(k,\ell)=(0,1)$, $\varepsilon=1$ and the H\"older inequality, we deduce that
\begin{align*}
\int_0^T \|\bm w(t, \cdot)&\|_{W^{1, p}(\Rd;\Rm)}^p dt\\
\le &K\int_0^T\left(\int_0^t e^{\left (\tilde{H}_p+\frac{\omega}{p}+1\right )s} s^{-\frac{1}{2}}\|\g(t-s, \cdot)\|_{L^p(\Rd;\Rm)}ds\right)^p dt\\
\le &K\int_0^T\bigg (\int_0^t s^{-\frac{1}{2}}\|\g(t-s, \cdot)\|_{L^p(\Rd;\Rm)}^pds\bigg )\bigg (\int_0^t e^{\frac{p}{p-1}\left (\tilde{H}_p+\frac{\omega}{p}+1\right )s}s^{-\frac{1}{2}}ds\bigg )^{p-1}dt\\
\le & K\int_0^T\left(\int_0^ts^{-\frac{1}{2}}\|\g(t-s, \cdot)\|_{L^p(\Rd;\Rm)}^pds\right)dt\\
\leq & K\bigg (\int_0^T\|\g(r, \cdot)\|_{L^p(\Rd;\Rm)}^pdr\bigg )\bigg (\int_0^Ts^{-\frac{1}{2}}ds\bigg )\\
\le & K\|\g\|_{L^p((0,T);L^p(\Rd;\Rm))}^p.
\end{align*}
Similar arguments together with estimates \eqref{teta1teta2} allow us to prove \eqref{521} and \eqref{522}.
\end{proof}

\begin{cor}
Assume that there exists $M\geq 1$ and $\omega\in\R$ such that $\|S_{\nu}(t)\|_{\mathcal{L}(L^1(\Rd))}\le Me^{\omega t}$ for any $t\in (0,\infty)$. Assume Hypotheses $\ref{hyp-derivata-1}$ hold true, fix $p\in [2,\infty)$, $\f \in B^{\theta}_{p,p}(\Rd;\Rm)$, for some $\theta>1-\frac{2}{p}$, and $\g \in L^p((0,T);L^p(\Rd;\Rm))$. The mild solution $\bm{v}$ to the problem \eqref{evo-sys} (introduced in \eqref{Sol_Mild_Parab}) belongs to the space $L^p((0, T);W^{1,p}(\Rd;\Rm))$ and there exists a positive constant $c_1$, independent of $\f$ and $\g$, such that
\begin{equation} 
\|\bm{v}\|_{L^p((0, T);W^{1,p}(\Rd;\Rm))}\le c_1(\|\f\|_{B^{\theta}_{p,p}(\Rd;\Rm)}+\|\bm{g}\|_{L^p((0, T);L^p(\Rd;\Rm))}).
\label{stima-1}
\end{equation}
If Hypotheses $\ref{hyp-derivata-1}$ are replaced by Hypotheses $\ref{hyp-derivata-2}$ and $\f\in B^{\theta_1}_{p,p}(\R^d;\R^m)$ for some $\theta_1>\theta-\frac{2}{p}$ and some $\theta\in (0,2)$, then $\bm{v}$ belongs to $L^p((0,T);B^\theta_{p,p}(\Rd;\Rm))$ and there exists a positive constant $c_1$, independent of $\f$ and $\g$, such that
\begin{equation*}
\|\bm{v}\|_{L^p((0, T);B^\theta_{p,p}(\Rd;\Rm))}\le c_2(\|\f\|_{B^{\theta_1}_{p,p}(\Rd;\Rm)}+ \|\bm{g}\|_{L^p((0, T);L^p(\Rd;\Rm))}).
%\label{stima-2}
\end{equation*}
Finally, if Hypotheses $\ref{hyp-derivata-2}$ are replaced by Hypotheses $\ref{hyp-derivata-3}$, $\f$ belongs to $B^{\theta_1}_{p,p}(\Rd;\Rm)$ for some $\theta_1>2+\theta-\frac{2}{p}$ and some $\theta\in (0,1)$ and $\g$ belongs to $L^p((0,T);B^{\theta_2}_{p,p}(\Rd;\Rm))$ for some $\theta_2 \in (\theta,1)$ then $\bm{v}$ belongs to $L^p((0,T);B^{2+\theta}_{p,p}(\Rd;\Rm))$ and there exists a positive constant $c_3$, independent of $\f$ and $\g$, such that
\begin{equation}
\|\bm{v}\|_{L^p((0, T);B^{2+\theta}_{p,p}(\Rd;\Rm))}\le c_3(\|\f\|_{B^{\theta_1}_{p,p}(\Rd;\Rm)}+ \|\bm{g}\|_{L^p((0, T);B^{\theta_2}_{p,p}(\Rd;\Rm))}).
\label{stima-3}
\end{equation}
Moreover, $\bm{v}$ solves the Cauchy problem \eqref{evo-sys}.
\end{cor}

\begin{proof}
Estimates \eqref{stima-1}-\eqref{stima-3} follow from Theorem \ref{appr} and Proposition \ref{prop-5.8}. In particular, estimates \eqref{sob_est}, \eqref{teta1teta2} and the continuity of the map $t\mapsto\bm{T}(t)\f$ in $(0,\infty)$ with values in $B^{\theta}_{p,p}(\R^d;\R^m)$ (where $\theta\in (0,1)$ under Hypotheses \ref{hyp-derivata-1}, $\theta\in (1,2)$ under Hypotheses \ref{hyp-derivata-2} and $\theta\in (2,3)$ under Hypotheses \ref{hyp-derivata-3}, show that the function $\bm{T}(\cdot)\f$ belongs to $L^p((0,T);B^{\theta}_{p,p}(\R^d;\R^m))$ if $\theta$ is as in the statement of the theorem.

To complete the proof, we need to show that, if $\f\in B^{\theta_1}_{p,p}(\R^d;\R^m)$ and $
\bm{g}\in L^p((0,T);B^{\theta_2}_{p,p}(\R^d;\R^m))$, with $\theta_1$ and $\theta_2$ as in the statement of the theorem, then the function $\bm{v}$ solves the Cauchy problem \eqref{evo-sys}. For this purpose, we fix two sequences $(\f_n)_{n\in\N}$ and $(\bm{g}_n)_{n\in\N}$ of smooth and compactly supported functions, which converge to $\f$ in 
$B^{\theta_1}_{p,p}(\R^d;\R^m)$ and $\bm{g}$ in $L^p((0,T);B^{\theta_2}_{p,p}(\R^d;\R^m))$, respectively.
By Theorem \ref{thm-4.3}, for each $n\in\N$ the function $\bm{v}_n$, defined by
\begin{eqnarray*}
\bm{v}_n(t,x)=(\bm{T}(t)\f_n)(x)+\int_0^t(\bm{T}(t-s)\bm{g}_n(s,\cdot))(x)ds,\qquad\;\,t\in [0,T],\;\,x\in\R^d,
\end{eqnarray*}
belongs to $C^{1,2}([0,T]\times\R^d;\R^m)\cap L^p((0,T);B^{2+\theta}_{p,p}(\R^d;\R^m))$ and solves the Cauchy problem \eqref{evo-sys}, with $\bm{f}$ and $\bm{g}$ replaced with $\bm{f}_n$ and $\bm{g}_n$, respectively. Moreover, $\bm{v}_n$ converges to the function $\bm{v}$ in $L^p((0,T);B^{2+\theta}_{p,p}(\R^d;\R^m))$, thanks to estimate \eqref{stima-3}.
Since $D_t\bm{v}_n=\bm{\A}\bm{v}_n+\bm{g}_n$, it follows that the sequence $(D_t\bm{v}_n)_{n\in\N}$ converges in $L^p((0,T);L^p(B(0,r);\R^m))$ to the function $\bm{\A}\bm{v}+\bm{g}$ (for every $r\in (0,\infty)$). Recalling that the time derivative is a closed operator in $L^p((0,T);L^p(B(0,r);\R^m))$, with $W^{1,p}((0,T);L^p(B(0,r);\R^m))$ as domain, it follows that the function $\bm{v}$ admits time derivative, which belongs to $L^p((0,T);L^p(B(0,r);\R^m))$ for every $r\in (0,\infty)$ and $D_t\bm{v}=\bm{\A}\bm{v}+\bm{g}$. Since, clearly, $\bm{v}(0,\cdot)=\bm{f}$, we have proved that the function $\bm{v}$ solves the Cauchy problem \eqref{evo-sys}.
\end{proof}

\section{Examples}
\label{sect-6}
Here, we exhibit a class of operators that satisfies our assumptions and to which our results can be applied.
Let the coefficients of the operator $\bm \A$ be given by
\begin{align*}
Q(x)=(1+|x|^2)^k Q_0, \qquad\;\, B_i(x)=-x_i(1+|x|^2)^pI_m+(1+|x|^2)^r{B}_i^0,\qquad\;\, C(x)=-(1+|x|^2)^\gamma C_0
\end{align*}
for any $x \in \Rd$ and $i=1, \ldots,d$, where $I_m$ is the $m\times m$ identity matrix. We assume the basic conditions:
\begin{enumerate}[\rm (i)]
\item $Q_0, {B}_i^0$ $(i=1,\ldots,d)$ and $C_0 $ are constant matrices, $Q_0$ is symmetric, $Q_0, C_0$ are positive definite matrices and the matrices $B_i^0$ ($i=1,\ldots,d$) have null entries on the diagonal;
\item the exponents $k,p,r,\gamma$ are nonnegative, $k < p+1$ and  $\gamma >\max\{0,2r-k\}.$
\end{enumerate}
Under the previous assumptions Hypotheses \ref{base} are satisfied. In particular, Hypothesis \ref{base}(v) is satisfied with $\varphi(x)=1+|x|^2$ and any $\mu \in \R$.
Moreover, if there exists $\mu_1\in [0,2]$ such that
\begin{equation}\label{tizio-luca}
\left\{
\begin{array}{l}
\mu_1(2r-1)\le 2k;\\[1mm]
\max\{p,\gamma\}>\frac{1}{2}\max\{2k-1,(2r-1)(2-\mu_1)\},
\end{array}
\right.
\end{equation}
then Hypotheses \ref{hyp-derivata-1} are satisfied.
%Here,
%\begin{eqnarray*} 
%\Gamma_h(\mu_1,\ldots,\mu_h)=\max_{i=1,\ldots,h}\{2k-1,(2r-i)(2-\mu_i)\}
%\end{eqnarray*}
%for $h=1,2,3$. 
On the other hand, if  there exist $\mu_1\in [0,2]$ such that
\begin{equation}\label{tizio1-luca}
\left\{
\begin{array}{l}
\mu_1(2r-1) \le 2k;\\[1mm]
\max\{p,\gamma\}>\frac{1}{2}\max\{2k-1,(2r-1)(2-\mu_1),2r-2\},
\end{array}
\right.
\end{equation}
then Hypotheses \ref{hyp-derivata-2} are satisfied, with $\alpha_1=\alpha_2=\mu_2=\rho_2=\tau_1=\tau_2=1$.
Finally, \eqref{tizio1-luca} ensure also that Hypotheses \ref{hyp-derivata-3} are satisfied, taking $\alpha_1=\alpha_2=\alpha_3=\mu_2=\mu_3=\rho_2=\rho_3=\tau_1=\tau_2=\tau_3=1$.

Under the previous conditions, we can then apply the results in Sections \ref{sect-3} and \ref{sect-4}.

As far as the results in Section \ref{sect-5} are concerned, condition \eqref{cond-diver} is satisfied if
$\gamma>\max\{p,k-1\}$. In such a case, the conditions in (ii) together with \eqref{tizio-luca} become
\begin{equation*}
\left\{
\begin{array}{l}
k<p+1;\\
\gamma > \max\{2r-k,0\};\\
\mu_1(2r-1)\le 2k;\\[1mm]
\gamma>\frac{1}{2}\max\{2k-1,(2r-1)(2-\mu_1)\},
\end{array}
\right.
\end{equation*}
whereas conditions in (ii) together with \eqref{tizio1-luca} reduce to
\begin{equation*}
\left\{
\begin{array}{l}
k<p+1;\\
\gamma > \max\{2r-k,0\};\\
\mu_1(2r-1)\le 2k;\\[1mm]
\gamma>\frac{1}{2}\max\{2k-1,(2r-1)(2-\mu_1),2r-2\}.
\end{array}
\right.
\end{equation*}

% Finally, if  there exist $\alpha_i,\mu_i, \tau_i \in [0,2]$ for $i=1,2,3$ and $\rho_2,\rho_3 \in [0,2]$ such that
% \begin{equation}\label{tizio2}
% \left\{
% \begin{array}{l}
% \max\{\alpha_1(2k-1),\alpha_2(2k-2),\mu_1(2r-1)\} \le 2k;\\[1mm]
% \max\{\tau_1(2\gamma-1), \tau_2(2\gamma-2), \tau_3(2\gamma-3)\}<2\gamma;\\[1mm]
% \max\{p,\gamma\}>\frac{1}{2}\max\{\Gamma_3,\Phi_2, \max_{i=2,3}(2r-i)\mu_i, (2k-3)\alpha_3\},
% \end{array}
% \right.
% \end{equation}
% where
% \begin{eqnarray*}
% \Phi_2= \max_{i\in\{1,2\}}\{(2p-i)\rho_{i+1}, (2p-i)(2-\rho_{i+1}),(2\gamma-i)\tau_i\},
% \end{eqnarray*}
% then Hypotheses \ref{hyp-derivata-3} are satisfied. In particular, choosing $\alpha_i=\tau_i=\mu_i=0$ for $i=1,2,3$ and $\rho_2=\rho_3=1$, all the conditions in \eqref{tizio2} are satisfied if \eqref{condition_p} holds true.

\appendix

\section{Interior estimates}
In this section we will recall some interior regularity results that we use throughout the paper. We will provide the proof of Theorem \ref{teo-reg-int}, while for the proof of Theorem \ref{thm-A2} we refer to \cite[Theorem A.2]{AALT}.

\begin{thm}
\label{teo-reg-int}
Let $\Omega$ be a domain of $\R^d$ and let the entries of the coefficients of the operator $\bm{\A}$ belong to 
$\mathcal{C}^{k+\alpha}_{\rm loc}(\Omega)$ for some $k\in\N$ and $\alpha\in (0,1)$. Let $\uu\in\mathcal{C}^{1+\alpha/2,2+\alpha}_{\rm loc}([0,T]\times\Omega;\R^m)$ be such that $\uu(0,\cdot)\in \mathcal{C}^{2+k+\alpha}_{\rm loc}(\Omega;\R^m)$, $D_t\uu-\bm{\A}\uu \in \mathcal{C}^{\alpha/2,k+\alpha}_{\rm loc}([0,T]\times\Omega;\R^m)$. Then, $\uu\in \mathcal{C}^{1+\alpha/2,2+k+\alpha}_{\rm loc}([0,T] \times \Omega;\R^m)$, $D_t\uu \in \mathcal{C}^{\alpha,k+\alpha}_{\rm loc}([0,T]\times\Omega;\R^m)$ and $D_x^{\beta}D_t\uu=D_tD_x^{\beta}\uu$ in $(0,T)\times \Omega$ for any $|\beta|\le k$.
\end{thm}

\begin{proof}
The statement has been proved in \cite[Theorem A.1]{AAL-1} in the case when $\widehat B_i=0$ for every $i=1,\ldots,d$.
To prove it in the general case, denote by $\bm{\mathcal{B}}$ the operator defined on smooth functions $\bm{w}$ by $\bm{\mathcal B}\bm{w}=\sum_{i=1}^d\widehat B_iD_i\bm{w}$ and set $\bm{\A}_0=\bm{\A}-\bm{\mathcal B}$.
It is clear that $\bm{\mathcal B}$ maps $C^{0,h+\alpha}_{\rm loc}([0,T]\times\Omega;\R^m)$ into $C^{0,h-1+\alpha}_{\rm loc}([0,T]\times\Omega;\R^m)$ for every $h\in\N$. 
Since $\uu\in C^{1+\alpha/2,2+\alpha}_{\rm loc}([0,T]\times\Omega;\R^m)$, it follows that $D_t\uu-\bm{\A}_0\uu\in
C^{\alpha,1+\alpha}_{\rm loc}([0,T]\times\R^d;\R^m)$. Hence, \cite[Theorem A.1]{AAL-1} implies that $\uu\in C^{1+\alpha/2,3+\alpha}_{\rm loc}([0,T]\times\Omega)$. Iterating this argument, in a finite number of steps, we get the assertion. 
\end{proof}

\begin{thm}
\label{thm-A2}
Let $\uu\in C^{1+\alpha/2,2+\alpha}_{\rm loc}((0,T]\times\Rd;\R^m)$ solves the differential
equation $D_t\uu =\bm{\A}\uu +\bm{g}$ in $(0,T]\times\Rd$ for some $\bm{g}\in \mathcal{C}_b^{\alpha/2,\alpha}((0,T]\times\Rd;\R^m)$. For any $\tau\in (0,T)$ and any pair of bounded open sets $\Omega_1$ and $\Omega_2$ such that $\overline{\Omega_1}\subset\Omega_2$, there exists a positive constant $c$, depending on $\Omega_1$, $\Omega_2$, $\tau$ and $T$, but being independent of $\uu$, such that
\begin{equation*} 
\|\uu\|_{C_b^{1+\alpha/2,2+\alpha}((\tau,T)\times\Omega_1;\R^m)}\le c(\|\uu\|_{C_b((\tau/2,T)\times\Omega_2;\R^m)}+ \|\bm{g}\|_{\mathcal{C}_b^{\alpha/2,\alpha}((\tau/2,T)\times\Omega_2;\R^m)}).
\end{equation*}
\end{thm}

\section*{Declarations}

%\subsection*{Acknowledgments} The authors would like to thank Professor A. Lunardi and  Professor E. Priola for many useful discussions and comments.

\subsection*{Fundings} The authors are members of GNAMPA (Gruppo Nazionale per l’Analisi Matematica, la Probabilit\`a
e le loro Applicazioni) of the Italian Istituto Nazionale di Alta Matematica (INdAM).

\subsection*{Research Data Policy and Data Availability Statements} Data sharing not applicable to this article as no datasets were generated or analysed during the current study.

\end{document}